%% file: main.tex
\documentclass[12pt]{article}
\input{settings.tex}

\title{4-Rank Distribution of Picard Groups of Hyperelliptic Curves via $C$-Symmetric Matrices}
\author{Elia Gorokhovsky, August Liu}
\date{\today}

\begin{document}
\maketitle
\begin{abstract}
We determine the large-genus limiting distribution of the 4-rank of the Picard group of hyperelliptic curves over a fixed finite field $\F_q$ of odd characteristic. This is a function field analogue of a result of Fouvry and Kl\"uners. Our computation agrees with (the Picard group analogue of) the Cohen--Lenstra--Gerth heuristics in the case $q \equiv 3\pmod{4}$, i.e., in the absence of fourth roots of unity in the base field. When fourth roots of unity are present, the result is of the same form as the conjectured distribution for class groups of quadratic extensions of number fields containing certain roots of unity. The limiting distribution does not change when imposing finitely many conditions on the ramification behavior of the curves. In the process, we determine the rank distribution of a certain class of random matrix ensembles over finite fields determined by symmetry conditions.
\end{abstract}

\newcommand{\details}[1]{}
\newcommand{\intuition}[1]{}

\section{Introduction}

In 1984, Cohen and Lenstra \cite{cohenlenstra1984} gave conjectural distributions of the $p$-parts of class groups of imaginary and real quadratic fields for odd primes $p$, ordered by absolute discriminant. In the case of imaginary quadratic fields, they conjectured that a finite abelian $p$-group $A$ appears as the $p$-part of a class group of a ``random'' such field with probability inversely proportional to the order of its automorphism group, $|\Aut(A)|$. The conjectured distribution for imaginary (respectively, real) quadratic fields also arises as the limiting distribution of cokernels of Haar-random $n\times n$ (respectively, $n \times (n + 1)$) matrices over $\Z_p$ as $n \to \infty$ (see \cite{wood_probability_2023} for an overview).

Cohen and Lenstra did not formulate their heuristics for the case $p = 2$ because the 2-parts of class groups have ``deterministic'' behavior. Namely, Gauss genus theory \cite[Articles 229--287]{gauss} shows that the size of the 2-torsion subgroup $\Cl(K)[2]$ of a quadratic number field $K$ is determined by the number of primes ramified in the extension $K/\Q$ and whether $K$ has any fundamental unit of norm 1. In particular, when $K$ ranges over imaginary or real quadratic number fields, the average size of $\Cl(K)[2]$ is infinite.

Gerth \cite{gerthconj} extended the Cohen--Lenstra heuristics to the 2-part of the class group by conjecturing that all such deterministic behavior in the class group comes from genus theory. That is, as $K$ ranges over quadratic number fields, the subgroup $2\Cl(K)[2^\infty] \subseteq \Cl(K)[2^\infty]$ of the 2-part of $\Cl(K)$ should follow a Cohen--Lenstra-type distribution.

In this paper, we give evidence toward a function field analogue of Gerth's conjecture by determining the limiting distribution of the subgroup $2\Pic^0(X)(\F_q)[4] \subset \Pic^0(X)(\F_q)[2]$ of the Picard group of $X$ as $X$ ranges over hyperelliptic curves of large genus over the finite field $\F_q$ satisfying finitely many local conditions. The following two theorems are immediate consequences of the more general Theorem~\ref{thm:mainthm}, which interpolates between the two. 

We define $\mu_{CL}(r)$ (respectively, $\mu_S(r)$) to be the limiting probability as $n \to \infty$ that a uniformly random (respectively, uniformly random symmetric) $n\times n$ matrix over $\F_2$ has corank $r$. Both of these probabilities are not hard to compute explicitly (see, e.g., \cite{fulmangoldstein2015ranks} for a comprehensive overview with effective error bounds).

\begin{theorem}\label{thm:intro-thm-local}
Fix an odd prime power $q$. Let $S, S', S'' \subset \PP^1_{\F_q}$ be disjoint finite sets of closed points. For integers $g \geq 1$, let $X_g \to \PP^1_{\F_q}$ be chosen uniformly at random among hyperelliptic curves of genus $g$ which are ramified at a set of points containing $S$, split at $S'$, and inert at $S''$. 

Let $\mu \coloneqq \mu_{CL}$ if $q \equiv 3\pmod{4}$ and $\mu \coloneqq \mu_S$ if $q \equiv 1\pmod{4}$. Then for integers $r \geq 0$ we have \[
\lim_{g\to\infty} \PP[\dim_{\F_2} 2\Pic^0(X_g)(\F_q)[4] = r] = \mu(r).
\]
\end{theorem}

\begin{theorem}\label{thm:intro-thm-monic}
Fix an odd prime power $q$ and an affine chart $\Spec \F_q[x] = \A^1_{\F_q} \subset \PP^1_{\F_q}$. For integers $d \geq 1$ let $X_d \to \PP^1_{\F_q}$ be the hyperelliptic curve defined by $y^2 = f(x)$, where $f(x) \in \F_q[x]$ is chosen uniformly at random among monic squarefree polynomials over $\F_q$ of degree $d$.

Let $\mu \coloneqq \mu_{CL}$ if $q \equiv 3\pmod{4}$ and $\mu \coloneqq \mu_S$ if $q \equiv 1\pmod{4}$. Then for integers $r \geq 0$ we have \[
\lim_{d\to\infty} \PP[\dim_{\F_2} 2\Pic^0(X_d)(\F_q)[4] = r] = \mu(r).
\]
\end{theorem}

The key detail to notice about Theorem~\ref{thm:intro-thm-monic} is that when $d$ is odd, we are constraining the behavior of $X_d \to \PP^1_{\F_q}$ at a ramification point (the point at $\infty$). Theorem~\ref{thm:mainthm} allows us to fix a finite number of ramification points and splitting behavior at a finite number of unramified points, as well as behavior at one ramification point of minimal degree.

The story of the 2-part of the class group is much more complete on the number field side. In general, the distribution of the odd part of the class group is still completely mysterious (aside from the 3-torsion, whose average size can be computed with methods that do not generalize easily; see \cite{davenport_density_1997, bhargava_davenport-heilbronn_2013}). However, the deterministic ``base layer'' given by Gauss genus theory enables working directly with $2^\infty$-torsion elements in the class group. 

In 1934, R\'edei \cite{Redei1934} showed that the 4-rank of the class group (that is, the dimension of $2\Cl(K)[4]$ as an $\F_2$-vector space) of a quadratic number field can be expressed as the nullity of a matrix over $\F_2$ whose entries are Legendre symbols evaluated at rational primes where $K$ is ramified. Based on this construction, in 1984 Gerth \cite{Gerth1984-hc} determined the distribution of the 4-rank of $\Cl(K)$ as $K$ ranges over quadratic number fields with a fixed number $r$ of ramified primes, in the large-$r$ limit. 

The \textit{R\'edei matrix} $M_K$ of a quadratic number field $K$ satisfies certain symmetry conditions due to quadratic reciprocity. In particular, we have $M_K - M_K^\intercal = C$, where $C$ is a matrix that encodes the residues modulo 4 of rational primes that ramify in $K$. We call matrices with such a symmetry property \textit{$C$-symmetric}. In Theorem~\ref{thm:c-symmetric-ranks} and Corollary~\ref{cor:c-symmetric-ranks} we determine the $n\to\infty$ limiting corank distribution of $C$-symmetric $n\times n$ random matrices when $C$ has large enough rank (at least on the order of $\log n$). In particular, this specializes to give an alternate effective proof for the result of Gerth \cite{Gerth1984-hc} and Koymans and Pagano \cite{Koymans2020EffectiveCO} that random matrices satisfying the same symmetry conditions as the R\'edei matrix have the same limiting corank distribution as random matrices with no symmetry constraints.

After Gerth, Fouvry and Kl\"uners \cite{fouvrykluners2010} determined the full distribution of the 4-rank of $\Cl(K)$ as $K$ ranges over quadratic number fields ordered by absolute discriminant using moment methods and hence techniques similar to the work of Heath-Brown \cite{heath-brown_size_1993, heath-brown_size_1994} on distribution of 2-Selmer ranks of congruent number elliptic curves. Their strategy was to express powers of the size of the kernel of the R\'edei matrix directly in terms of Legendre symbols of ramified primes, then use the averages of these powers to deduce the distribution of the 4-rank. Finally, Smith proved Gerth's conjecture for the 8-rank of the class group in 2016 \cite{smith_governing_2016} and for the full 2-part of the class group in 2022 \cite{smith_distribution_2023, smith_distribution_2023b}. Smith's results apply to arbitrary number fields not containing fourth roots of unity.

There has been some work toward similar results with ramification conditions imposed at finitely many primes, as in our Theorem~\ref{thm:intro-thm-local}. For instance, Altug, Shankar, Varma, and Wilson \cite[Theorem 4]{altug_number_2021} determined the average number of order 4 elements in class groups of quadratic number fields with some local conditions, ordered by discriminant. This result is similar to (but not quite the same as) determining the average size of $2\Cl(K)[4]$ rather than the full distribution. The method of proof is also very distinct from the approaches of Gerth, Fouvry--Kl\"uners, and this paper; it proceeds by determining asymptotic counts of quartic number fields with prescribed ramification and Galois group $D_4$.

Compared to the work of Fouvry and Kl\"uners, in the function field world we have the advantage of a strong effective Chebotarev density theorem and a good probabilistic theory of random subsets of $\PP^1$ that allow us to prove a type of effective equidistribution result about the R\'edei matrix associated to a hyperelliptic curve rather than computing the 4-rank of the Picard group directly from the data of the curve's ramification points. This enables us to study the 4-rank of the Picard group as the corank of a random matrix in Section~\ref{sect:random-matrices}. However, since we are working with Picard groups rather than class groups, there is some extra work that goes into defining such a R\'edei matrix in the first place. We do this in Section~\ref{sect:redei}. 

On the function field side, one may study the distribution of class groups $\Cl(K)$ of function fields $K$ or Picard groups $\Pic^0(X)(\F_q)$ of the associated projective curves $X$. These two group invariants are related by excision: if $X \to \PP^1_{\F_q}$ is a hyperelliptic curve over $\F_q$ with corresponding function field extension $K/\F_q(t)$, then the sequence \[
0 \longrightarrow R \longrightarrow \Pic^0(X)(\F_q) \longrightarrow \Cl(K) \longrightarrow \Z/\delta\Z \longrightarrow 0
\]
is exact, where $R$ is the group of degree 0 divisor classes with a representative supported at the points above $\infty \in \PP^1_{\F_q}$ and $\delta$ is the greatest common divisor of the degrees of the points of $X$ above $\infty$. In other words, if $\infty$ ramifies in $X \to \PP^1_{\F_q}$, then the class group agrees with the Picard group; if $\infty$ is split, then the class group is the quotient of the Picard group by the cyclic subgroup generated by the difference of the points above $\infty$; and if $\infty$ is inert, then the class group contains the Picard group as an index-2 subgroup. As a result, in the case where $\infty$ is split, the question of the distribution of the Picard group \textit{together with} the difference of the points above $\infty$ can be considered a refinement of the class group distribution question over quadratic function fields (see, e.g. \cite{wood_cohen-lenstra_2018}).

There is work by Bae and Jung \cite{bae_4-rank_2012} which follows the strategy of Fouvry and Kl\"uners to get 4-rank distributions for the narrow class groups and ideal class groups of quadratic function fields over $\F_q(t)$ when $q \equiv 3\pmod{4}$. Since we determine the 4-rank of Picard groups rather than class groups, neither work implies the other. However, our method allows us to access the case of $q \equiv 1\pmod{4}$ as well as imposing finitely many local conditions. It is likely that a similar strategy would yield distributions for ideal class groups as well.

Over function fields, much more is known about the distribution of the odd part of the class (or Picard) group. In 1989, Friedman and Washington \cite{friedmanDistributionDivisorClass1989} observed that for odd $p$, the $p$-part of the Picard group of a random curve over a finite field $\F_q$ can be related to a random (symplectic) matrix, and in 2006 and 2008 Achter \cite{achter2006classgroups, achter2008cohenlenstra} used this interpretation to give asymptotic formulas for the distribution of the Picard group in families of curves of fixed genus $g$ as $q \to \infty$. In the process, Achter observed that one gets different distributions for the $p$-torsion of the Picard group depending on the residue class of $q$ modulo $p$. In particular, if $q \equiv 1 \pmod{p}$ (i.e., if there are $p$th roots of unity in the base field $\F_q(t)$), the correct distribution for the analogue of the imaginary quadratic case matches the limiting cokernel distribution of $n\times n$ \textit{symmetric} matrices as $n \to \infty$. This observation suggests that the Cohen--Lenstra heuristics need to be modified when generalizing to the case of an arbitrary number field potentially containing roots of unity (see, e.g. \cite{sawinConjecturesDistributionsClass2024}). We observe the same phenomenon in our results, which suggests a heuristic for the distributions of 2-parts of class groups of quadratic number fields containing roots of unity. 

Later work by Ellenberg, Venkatesh, and Westerland \cite{evw16} and Landesman-Levy \cite{landesman2024cohenlenstramomentsfunctionfields} used powerful techniques from algebraic topology to prove a version of the Cohen--Lenstra heuristics over function fields. These methods can show that each limiting \textit{moment} of the class group distribution of random high-genus hyperelliptic curves matches the corresponding moment of the Cohen--Lenstra distribution when $q$ is large enough (depending on the chosen moment). However, they cannot prove that the $g \to \infty$ limit of the class group distribution for genus-$g$ hyperelliptic curves over $\F_q$ is the Cohen--Lenstra distribution for any fixed $q$.

\subsection{Outline of the argument}

There are three major steps in the proof:
\begin{enumerate}
    \item First, in Section~\ref{sect:redei} we express the 4-rank of the Picard group of a given hyperelliptic curve $X$ over $\F_q$ as the nullity of a matrix analogous to the R\'edei matrices used to understand 4-ranks of class groups of number fields. In Subsection~\ref{sect:pairing} we define a pairing (valued in $\F_2$) between rational 2-torsion divisor classes of $X$ and $\overline{\F_q}$-isomorphism classes of \'etale 2-covers of $X$. The left kernel of this pairing has the same dimension as $2\Pic^0(X)(\F_q)[4]$. In Subsection~\ref{sect:basis} we give a two-to-one parametrization of $\Pic^0(X)(\F_q)[2]$ by rational degree-0 divisors supported at the ramification points of $X$, modulo 2. We lift our pairing along this parametrization to obtain the matrix we wish to study.

    In Subsection~\ref{sect:param-curves}, we parametrize hyperelliptic curves by sets of distinct points in $\PP^1$ with some extra data, and we describe how to compute the pairing from these data. The goal is to study the associated matrix as a random matrix after picking points in $\PP^1$ at random. In Subsection~\ref{sect:symmetry}, we show that the pairing matrix is $C$-symmetric for a particular $C$, in addition to having each row and column sum to zero.

    \item Next, in Section~\ref{sect:random-matrices} we study the nullity distribution of large random $C$-symmetric matrices over a general finite field $\F_\ell$. In the case of symmetric matrices, we can rely on known results to determine exactly the nullity distribution (Subsection~\ref{sect:symmetric-random-matrices}). 
    
    In Subsection~\ref{sect:c-symmetric-random-matrices}, we determine the limiting nullity distribution of random $C$-symmetric matrices when $C$ has large enough rank. The argument proceeds by comparing $C$-symmetric matrices to uniformly random matrices with no symmetry conditions. 

    If $M$ is a random $n\times n$ $C$-symmetric matrix and $B \in \GL_n(\F_\ell)$ is a uniformly random invertible matrix, then $B^\intercal M B$ has the same nullity distribution, but if $C$ has high enough rank, the symmetry constraints imposed by $C$ are randomly shuffled around and end up averaging out. To formalize this, we compute the probability that $B^\intercal M B$ vanishes on an arbitrary subspace $V \subseteq \F_\ell^n$ and see that it agrees in the limit with the probability that a uniformly random matrix with no symmetry constraints vanishes on $V$.

    \item Finally, in Section~\ref{sect:equidistribution}, we show that if $X$ is a uniformly random hyperelliptic curve of fixed genus, then the associated R\'edei matrix behaves like a random matrix with appropriate symmetry constraints. This part of the proof is of similar flavor to a result of Park \cite{park2024primeselmerrankscyclic} on ranks of prime Selmer groups in cyclic prime twist families of elliptic curves.
    
    In Subsections~\ref{sect:chebotarev}--\ref{sect:subfamilies} we show that the R\'edei matrix nearly equidistributes when $X$ is chosen uniformly within subfamilies of hyperelliptic curves where the degree of each of the branch points is prescribed, but each branch point itself is random among all points of that degree. To do this, we first use an effective function field version of the Chebotarev density theorem to show that, after conditioning on a small corner of the matrix, the rest of the matrix equidistributes (Subsection~\ref{sect:chebotarev}). We need to condition on a small corner because the Chebotarev density theorem does not provide good enough error bounds when ranging over low-degree points.

    In Subsection~\ref{sect:averaging}, we remove the conditioning at the cost of not being able to show that the entire matrix equidistributes. Instead, we introduce a rank-preserving group action on the set of $C$-symmetric matrices and show that, after acting on the R\'edei matrix randomly, the resulting matrix equidistributes. We can conclude that the nullity of the R\'edei matrix distributes as the nullity of a random $C$-symmetric matrix, even if we cannot show the R\'edei matrix itself behaves randomly. To get good bounds, we need to assume some properties of the prescribed degrees.

    In Subsection~\ref{sect:branch-degrees}, we use the theory of random logarithmic combinatorial structures developed by Arratia, Barbour, and Tavar\'e \cite{arratiabarbourtavare} to show that these assumed properties hold in almost all subfamilies, thereby obtaining an equidistribution result when $X$ ranges over all hyperelliptic curves of fixed genus. We use this in Subsection~\ref{sect:main-results} to prove our main results. 
    
\end{enumerate}

\subsection{Notation}
\begin{itemize}
    \item For $r$ an integer and $\ell$ a prime power we write \[
    \mu_{S, \ell}(r) \coloneqq \frac{|\wedge^2 \F_\ell^r|}{|\GL_r(\F_\ell)|}\prod_{k=0}^\infty (1 - \ell^{-2k - 1})
    \]
    for the limiting corank distribution of large uniformly random symmetric matrices over $\F_\ell$ and \[
    \mu_{CL, \ell}(r) \coloneqq \frac{1}{|\GL_r(\F_\ell)|}\prod_{i=r+1}^\infty (1 - \ell^{-i})
    \]
    for the limiting corank distribution of large uniformly random matrices over $\F_\ell$.
    \item For $p \in \Spec \Z$ a prime or $p \in X$ a closed point on a curve, we denote by $v_p(\cdot)$ the associated valuation on the corresponding function field. For example, if $p$ is a prime, $v_p(\cdot)$ denotes the $p$-adic valuation. If $p \in X$ is a closed point on a curve and $D$ is a divisor on $X$, then $v_p(D)$ denotes the coefficient of $p$ in $D$.
    \item For $q$ a prime power, we denote by $\binom{n}{k}_q$ the $q$-binomial symbol which counts the number of subspaces of $\F_q^n$ of dimension $k$: \[
    \binom{n}{k}_q \coloneqq \prod_{i=0}^{k-1} \frac{q^n - q^i}{q^k - q^i} = q^{kn}\frac{\prod_{i=0}^{k-1} (1 - q^{i-n})}{\#\GL_k(\F_q)}.
    \]
    \item If $\ip{\cdot}{\cdot}\colon \F_2^n \times \F_2^n$ is a pairing, we sometimes identify it with the associated $n\times n$ matrix $M$ with $\ip{v}{w} = v^\intercal M w$.
    \item We denote by $(\cdot)_+$ the isomorphism between the multiplicative group $\{\pm 1\}$ and the finite field $\F_2$. If $k$ is a finite field of odd order, we also denote by $(\cdot)_+$ the isomorphism $k^\times/(k^\times)^2 \cong \F_2$.
    \item When we use big-$O$ notation, we always express dependence in the implicit constant by subscripts. For example, if $f = O_\alpha(g)$, then $|f| \leq C|g|$ for a constant depending only on $\alpha$. Note that if $f$ and $g$ are functions on the natural numbers, we can equivalently ask for $|f(n)| \leq C|g(n)|$ for all but finitely many $n$.
    \item We use $\PP$ for probability, $\E$ for expectation, and $\Var$ for variance.
    \item For a positive integer $n$ we denote by $[n]$ the set $\{1, \dots, n\}$.
    \item We sometimes write $\PP^1$ for $\PP^1_k$ when the base field $k$ is implicitly clear.
    \item Curves are regular and proper.
    \item For $X$ a scheme over a field $k$, we write $X \times_k \overline{k}$ for $X \times_{\Spec(k)} \Spec(\overline{k})$.
    \item If $X$ is a scheme over $k$ and $p \in X$ a closed point, we write $\deg(p)$ for the dimension over $k$ of the residue field at $p$.
    \item We write $\pi_1^\etale$ for the \'etale fundamental group. We will omit the choice of base point with the understanding that this group is only defined up to inner automorphisms.
    \item If $X$ is a scheme over a finite field $k$ and $p$ is a closed point, we write $\Frob_p$ for the associated (geometric) Frobenius. More precisely, if $k(p)$ is the residue field of $X$ at $p$ with Galois group $\Gal(\overline{k(p)}/k(p))$ generated by $\varphi\colon x \mapsto x^{\#k(p)}$, then $\Frob_p$ is the (conjugacy class of) the image of $\varphi^{-1}$ under the map $\Gal(\overline{k}(p)/k(p)) \cong \pi_1^\etale(\Spec(k(p))) \to \pi_1^\etale(X)$.
    \item If $X$ is a smooth, proper curve over a perfect field $k$, we write $\Pic^0(X)(\overline{k})$ for the class group of degree-0 Cartier divisors of $X \times_k \overline{k}$, and we write $\Pic^0(X)(k)$ for the group of degree-0 Galois-invariant Cartier divisor classes of $X \times_k \overline{k}$. When $\Br(k) = 0$ (e.g., if $k$ is a finite field or algebraically closed), a standard argument using the Hochschild-Serre spectral sequence shows this is naturally isomorphic to the group of degree-0 Cartier divisors of $X$ up to linear equivalence. 
    \details{
    We have the Hochschild-Serre spectral sequence \[
    E_2^{r, s} = H^r(\Gal(\overline{k}/k), H_\etale^s(X \times_k \overline{k}, \G_m)) \implies H_\etale^{r + s}(X, \G_m)
    \]
    We want to compute $\Pic(X)(k) \cong H^1_\etale(X, \G_m)$. This has a filtration \[
    0 \longrightarrow E_\infty^{1, 0} \longrightarrow H^1_\etale(X, \G_m) \longrightarrow E_\infty^{0, 1} \longrightarrow 0
    \]
    In fact, $E^{1, 0}$ stabilizes to \[
    E_\infty^{1, 0} = E_2^{1, 0} = H^1(\Gal(\overline{k}/k), H_\etale^0(X \times_k \overline{k}, \G_m)).
    \]
    Since $X \times_k \overline{k}$ is a projective $\overline{k}$-scheme, $H_\etale^0(X \times_k \overline{k}, \G_m) \cong H_\etale^0(X \times_k \overline{k}, \sheafO_{X \times_k \overline{k}})^\times \cong \overline{k}^\times$. By Hilbert's Theorem 90, $H^1(\Gal(\overline{k}/k), \overline{k}^\times) = 0$, so $H_\etale^1(X, \G_m) \cong E_\infty^{0, 1}$.

    Note that $E^{0, 1}$ stabilizes after 3 pages to \begin{align*}
    E_\infty^{0, 1} = E_3^{0, 1} &= \ker(E_2^{0, 1} \to E_2^{2, 0}) \\
    &= \ker(H^0(\Gal(\overline{k}/k), H_\etale^1(X \times_k \overline{k}, \G_m)) \to H^2(\Gal(\overline{k}/k), H^0(C_{\overline{k}}, \G_m))) \\
    &\cong \ker(\Pic(X)(\overline{k})^{\Gal(\overline{k}/k)} \to H^2(\Gal(\overline{k}/k), \overline{k}^\times)) \\
    &= \Pic(X)(\overline{k})^{\Gal(\overline{k}/k)}
    \end{align*}
    as we wanted. 
    }
    We also write $\Div(X)(k)$ (respectively, $\Div^0(X)(k)$, $\Prin(X)(k)$) for the group of Cartier divisors (respectively, degree 0 and principal Cartier divisors) of $X$.
\end{itemize}

\section{Rédei Matrix for Function Fields}\label{sect:redei}
In \cite{Redei1934}, R\'edei showed that the $4$-class rank of a quadratic number field is determined by a matrix, known as the Rédei matrix, whose entries are Legendre symbols considered additively in $\F_2$. In this section, we do the same for function fields, namely constructing a matrix in $\F_2$ whose nullity relates to the $4$-rank of $\Pic^0(X)(\F_q)$ for a hyperelliptic curve $X$ over $\F_q$ where $q$ is an odd prime power.

\subsection{The Pairing $\ip{\cdot}{\cdot}_X$}\label{sect:pairing}

In this subsection, we define a pairing between the 2-torsion in the Picard group of a curve and the group of quadratic twist families of \'etale 2-covers of the same curve. In the following sections, we will explicitly understand the 2-torsion elements in the Picard group and give a basis for the 2-torsion as an $\F_2$-vector space. Using this basis, the matrix associated to this pairing will be our analogue of the R\'edei matrix.

Let $X$ be a regular, proper geometrically connected curve defined over a finite field $k$ of odd characteristic. We say a \textit{cover} of $X$ is a finite flat surjective map of regular, proper one-dimensional schemes. The cover need not be connected.

Let $\pi_1^\twist(X)$ be the image of $\pi_1^\etale(X \times_k \overline{k})^\ab$ in $\pi_1^\etale(X)^\ab$. For any finite abelian group $A$, the group $\Hom(\pi_1^\twist(X), A)$ classifies Galois covers of $X$ with Galois group a subgroup of $A$ up to $\overline{k}$-isomorphism.

Indeed, by abelianizing the homotopy exact sequence we get a split exact sequence \[
0 \longrightarrow \pi_1^\twist(X) \longrightarrow \pi_1^\etale(X)^\ab \longrightarrow \Gal(\overline{k}/k) \longrightarrow 0
\]
and after taking $\Hom(\cdot, A)$, one sees that $\Hom(\pi_1^\twist(X), A)$ is a quotient of $\Hom(\pi_1^\etale(X)^\ab, A)$ which is isomorphic to the image of the map $\Hom(\pi_1^\etale(X)^\ab, A) \to \Hom(\pi_1^\etale(X \times_k \overline{k})^\ab, A)$ taking an \'etale cover to its $\overline{k}$-isomorphism class.

\details{
We have an exact sequence \[
1 \longrightarrow \pi_1^\etale(X \times_k \overline{k}) \longrightarrow \pi_1^\etale(X) \longrightarrow \Gal(\overline{k}/k) \longrightarrow 1
\]
induced by the maps of schemes $X \times_k \overline{k} \to X \to \Spec(k)$. The inclusion of any rational point $\Spec(k) \to X \to \Spec(k)$ induces a splitting of this sequence. After abelianizing, we get the exact sequence of abelian groups \[
\pi_1^\etale(X \times_k \overline{k})^\ab \longrightarrow \pi_1^\etale(X)^\ab \longrightarrow \Gal(\overline{k}/k) \longrightarrow 0
\]
which can be read \[
0 \longrightarrow \pi_1^\twist(X) \longrightarrow \pi_1^\etale(X)^\ab \longrightarrow \Gal(\overline{k}/k) \longrightarrow 0.
\]
The splitting $\Gal(\overline{k}/k) \to \pi_1^\etale(X)$ of the original sequence induces a splitting of the abelianized sequence. Since this short exact sequence is split, it remains exact after taking $\Hom(\cdot, A)$: \[
0 \longrightarrow \Hom(\Gal(\overline{k}/k), A) \longrightarrow \Hom(\pi_1^\etale(X)^\ab, A) \longrightarrow \Hom(\pi_1^\twist(X), A) \longrightarrow 0
\]
The middle group parametrizes Galois covers of $X$ with Galois group $A$. 

The second map also appears in the sequence \[
\Hom(\pi_1^\etale(X)^\ab, A) \longrightarrow \Hom(\pi_1^\twist(X), A) \longrightarrow \Hom(\pi_1^\etale(X \times_k \overline{k})^\ab, A)
\]
coming from $\pi_1^\etale(X \times_k \overline{k})^\ab \to \pi_1^\twist(X) \to \pi_1^\etale(X)^\ab$. The first map is surjective by the discussion above, and the second map is injective by left exactness of $\Hom(\cdot, A)$. Thus, this sequence expresses $\Hom(\pi_1^\twist(X), A)$ as precisely the image of the map $\Hom(\pi_1^\etale(X)^\ab, A) \to \Hom(\pi_1^\etale(X \times_k \overline{k})^\ab, A)$ which sends a Galois $A$-cover of $X$ to the corresponding Galois $A$-cover of $X \times_k \overline{k}$, i.e., its $\overline{k}$-isomorphism class. In other words, as a subgroup of $\Hom(\pi_1^\etale(X \times_k \overline{k})^\ab, A)$, the group $\Hom(\pi_1^\twist(X), A)$ consists of those $\overline{k}$-isomorphism classes of Galois $A$-covers of $X$ which come from actual Galois $A$-covers of $X$ defined over $k$.
}

Class field theory gives an isomorphism from the profinite completion $\widehat{\Pic}(X)(k)$ of $\Pic(X)(k)$ to the abelianization of the full \'etale fundamental group $\pi_1^\etale(X)^\ab$ sending a divisor $[p]$ represented by $p \in X$ to $\Frob_p$. There is a map $\pi_1^\etale(X)^\ab \to \pi_1^\etale(\Spec(k)) \cong \Gal(\overline{k}/k) \cong \widehat{\Z}$ coming from the structure map $X \to \Spec(k)$ sending the Frobenius of a degree-$d$ point to $d \in \widehat{\Z}$. Thus, under the identification $\widehat{\Pic}(X)(k) \cong \pi_1^\etale(X)^\ab$, the map $\pi_1^\etale(X)^\ab \to \widehat{\Z}$ corresponds to the degree map $\widehat{\Pic}(X)(k) \to \widehat{\Z}$. The kernel of $\pi_1^\etale(X)^\ab \to \widehat{\Z}$ is precisely $\pi_1^\twist(X)$ \details{from the exact sequence\[
1 \longrightarrow \pi_1^\etale(X \times_k \overline{k}) \longrightarrow \pi_1^\etale(X) \longrightarrow \Gal(\overline{k}/k) \longrightarrow 1
\]
and right-exactness of abelianization
}, whereas the kernel of the degree map $\widehat{\Pic}(X)(k) \to \widehat{\Z}$ is $\Pic^0(X)(k)$. This gives us an identification $\Pic^0(X)(k) \cong \pi_1^\twist(X)$. Thus we have a perfect pairing \[
\Pic^0(X)(k) \times \Hom(\pi_1^\twist(X), \Q/\Z) \to \Q/\Z
\]
given by the evaluation map. Restricting this pairing to $2$-torsion gives us a pairing \[
\ip{\cdot}{\cdot}_X\colon \Pic^0(X)(k)[2] \times \Hom(\pi_1^\twist(X), \F_2) \to \F_2
\]
whose left kernel is $\Pic^0(X)(k)[2] \cap 2\Pic^0(X)(k)$, an elementary abelian $2$-group whose rank coincides with the $4$-rank of $\Pic^0(X)(k)$. If $X$ is a hyperelliptic curve with a map $\pi\colon X \to \PP^1_k$, we will sometimes write $\ip{\cdot}{\cdot}_\pi$ instead of $\ip{\cdot}{\cdot}_X$.

Explicitly, suppose we are given a $2$-torsion divisor class represented by $D = \sum_{p \in X} a_p[p]$ and a character $\chi\colon \pi_1^\twist(X) \to \F_2$. Let $\widetilde{\chi}$ be a lift of $\chi$ to $\Hom(\pi_1^\etale(X), \F_2)$. Then the pairing evaluates as \begin{equation}\label{eq:pairing-formula}
    \ip{D}{\chi}_X=\sum_{p\in X}a_p\widetilde{\chi}(\Frob_p).
\end{equation}

\subsection{Divisors, branched covers, and second-order classes}\label{sect:second-order}

In this subsection, let $X$ be a geometrically integral, regular, proper curve over a finite field $k$ of odd characteristic. Let $K \coloneqq k(X)$ be the function field of $X$. 

In this subsection we spell out some well-known facts about 2-covers that will help us make the pairing $\ip{\cdot}{\cdot}_X$ concrete in the following three subsections.

The first is a lemma which is a geometric analogue of a special case of Kummer theory. We will apply this in two ways: first, to classify hyperelliptic curves in terms of their branch points and one additional datum (with $X = \PP^1$); and second, to relate \'etale 2-covers of a hyperelliptic curve to 2-torsion divisor classes (with $X$ a hyperelliptic curve).

Note that we have a natural group structure on the set of isomorphism classes of branched 2-covers of $X$ because they are parametrized by the group $\Hom(\Gal(K^{\text{sep}}/K), \F_2)$ of $\F_2$-valued characters on the absolute Galois group of $K$ (using the correspondence between finite branched covers of $X$ and finite extensions of $K$) \details{Note that isomorphisms of branched 2-covers of $X$ must commute with the covering maps to $X$}. We say two branched covers of $X$ are isomorphic over $\overline{k}$ (or $\overline{k}$-isomorphic) if they are isomorphic as branched covers of $X \times_k \overline{k}$ after base change to $\overline{k}$. Base change to $\overline{k}$ respects the group structure on the set of branched 2-covers of $X$, so we also get a natural group structure on the set of $\overline{k}$-isomorphism classes of branched 2-covers of $X$, which we can view as a quotient group of the set of $k$-isomorphism classes.\details{
To be precise, we can think of isomorphism classes of branched 2-covers of $X \times_k \overline{k}$ as parametrized by $\Hom(\Gal(K^\text{sep}/\overline{k}K), \F_2)$. Consequently, we identify the set of $\overline{k}$-isomorphism classes of branched 2-covers of $X$ as the image of the restriction map \[
\Hom(\Gal(K^{\text{sep}}/K), \F_2) \to \Hom(\Gal(K^\text{sep}/\overline{k}K), \F_2)
\]
which has a natural group structure.
}

\begin{lemma}\label{lem:kummer} 
Let $X$ be a regular, proper, geometrically integral curve over a finite field $k$ of odd characteristic. Let $K \coloneqq k(X)$ be the function field of $X$. 
\begin{enumerate}[label=(\alph*)]
    \item We have an isomorphism of groups \[
    K^\times/(K^\times)^2 \xleftrightarrow{\quad\sim\quad} \left\{\text{branched 2-covers of } X\right\}/\text{isomorphism}
    \]
    given by sending a rational function $f \in K^\times$ to the regular, proper, integral curve $X_f \to X$ with function field $K(\sqrt{f})$ (or to the disjoint union of two copies of $X$ if $f \in (K^\times)^2$).

    The branch points of $X_f \to X$ are exactly those points of $X$ where $f$ has a zero or pole of odd order.

    \item The isomorphism from part (a) descends to an isomorphism \[
    K^\times/k^\times(K^\times)^2 \xleftrightarrow{\quad\sim\quad} \left\{\text{branched 2-covers of } X\right\}/\overline{k}\text{-isomorphism}
    \]
    sending $k^\times(K^\times)^2$ to the disjoint union of two copies of $X$.

    \item There is an inclusion $\Pic^0(X)(k)[2] \hookrightarrow K^\times/k^\times(K^\times)^2$ which, together with the isomorphism from part (b), identifies $\Pic^0(X)(k)[2]$ with the group of $\overline{k}$-isomorphism classes of \'etale 2-covers of $X$, i.e., the map from part (b) induces an isomorphism \[
    \Pic^0(X)(k)[2] \xleftrightarrow{\quad\sim\quad} \Hom(\pi_1^\twist(X), \F_2).
    \]
\end{enumerate}
\end{lemma}
\begin{remark}
Note that, since $|k^\times/(k^\times)^2| = 2$, part (b) shows that there are exactly two nonisomorphic branched covers of $X$ (up to $k$-isomorphism) in each $\overline{k}$-isomorphism class. These are given by $X_f$ and $X_{af}$, where $f \in K^\times$ and $a \in k^\times - (k^\times)^2$. 
\end{remark}
\begin{proof}
\begin{enumerate}[label=(\alph*)]
    \item The first isomorphism is exactly given by the usual Kummer theory together with the correspondence between finite branched covers of $X$ and finite extensions of $K$. Indeed, by construction this correspondence is an isomorphism with respect to the group structure we defined on the set of branched 2-covers above.

    To see that the branch points of $X_f \to X$ are exactly the points where $f$ has a zero or pole of odd order, we look at one point of $X$ at a time. Fix $p \in X$ and consider the discrete valuation ring $\sheafO_{X, p}$ of $X$ at $p$. Let $t$ be a choice of uniformizer at $p$. The image of $f$ in the fraction field $K = \Frac(\sheafO_{X, p})$ is of the form $t^{a_p}u$, where $a_p \in \Z$ is the order of vanishing of $f$ at $p$ and $u \in \sheafO_{X, p}^\times$ is a unit. 

    If $a_p$ is odd, then $k(X_f) = K(\sqrt{f})$ contains a square root of the uniformizer $ut$ at $p$, so $p$ is ramified in $X_f \to X$. Otherwise, $p$ is unramified. 

    \item To obtain the second isomorphism, we claim that two covers $X_{f_1} \to X$ and $X_{f_2} \to X$ are isomorphic over $\overline{k}$ if and only if $f_1/f_2 \in k^\times(K^\times)^2$.

    From Kummer theory for the function field $\overline{k}K$ of $X \times_k \overline{k}$, we have that $X_{f_1} \to X$ and $X_{f_2} \to X$ are isomorphic over $\overline{k}$ if and only if $f_1/f_2 \in (\overline{k}K^\times)^2$.\details{We are implicitly using that the function field of $X_f \times_k \overline{k}$ is $(\overline{k}K)(\sqrt{f})$, i.e., that the correspondence in part (a) plays well with base change to $\overline{k}$.} Thus, it suffices to show that $(\overline{k}K^\times)^2 \cap K^\times = k^\times(K^\times)^2$.

    Consider the short exact sequence $1\rightarrow \mu_2\rightarrow (\overline{k}K)^\times \xrightarrow{(\cdot)^2} ((\overline{k}K)^\times)^2\rightarrow 1$ of $\Gal(\overline{k}K/K)$ modules. Part of the associated long exact sequence reads \[
    K^\times \xrightarrow{\ \ (\cdot)^2\ \ } (\overline{k}K^\times)^2 \cap K^\times \xrightarrow{\ \ \delta\ \ } H^1(\Gal(\overline{k}K/K), \mu_2).
    \]
    We have $H^1(\Gal(\overline{k}K/K), \mu_2) \cong H^1(\Gal(\overline{k}/k), \mu_2) \cong k^\times/(k^\times)^2$ via the connecting map from the Kummer exact sequence for $\overline{k}^\times$.\details{
    We are using the fact that the map $\Gal(\overline{k}/k) \to \Gal(\overline{k}K/K)$ given by the $\Gal(\overline{k}/k)$-action on coefficients is an isomorphism (whose inverse is given by restriction to $\overline{k} \subset \overline{k}K$).
    } In particular, we notice that $k^\times$ maps surjectively onto $H^1(\Gal(\overline{k}K/K), \mu_2)$ via the connecting map $\delta$. Since $(K^\times)^2 = \ker \delta$, we have that $k^\times(K^\times)^2 = (\overline{k}K^\times)^2 \cap K^\times$ as we wanted.

    It follows that the isomorphism defined in (a) descends to an isomorphism of quotient groups as we wanted.

    \item The map $\Pic^0(X)(k)[2] \to K^\times/k^\times(K^\times)^2$ is defined as follows. 
    
    Given a 2-torsion divisor class represented by a divisor $D$, we have that $2D = \div f$ is principal. The function $f$ is defined up to scaling by constants $k^\times$, so we obtain a class in $K^\times/k^\times$. Note that if $D, D'$ are two divisors with $2D = \div f$ and $2D' = \div f'$, then $2(D + D') = \div ff'$.
    
    If $D \sim D'$ are linearly equivalent, then $D = D' + \div g$ for some $g \in K^\times$. Then $2D = 2D' + \div g^2$, so a rational function cutting out $2D$ differs from a rational function cutting out $2D'$ by a square in $K^\times$. Thus, we have a well-defined way to associate 2-torsion divisor classes to classes in $K^\times/k^\times(K^\times)^2$. Moreover, the resulting map $\Pic^0(X)(k)[2] \to K^\times/k^\times(K^\times)^2$ is a group homomorphism.

    To show the map $\Pic^0(X)(k)[2] \to K^\times/k^\times(K^\times)^2$ is an injection, suppose that $D$ is a divisor on $X$ such that $2D$ is cut out by $f \in k^\times(K^\times)^2$, i.e., $\div f = 2D$. We may assume $f \in (K^\times)^2$ because scaling does not change its zeroes or poles. Then $D$ is cut out by any square root of $f$ and is principal.

    Finally, $f \in K^\times/k^\times(K^\times)^2$ is in the image of the inclusion $\Pic^0(X)(k)[2] \to K^\times/k^\times(K^\times)^2$ if and only if $\div f = 2D$ has only even coefficients, i.e., $f$ has even order of vanishing everywhere. So, the image of this inclusion corresponds exactly to classes of \'etale covers of $X$ by part (a).
\end{enumerate}
\end{proof}

Part (a) of Lemma~\ref{lem:kummer} shows how the ramification of $X_f \to X$ can be read off of the divisor \[
\div f = \sum_{p \in X} a_p[p].
\]
Since the divisor $\sum_{p \in X} a_p[p]$ only determines $X_f$ up to $\overline{k}$-isomorphism, we will need a finer way to tell apart the two covers in a given $\overline{k}$-isomorphism class. One way to do this is to look at the splitting behavior of $X_f \to X$ over an unbranched point of odd degree: the two covers of $X$ will have different splitting behavior at every point. For technical reasons (Lemma~\ref{lem:pairing-by-hyperelliptics}) we will find it helpful to be able to distinguish covers of $X$ also by looking at ``second-order'' behavior at branch points.

Fix a point $p \in X$ and let $t$ be a uniformizer at $p$. Let $\hat{f} = t^{a_p}u$ be the image of $f$ in the completion $K_p$, where $u \in \widehat{\sheafO}_{X, p}^\times$ is a unit in the completed local ring. 

Let $c_t(f, p) \in k(p)^\times$ be the image of $u$ in the residue field $k(p)$ of $\widehat{\sheafO}_{X, p}$. Changing $X_f$ by $k$-isomorphism will change $f$ by a factor in $(K^\times)^2$, which will change each $a_p$ by a multiple of $2$ and $c_t(f, p)$ by a factor in $(k(p)^\times)^2$. This leads us to define an invariant $c_t(X_f \to X, p) \in k(p)^\times/(k(p)^\times)^2 \cong k^\times/(k^\times)^2$, which we call the \textit{second-order class} of $X_f \to X$ at $p$. It is invariant under $k$-isomorphism over $X$, but not under $\overline{k}$-isomorphism.

Twisting $f$ by a constant $a \in k^\times/(k^\times)^2$ results in \begin{equation}\label{eq:twisting-second-order}
c_t(X_{af} \to X, p) = a^{\deg(p)}c_t(X_f \to X, p),
\end{equation}
since the map $k^\times/(k^\times)^2 \to k(p)^\times/k(p)^\times)^2 \cong k^\times/(k^\times)^2$ is identity if $\deg(p)$ is odd and trivial if $\deg(p)$ is even.

If $s, t$ are two uniformizers at $p$, then $t = vs$ for some $v \in \widehat{\sheafO}_{X, p}^\times$, whence $\hat{f} = s^{a_p}v^{a_p}u$ and $c_s(X_f \to X, p)$ differs from $c_t(X_f \to X)$ by the image of $v^{a_p}$ in $k^\times/(k^\times)^2$. If $v$ maps to a square in $k(p)^\times$, then we always have $c_s(X_f \to X, p) = c_t(X_f \to X, p)$. By Hensel's lemma, the reduction map $\widehat{\sheafO}_{X, p}^\times/(\widehat{\sheafO}_{X, p}^\times)^2 \to k(p)^\times/(k(p)^\times)^2$ is an isomorphism. Thus, $v$ maps to a square in $k(p)^\times$ if and only if $v$ is a square in $\widehat{\sheafO}_{X, p}^\times$.\details{By definition, the map $\widehat{\sheafO}_{X, p}^\times \to k(p)^\times/(k(p)^\times)^2$ is surjective, and the kernel contains $(\widehat{\sheafO}_{X, p}^\times)^2$. To see that the kernel is precisely $(\widehat{\sheafO}_{X, p}^\times)^2$, let $g \in \widehat{\sheafO}_{X, p}^\times$ map to a square in $k(p)$. Then the polynomial $x^2 - g \in \widehat{\sheafO}_{X, p}[x]$ has a simple root after passing to $k(p)$ (we are using that $k$ had odd characteristic here). So, by Hensel's lemma, it has a root in $\widehat{\sheafO}_{X, p}$ as well.} So, if $2 \mid a_p$ then $c_t(X_f \to X, p)$ is independent of the choice of uniformizer $t$, whereas if $2 \nmid a_p$ then $c_t(X_f \to X, p)$ depends only on the class of $t$ up to scaling by $(\widehat{\sheafO}_{X, p}^\times)^2$.

We will refer to a choice of uniformizer for $\widehat{\sheafO}_{X, p}$ at $p$ up to scaling by $(\widehat{\sheafO}_{X, p}^\times)^2$ as a \textit{uniformizer class}. Since $\widehat{\sheafO}_{X, p}^\times/(\widehat{\sheafO}_{X, p}^\times)^2 \cong k(p)^\times/(k(p)^\times)^2 \cong \Z/2\Z$, there are two different uniformizer classes at each point $p$.

If $2 \mid a_p$, we may omit the subscript and write $c(X_f \to X, p) \coloneqq c_t(X_f \to X, p)$. However, if $2 \nmid a_p$, then $c_t(X_f \to X, p)$ ranges over all of $k^\times/(k^\times)^2$ depending on the choice of $t$. Later, we will use second-order classes to divide up families of curves into smaller ones over which we can prove equidistribution of the R\'edei matrix.

We have a trichotomy depending on $a_p$ and $c_t(X_f \to X, p)$: \begin{enumerate}[label=(\arabic*)]
    \item If $2 \nmid a_p$, then $X_f \to X$ is ramified at $p$.
    \item If $2 \mid a_p$ and $c(X_f \to X, p) \coloneqq c_t(X_f \to X, p) = 1$, then $p$ splits in $X_f \to X$.
    \item If $2 \mid a_p$ and $c(X_f \to X, p) \coloneqq c_t(X_f \to X, p) \neq 1$, then $p$ is inert in $X_f \to X$.
\end{enumerate} \details{
We can compute the scheme-theoretic preimage of $p$ in $X_f \to X$ after passing to a formal neighborhood $\Spec \widehat{\sheafO}_{X, p}$ of $p$. This will be the spectrum of the ring $\widehat{\sheafO}_{X, p}[z]/(z^2 - f)$. If $2 \nmid a_p$, then $\widehat{\sheafO}_{X, p}[z]/(z^2 - f)$ contains a square root of $t$, i.e., $p$ is ramified. Otherwise, $\widehat{\sheafO}_{X, p} \to \widehat{\sheafO}_{X, p}[z]/(z^2 - f)$ is unramified. If $f$ is a square in $\widehat{\sheafO}_{X, p}[z]/(z^2 - f)$, then the ring splits as a product, and the scheme-theoretic preimage is two identical formal neighborhoods (the split case). Otherwise, the map $\widehat{\sheafO}_{X, p} \to \widehat{\sheafO}_{X, p}[z]/(z^2 - f)$ is obtained by adjoining the square root of a unit, which results in a residue field extension (the inert case).
}

\begin{remark}\label{rmk:determining-covers}
    Since twisting $f$ by $a \in k^\times/(k^\times)^2$ changes the second-order class of $X_f \to X$ at $p$ by a factor of $a^{\deg(p)}$, a $k$-isomorphism class of branched covers of $X$ is determined by a class in $K^\times/k^\times(K^\times)^2$ together with a second-order class at any one odd-degree point of $X$. In particular:
    \begin{itemize}
        \item Since a rational function on $\PP^1$ is determined up to scalars by its divisor, a hyperelliptic curve $X \to \PP^1$ is determined by its set of branch points together with a second-order class at any one odd-degree point of $\PP^1$.
        \item By Lemma~\ref{lem:kummer}(c), an \'etale 2-cover of $X$ is determined by a 2-torsion class in $\Pic^0(X)(k)$ together with a second-order class at any one odd-degree point of $X$.
    \end{itemize}
\end{remark}

\subsection{Parametrization of $\Pic^0(X_f)(k)[2]$ and $\Hom(\pi_1^\twist(X_f), \Z/2\Z)$}\label{sect:basis}

In this section, let $P$ be a smooth, proper, irreducible curve over a perfect field $k$. We will always apply the results of this section with $P = \PP^1_k$, but some of them hold more generally whenever $\Pic^0(P)(k) = 0$. In this subsection only, we do not assume $k$ is finite, but we only need the following results in that case. Fix a rational function $f \in K^\times = k(P)^\times$ and let $X$ be the 2-cover of $P$ associated to $f$ via the correspondence of Lemma~\ref{lem:kummer}(a).

We give an explicit basis for $\Pic^0(X)(k)[2]$ (and therefore, by Lemma~\ref{lem:kummer}(c), also for $\Hom(\pi_1^\twist(X), \Z/2\Z)$ when $k$ is finite) in terms of the branch points of $X \to P$. The matrix associated to the pairing of Subsection~\ref{sect:pairing}, expressed in this basis, is what we will call the R\'edei matrix.

The following sequence of three lemmas gives a function field analogue of the ``imaginary'' case of Gauss genus theory. Genus theory for finding the size of the 2-torsion in the ideal class group of an imaginary quadratic function field was developed by Artin \cite{artin_quadratische_1924}. See \cite{stokvis2025governingfieldshyperellipticfunction} for a good review of this and its generalizations and construction of a R\'edei map for the ideal class group of a quadratic function field. 

A similar result for the 2-torsion subgroup in the Jacobian of a hyperelliptic curve is due to Cornelissen \cite{cornelissen_two-torsion_2001}, which views this subgroup as a subspace in the 2-torsion in the $\overline{k}$-points of the Jacobian. Since we want to parametrize 2-torsion explicitly by sets of branch points, we give a slightly different model of the 2-torsion as explicitly generated by divisors supported in the set of ramification points of $X \to P$. A similar result in the case where $X \to P$ is ramified at a degree-1 point can be found in the thesis of Kosters \cite[Theorem 4.9]{kostersthesis}.

\begin{lemma}\label{lem:phi-surjection}
Let $k$ be a perfect field characteristic not 2 with $\Br(k) = 0$ and let $\pi\colon X \to P$ be a degree 2 branched cover of geometrically connected curves defined by $f \in k(P)^\times$ via the correspondence of Lemma~\ref{lem:kummer}(a) and ramified at a finite set of closed points $S \subset X$ containing at least one point of odd degree. Suppose $\Pic^0(P)(k) = 0$.

Denote by $\Z\langle S\rangle^0$ the group of degree zero $k$-rational divisors on $X$ supported in $S$. 

The natural map \[
\phi\colon \Z\langle S\rangle^0 \longrightarrow \Pic^0(X)(k)[2]
\]
sending a divisor supported in $S$ to its associated divisor class is surjective.
\end{lemma}
\begin{proof}[Proof of Lemma~\ref{lem:phi-surjection}.]
We only need $\Br(k) = 0$ so that we can view elements of $\Pic^0(X)(k)$ (a priori divisor classes on $X \times_k \overline{k}$ fixed by the action of $\Gal(\overline{k}/k)$) as formal $\Z$-linear combinations of closed points of $X$ up to rational equivalence. 

We observe that if $D$ is a degree zero divisor on $X$ supported in $S$, then $2D = \pi^*D'$ for some degree zero divisor $D'$ on $P$ supported at the branch points of $\pi$, whence $2D$ is principal by the assumption that $\Pic^0(P)(k) = 0$. So, the natural map \[
\Z\langle S\rangle^0 \lhook\joinrel\longrightarrow \Div^0(X)(k) \relbar\joinrel\twoheadrightarrow \Pic^0(X)(k)
\]
has image in $\Pic^0(X)(k)[2]$. We wish to show this map is a surjection onto $\Pic^0(X)(k)[2]$; in other words, we want to check that every $k$-rational 2-torsion divisor class contains a $k$-rational divisor supported in $S$.

Let $G = \langle\sigma\rangle$ be the Galois group of the cover $X \to P$. We observe (using $\Pic^0(P)(k) = 0$) that $\sigma$ acts on $\Pic^0(X)(k)$ by inversion. \details{Indeed, if $[D] \in \Pic(X)(k)$ is a divisor class, then $[D] + \sigma[D] = [D + \sigma D]$ is represented by a divisor $D + \sigma D$ which is of the form $\pi^*D'$ for some $D' \in \Div(P)$. This can be checked on the level of points. If $\deg D = 0$, then $\deg D' = 0$, so since $\Pic^0(X)(k) = 0$, we have that $D'$ is principal. Thus $D + \sigma D$ is principal, and $[D] + \sigma[D] = 0$.} So, $\Pic^0(X)(k)[2]$ is the group of $G$-invariants $\Pic^0(X)(k)^G$. 

We will first show that any $G$-invariant (i.e., 2-torsion) degree 0 divisor class on $X$ contains a $G$-invariant divisor. We have an exact sequence \[
0 \longrightarrow \Prin(X)(k) \longrightarrow \Div^0(X)(k) \longrightarrow \Pic^0(X)(k) \longrightarrow 0
\]
where $\Prin(X)(k)$ denotes $k$-rational principal divisors and $\Div^0(X)(k)$ denotes $k$-rational divisors of degree $0$.

The associated long exact sequence on cohomology begins \[
0 \longrightarrow \Prin(X)(k)^G \longrightarrow \Div^0(X)(k)^G \longrightarrow \Pic^0(X)(k)^G \longrightarrow H^1(G, \Prin(X)(k))
\]
so it suffices to show that $H^1(G, \Prin(X)(k)) = 0$. We have the short exact sequence \[
0 \longrightarrow k^\times \longrightarrow k(X)^\times \longrightarrow \Prin(X)(k) \longrightarrow 0
\]
which yields a long exact sequence on cohomology, part of which is: \[
H^1(G, k(X)^\times) \longrightarrow H^1(G, \Prin(X)(k)) \longrightarrow H^2(G, k^\times) \longrightarrow H^2(G, k(X)^\times).
\]
By Hilbert's Theorem 90, the first term is $H^1(G, k(X)^\times) = 0$, so we identify $H^1(G, \Prin(X)(k))$ with the kernel of $H^2(G, k^\times) \to H^2(G, k(X)^\times)$. We have $H^2(G, k^\times) \cong k^\times/(k^\times)^2$ and $H^2(G, k(X)^\times) \cong k(P)^\times/N(k(X)^\times)$. (Here $N(k(X)^\times)$ denotes the subgroup of $k(X)^\times$ consisting of norms of elements of $k(X)^\times$, i.e., rational functions of the form $h\sigma(h)$ for $h \in k(X)^\times$.)

Using these isomorphisms, the map $H^2(G, k^\times) \to H^2(G, k(X)^\times)$ is identified with the map \[
k^\times/(k^\times)^2 \longrightarrow k(P)^\times/N(k(X)^\times)
\]
induced by the inclusion $k^\times \hookrightarrow k(P)^\times \hookrightarrow k(X)^\times$. \details{
We compute the cohomology of $G$ with coefficients in an arbitrary $G$-module $A$ using $H^i(G, A) \cong \operatorname{Ext}^i_{\Z[G]}(\Z, A)$ via the usual resolution \[
\cdots \xrightarrow{\sigma - 1} \Z[G] \xrightarrow{\sigma + 1} \Z[G] \xrightarrow{\sigma - 1} \Z[G] \longrightarrow \Z \to 0
\]
After applying $\Hom_{\Z[G]}(\cdot, A)$ to the resolution we get the complex \[
\cdots \xrightarrow{\sigma - 1} A \xrightarrow{\sigma + 1} A \xrightarrow{\sigma - 1} A
\]
If we have another $G$-module $B$ and a map $A \to B$, we get an induced map on complexes. Taking $H^2$ of this map gives the map \[
H^2(G, A) \cong A^G/(\sigma + 1)A \longrightarrow B^G/(\sigma + 1)B \cong H^2(G, B)
\]
If we set $A = k^\times$ with the trivial $G$-action and $B = k(X)^\times$ we get the claimed result.
}We want to see that this map is injective; in other words, we want to show that if $a \in k^\times$ is a norm from $k(X)^\times$ (i.e., is of the form $h\sigma(h)$ for some $h \in k(X)^\times$), then $a \in (k^\times)^2$.

Suppose that $a = h\sigma(h)$ for some $h \in k(X)^\times$. Then we begin by observing that $\div(h)$ is supported away from $S$. Indeed, if $p \in S$, then $p$ is fixed by $\sigma$, so $v_p(h) = v_{\sigma(p)}(\sigma(h)) = v_p(\sigma(h)) = v_p(ah^{-1}) = v_p(h^{-1}) = -v_p(h)$, and so $v_p(h) = 0$. Let $p \in S$ be of odd degree and consider the value of $h$ at $p$, an element of the residue field $k(p)$ at $p$. We have $h(p) = \sigma(h)(\sigma(p)) = \sigma(h)(p) = ah^{-1}(p) = a(h(p))^{-1}$, so $a = h(p)^2$. Now $a$ is a square in the odd-degree extension $k(p)$ of $k$, so $a$ is a square in $k$. \details{or else $\sqrt{a}$ would generate an even-degree subextension of $k(p)$.}

This shows $H^1(G, \Prin(X)(k)) = 0$, so the map $\Div^0(X)(k)^G \to \Pic^0(X)(k)^G$ is surjective. The conclusion follows from the fact (which we will prove in Lemma~\ref{lem:moving-to-S}) that any divisor $D \in \Div^0(X)(k)^G$ is linearly equivalent to a divisor $D_S \in \Z\langle S\rangle^0$. Indeed, suppose $[D] \in \Pic^0(X)(k)[2]$. We have just shown that we can choose a representative $D$ for this divisor class such that $D$ is fixed by the action of $G$. Then Lemma~\ref{lem:moving-to-S} shows that $D$ is equivalent to a divisor supported on $S$, i.e., we can find a divisor $D_S \in [D]$ with $D_S \in \Z\langle S\rangle^0$. This is exactly the statement of surjectivity of $\phi$. 
\end{proof}

\begin{lemma}\label{lem:moving-to-S}
With setup as in Lemma~\ref{lem:phi-surjection}, let $D \in \Div^0(X)(k)^G$ be a divisor fixed by the automorphism group of the cover $X \to P$. Then $D$ is linearly equivalent to a divisor $D_S$ in $\Z\langle S\rangle^0$ such that for $p \in S$, we have $v_p(D_S) \equiv v_p(D) \pmod{2}$. 
\end{lemma}
\begin{proof}
Write $D = D' + E$, where $D' \in \Div(X)(k)$ is supported in $S$ and $E \in \Div(X)(k)$ is supported outside of $S$. Since $D'$ is supported in $S$, it is fixed by $\sigma$, whence $E = D - D'$ is fixed by $\sigma$ as well. Then $E = \pi^*\widetilde{E}$ for some divisor $\widetilde{E} \in \Div(P)(k)$. \details{This is because $E$ is a sum of fibers of the map $\pi$, and every fiber outside of $S$ is pulled back from a point of $X$. The same argument doesn't work if $E$ has points in $S$, since if $p \in S$, we only have $2[p]$ pulled back from $P$.}

Next we will show that $\widetilde{E}$ is equivalent to a divisor supported on the branch points of $\pi$.

Note that $D = D' + E$ has degree zero, so $\deg E = -\deg D'$. Moreover, we have $\deg E = 2\deg \widetilde{E}$. Since $D'$ is supported in $S$, the degree of $D'$ is a $\Z$-linear combination of degrees of points in $S$. In other words, if $d_0, \dots, d_n$ are the degrees of the points of $S$, we have $\gcd(d_0, \dots, d_n) \mid \deg D'$, so that $\gcd(d_0, \dots, d_n) \mid 2\deg \widetilde{E}$. However, since at least one of the degrees $d_0, \dots, d_n$ is odd, we have that $\gcd(d_0, \dots, d_n)$ is odd as well, whence $\gcd(d_0, \dots, d_n) \mid \deg \widetilde{E}$. This means there is a divisor $\widetilde{D}$ supported at the branch points of $\pi$ with the same degree as $\widetilde{E}$; since $\Pic^0(P)(k) = 0$, we must have $\widetilde{D} \sim \widetilde{E}$. Then $D' + \pi^*\widetilde{D}$ is supported in $S$ and is equivalent to $D$, as we wanted. 

The congruence condition comes from the fact that, for $p \in S$, we have $v_p(\pi^*\widetilde{D}) = 2v_{\pi(p)}(\widetilde{D})$.
\end{proof}

The map $\phi$ vanishes on $2\Z\langle S\rangle^0$, so it factors through $\Z\langle S\rangle^0/2\Z\langle S\rangle^0$. We need to understand the kernel of the induced map $\Z\langle S\rangle^0/2\Z\langle S\rangle^0 \to \Pic^0(X)(k)[2]$ for it to be a good parametrization of the 2-torsion elements. It turns out that this kernel is one-dimensional (as an $\F_2$-vector space) and essentially spanned by the divisor of the rational function $\sqrt{f}$ on $X$ (we may need to modify this divisor somewhat to get something supported on $S$).

\begin{lemma}\label{lem:phi-kernel}
Let setup be as in Lemma~\ref{lem:phi-surjection}, and assume further that $P$ has genus 0 (so the condition $\Pic^0(P)(\overline{k}) = 0$ is automatic). Then there is a rational function $h \in k(X)^\times$ such that $\div(h\sqrt{f}) \in \Z\langle S\rangle^0$. The kernel of $\phi$ is $\langle \div(h\sqrt{f}), 2\Z\langle S\rangle^0\rangle \subset \Z\langle S\rangle^0$.

In particular, the induced surjection of $\F_2$-vector spaces \[
\phi_2\colon \F_2^{|S| - 1} \cong \Z\langle S\rangle^0/2\Z\langle S \rangle^0 \relbar\joinrel\twoheadrightarrow \Pic^0(X)(k)[2]
\]
has a one-dimensional kernel spanned by the image of $\div(h\sqrt{f})$.
\end{lemma}
\begin{proof}
We observe that $\sigma(\div\sqrt{f}) = \div(-\sqrt{f}) = \div\sqrt{f}$, so by Lemma~\ref{lem:moving-to-S} the divisor $\div \sqrt{f}$ is equivalent to a divisor supported on $S$. Let $h \in k(X)^\times$ be the rational function witnessing this equivalence, so $\div(h\sqrt{f})$ is supported on $S$. 

We observe that $\sqrt{f}$ has odd valuation at each point of $S$ because the branch locus of $X \to P$ is exactly the set of points of $P$ where $f$ has odd valuation. By the congruence condition of Lemma~\ref{lem:moving-to-S}, the function $h\sqrt{f}$ also has odd valuation at the points of $S$.

Since $\div(h\sqrt{f})$ is principal, it lies in the kernel of $\phi$. So, it remains to check that there are no other inequivalent principal divisors supported on $S$.

We will first verify the statement over $\overline{k}$, then descend to $k$. Let $\overline{S}$ be the set of geometric ramification points and let $G_k$ be the absolute Galois group of $k$, so that $\Z\langle S\rangle \cong \Z\langle \overline{S}\rangle^{G_k}$. Let $\overline{\phi}\colon \Z\langle\overline{S}\rangle^0 \to \Pic^0(X)(\overline{k})[2]$ be the natural map sending a divisor on $X \times_k \overline{k}$ supported on $\overline{S}$ to its associated divisor class, and let $\overline{\phi}_2$ be the induced map $\Z\langle \overline{S}\rangle^0/2\Z\langle \overline{S}\rangle^0 \twoheadrightarrow \Pic^0(X)(\overline{k})[2]$.

By standard facts about Abelian varieties (see, e.g. \cite[Proposition 39.9.11]{stacks-project}), we have that $\Pic^0(X)(\overline{k})[2]$ is an $\F_2$-vector space of dimension $2g$, where $g$ is the genus of $X$. On the other hand, Riemann-Hurwitz gives us $\#\overline{S} = 2g + 2$, so $\Z\langle \overline{S}\rangle^0$ is free of rank $2g + 1$ over $\Z$. Thus, the map $\overline{\phi}_2\colon \Z\langle \overline{S}\rangle^0/2\Z\langle \overline{S}\rangle^0 \twoheadrightarrow \Pic^0(X)(\overline{k})[2]$ has one-dimensional kernel. The divisor $\div(h\sqrt{f})$ is not in $2\Z\langle \overline{S}\rangle^0$ because $h\sqrt{f}$ has odd valuation at points of $\overline{S}$, so the image of $\div(h\sqrt{f})$ in $\Z\langle \overline{S}\rangle^0/2\Z\langle \overline{S}\rangle^0$ spans the kernel of $\overline{\phi}_2$. Thus, the kernel of the map $\overline{\phi}\colon \Z\langle \overline{S}\rangle^0 \to \Pic^0(X)(\overline{k})[2]$ is generated by $\div(h\sqrt{f}) + 2\Z\langle \overline{S}\rangle^0$.

Now by left exactness of taking $G_k$-invariants, we have \[
\ker(\phi) = \ker(\overline{\phi})^{G_k}
\]
and the subgroup on the right hand side is generated by $\div(h\sqrt{f}) + 2\Z\langle S\rangle^0$.
\end{proof}

Combined with Lemma~\ref{lem:kummer}(c), we also get a two-to-one parametrization of $\overline{k}$-isomorphism classes of \'etale 2-covers of $X$ in $\Hom(\pi_1^\twist(X), \Z/2\Z)$ by the same vector space $\Z\langle S\rangle^0/2\Z\langle S\rangle^0$. However, we will not need to use this directly.

The end goal is to get a matrix representation for the pairing $\ip{\cdot}{\cdot}_{X}$. To avoid having to work with the unwieldy relation coming from the extra principal divisor on $X$, we will lift $\ip{\cdot}{\cdot}_{X}$ to a pairing on $\Z\langle S\rangle^0/2\Z\langle S\rangle^0$, for which we can give an easy-to-handle explicit basis.

\begin{lemma}\label{lem:basis}
With setup as in Lemma~\ref{lem:phi-kernel}, let $S = \{p_0, \dots, p_m\}$ such that $p_0$ has odd degree. Then $\Z\langle S\rangle^0/2\Z\langle S\rangle^0$ is an $\F_2$-vector space of dimension $m$ spanned by the images $\overline{e}_i$ of the elements \[
e_i \coloneqq \deg(p_0)[p_i] - \deg(p_i)[p_0] \in \Z\langle S\rangle^0, \qquad\text{ for } 1 \leq i \leq m.
\]
In this basis, the kernel of the map $\phi_2$ in Lemma~\ref{lem:phi-kernel} is spanned by the all-ones vector $(1, \dots, 1)$.
\end{lemma}
\begin{proof}
We view $\Z\langle S\rangle^0/2\Z\langle S\rangle^0$ as a subspace of the $\F_2$-vector space $\F_2\langle S\rangle$ spanned by $S$. \details{We have a map $\Z\langle S\rangle^0 \hookrightarrow \Z\langle S\rangle \twoheadrightarrow \F_2\langle S \rangle$ which descends to $\Z\langle S\rangle^0/2\Z\langle S\rangle^0 \to \F_2\langle S\rangle$. This map is injective because $2\Z\langle S\rangle \cap \Z\langle S \rangle^0 = 2\Z\langle S \rangle^0$.} Since $\deg(p_0)$ is odd, the $\overline{e}_i$ are linearly independent inside $\F_2\langle S\rangle$.

Since $\Z\langle S\rangle^0$ is free abelian of rank $m$, the quotient $\Z\langle S\rangle^0/2\Z\langle S\rangle^0$ is an $m$-dimensional $\F_2$-vector space, so the $\overline{e}_i$ must form a basis.

In the proof of Lemma~\ref{lem:phi-kernel}, we observed that $h\sqrt{f}$ has odd valuation at each point of $S$. Thus, $\div(h\sqrt{f})$ has a nonzero $\overline{e}_i$ component for each $i$, so must map to the all-ones vector under $\Z\langle S\rangle^0 \to \Z\langle S\rangle^0/2\Z\langle S\rangle^0 \cong \F_2^m$.
\end{proof}

\subsection{The matrix of the pairing for a fixed curve}

In this subsection, let $X \to P$ be a degree 2 branched cover with assumptions as in Lemma~\ref{lem:basis} (in particular, assume $P$ has genus 0 and $X \to P$ is branched at a point $p_0$ of odd degree). We describe a matrix whose kernel is related to the left kernel of $\ip{\cdot}{\cdot}_X$. The entries of this matrix will be Artin symbols describing the ramification of points in hyperelliptic curves. Writing the matrix in this form will allow us to prove that it equidistributes among matrices with certain symmetry conditions as $X$ varies over all hyperelliptic curves of fixed genus. We freely use notation from the statements of Lemmas~\ref{lem:phi-surjection}, \ref{lem:phi-kernel}, and \ref{lem:basis}.

The first step is to get a handle on what exactly the pairing $\ip{\cdot}{\cdot}_X$ is in terms of the basis described in Lemma~\ref{lem:basis}. 

Using the isomorphism $\Pic^0(X)(k)[2] \cong \Hom(\pi_1^\twist(X), \Z/2\Z)$ from Lemma~\ref{lem:kummer}(c), we view $\ip{\cdot}{\cdot}_X$ as a pairing \[
\ip{\cdot}{\cdot}_X\colon \Pic^0(X)(k)[2] \times \Pic^0(X)(k)[2] \to \F_2
\]
and we get an induced pairing \[
\ip{\cdot}{\cdot}_X'\colon \Z\langle S\rangle^0/2\Z\langle S\rangle^0 \times \Z\langle S \rangle^0/2\Z\langle S \rangle^0 \to \F_2
\]
via the map $\phi_2\colon \Z\langle S\rangle^0/2\Z\langle S\rangle^0 \to \Pic^0(X)(k)[2]$ discussed in Subsection~\ref{sect:basis} (especially Lemma~\ref{lem:phi-kernel}). The left kernel of $\ip{\cdot}{\cdot}_X'$ has dimension one more than the left kernel of $\ip{\cdot}{\cdot}_X$. We will also implicitly think of $\ip{\cdot}{\cdot}_X'$ as a map $\Z\langle S\rangle^0 \times \Z\langle S \rangle^0 \to \F_2$ and write $\ip{D}{D'}_X'$ for $D, D' \in \Z\langle S\rangle^0$.

Recall the basis $\{\overline{e}_i \mid i = 1, \dots, m\}$ for $\Z\langle S\rangle^0/2\Z\langle S\rangle^0$ defined in Lemma~\ref{lem:basis}. We want to compute $\ip{e_i}{e_j}_X' = \ip{\phi_2(\overline{e}_i)}{\phi_2(\overline{e}_j)}_X$ for $1 \leq i, j \leq m$. To do this, let $\chi_j \in \Hom(\pi_1^\twist(X), \Z/2\Z)$ correspond to $\phi_2(\overline{e}_j)$ via the isomorphism in Lemma~\ref{lem:kummer}(c).

We need to pick a lift of $\chi_j$ to $\Hom(\pi_1^\etale(X), \F_2)$ to use to compute the pairing using \eqref{eq:pairing-formula}. Recall that elements of $\Hom(\pi_1^\etale(X), \F_2)$ correspond to \'etale degree 2 covers of $X$. For $p$ a closed point in $X$, evaluation of a character $\pi_1^\etale(X) \to \F_2$ at $\Frob_p$ detects the action of $\Frob_p$ on the fiber over $p$ of the associated cover. If $\Frob_p$ maps to $1 \in \F_2$, then the fiber above $p$ has a nontrivial Galois action and therefore consists of a single closed point with residue field a quadratic extension of $k(p)$; in other words, the associated cover is inert at $p$. Similarly, if $\Frob_p$ maps to $0$, then this cover is split. 

We will choose the lift of $\chi_j$ corresponding to the cover $X_j \to X$ which is split at $p_0$ (see Remark~\ref{rmk:determining-covers}), i.e., the lift $\widetilde{\chi_j}$ satisfying $\widetilde{\chi_j}(\Frob_{p_0}) = 0$. Then we have \begin{align*}
\ip{e_i}{e_j}_X' &= \ip{\phi(e_i)}{\phi(e_j)}_X \\
&= \deg(p_0)\widetilde{\chi_j}(\Frob_{p_i}) - \deg(p_i)\widetilde{\chi_j}(\Frob_{p_0})  \\
&=\widetilde{\chi_j}(\Frob_{p_i}).
\end{align*}
In other words, $\ip{e_i}{e_j}_X'$ is $0$ if $X_j \to X$ is split at $p_i$ and $1$ if $X_j \to X$ is inert at $p_i$, i.e., \[
\ip{e_i}{e_j}_X' = c(X_j \to X, p_i)_+
\]
where the $_+$ subscript denotes that the right hand side is being viewed as an element of $\F_2$.

\subsection{Parametrization of hyperelliptic curves and their R\'edei matrices}\label{sect:param-curves}

In this subsection we describe a parametrization of all hyperelliptic curves defined over $k$ that lends itself well to computing the pairing $\ip{\cdot}{\cdot}_X'$ associated to each curve. Later, we see that as the chosen curve ranges over certain cross-sections of this parametrization, the associated R\'edei matrix nearly equidistributes. Here again, $k$ is a finite field of odd characteristic.

We consider tuples of the form $B = (p_0, p_1, \dots, p_n, t)$ where $p_0, p_1, \dots, p_n \in \PP^1_k$ are closed points satisfying:
\begin{itemize}
    \item $\deg(p_1) \leq \dots \leq \deg(p_n)$;
    \item $\deg(p_0)$ is odd;
    \item if $\deg(p_i)$ is odd for some $1 \leq i \leq n$, then $\deg(p_0) \leq \deg(p_i)$;
    \item $\sum_{i=0}^n \deg(p_i)$ is even;
\end{itemize}
and $t$ is a choice of uniformizer at $p_0$ defined up to $(\widehat{\sheafO}_{\PP^1, p_0}^\times)^2$ (i.e., a uniformizer class at $p_0$). To such a tuple $B$ we associate the unique hyperelliptic curve $\pi_B\colon X_B \to \PP^1_k$ branched at $p_0, \dots, p_n$ such that $c_t(\pi_B, p_0) = 1$ if $\deg(p_0)$ is odd (see Remark~\ref{rmk:determining-covers} for why such a curve exists and is unique up to $k$-isomorphism).

Each hyperelliptic curve $X \to \PP^1_k$ with an odd branch point is obtained in this way, although most curves are associated to many tuples $B$ because these tuples are ordered. Note that the number of tuples associated to a curve depends only on the degrees (with multiplicity) of its ramification points. \details{Suppose $X$ is branched at points of degrees $d_0, d_1 \leq \dots \leq d_n$ with $v_2(d_0) \leq v_2(d_i)$ for each $1 \leq i \leq n$ and such that if $v_2(d_0) = v_2(d_i)$ then $d_0 \leq d_i$. Say $X$ has $n_i$ branch points of degree $d_i$. Then the number of tuples associated to $X$ is $\prod_i n_i!$, since a choice of $p_0$ and $C$ uniquely determines $t$ up to $(\widehat{\sheafO}_{\PP^1, p_0}^\times)^2$ subject to the constraint $c_t(X \to \PP^1_k, p_0) = 1$.}

We will use this parametrization to define a uniformly random hyperelliptic curve with fixed total branch degree $d$ as follows. First, pick a set of branch points of total degree $d$ at random according to an appropriate distribution (see Subsection~\ref{sect:branch-degrees}). If every point in this set has even degree, output one of the two hyperelliptic curves with this branch locus uniformly at random. Otherwise, order the points as $(p_0, \dots, p_n)$ satisfying the above conditions, pick a uniformizer class $t$ at $p_0$ at random, and output $X_B$ for $B = (p_0, \dots, p_n, t)$.

Almost all of our analysis concerns the second case where one of the random branch points has odd degree. We will see in Corollary~\ref{cor:assumptions} that this happens almost all the time as $d$ grows. In fact, we will show that the number of odd-degree branch points concentrates around half of the total number of branch points.

If $p \in \PP^1_k$ is a branch point of $\pi\colon X \to \PP^1_k$, we write $\pi^{-1}(p)$ for the unique closed point of $X$ above $p$, i.e., the \textit{set-theoretic} preimage of $p$ under $\pi$.

Given points $p_0, p \in \PP^1_k$ with $p_0$ of odd degree, define \[
e_{p_0, p} \coloneqq \deg(p_0)[p] - \deg(p)[p_0]
\]
and, if $\pi\colon X \to \PP^1_k$ is a hyperelliptic curve branched at $p_0$ and $p$, we write \[
\pi^{-1}e_{p_0, p} \coloneqq\deg(p_0)[\pi^{-1}(p)] - \deg(p)[\pi^{-1}(p_0)] = \frac{1}{2}\pi^*e_{p_0, p}.
\]
\begin{lemma}\label{lem:pairing-by-hyperelliptics}
Fix $p_0 \in \PP^1_k$ of odd degree and a uniformizer class $t$ at $p_0$. For each $p \in \PP^1_k - \{p_0\}$, there is a hyperelliptic curve $H_p \to \PP^1_k$ such that for each $p' \in \PP^1_k - \{p_0, p\}$ and for each hyperelliptic curve $\pi\colon X \to \PP^1_k$ branched at $p_0, p, p'$ (and possibly other points) we have \[
c(H_p \to \PP^1, p')_+ + \deg(p)\deg(p')c_t(X \to \PP^1, p_0)_+ = \ip{\pi^{-1}e_{p_0, p'}}{\pi^{-1}e_{p_0, p}}_X'
\]
(where the $_+$ subscripts indicate that the left hand side is viewed additively as an element of $\F_2$).

Moreover, $H_p \to \PP^1_k$ is \'etale away from $\{p_0, p\}$.
\end{lemma}
The slogan for this lemma is that the pairing $\ip{\cdot}{\cdot}_X'$ evaluated on the two generators of $\Z\langle S\rangle^0/2\Z\langle S\rangle^0$ corresponding to branch points $p, p'$ is independent of the other branch points of $X$. This will let us read the pairing $\ip{\cdot}{\cdot}_{X_B}'$ off from the tuple $B = (p_0, p_1, \dots, p_m, t)$ even as we vary $p_1, \dots, p_m$. In particular, since we chose $c_t(\pi_B, p_0) = 1$, we have \[
\ip{\pi^{-1}e_{p_0, p'}}{\pi^{-1}e_{p_0, p}}_{X_B}' = c(H_p \to \PP^1_k, p')_+.
\]
\begin{proof}
Let $f_p \in k(\PP^1_k)^\times$ such that $\div f_p = e_{p_0, p}$ and such that if $H_p \to \PP^1_k$ is the hyperelliptic curve defined by $f_p$, then $c_t(H_p \to \PP^1, p_0) = 1$. The branch locus of $H_p$ is contained in the support of $e_{p_0, p}$. Using $k(\PP^1_k) \subset k(X)$, we can view $f_p$ as a rational function on $X$.

Now $f_p$ has even valuation at $p'$, so we have $c(H_p \to \PP^1, p') = c(X_{f_p} \to X, \pi^{-1}(p'))$\details{We are just checking whether a particular function is a square in the same residue field on both sides}. We will now compare $c(X_{f_p} \to X, \pi^{-1}(p'))$ to $\ip{\pi^{-1}e_{p_0, p'}}{\pi^{-1}e_{p_0, p}}_X'$. Recall that \[
\ip{\pi^{-1}e_{p_0, p'}}{\pi^{-1}e_{p_0, p}}_X' = c(X_{g_p} \to X, \pi^{-1}(p'))_+
\]
where $X_{g_p}$ is the \'etale 2-cover of $X$ corresponding to $\pi^{-1}e_{p_0, p}$ under Lemma~\ref{lem:kummer}(c) satisfying $c(X_{g_p} \to X, \pi^{-1}(p_0)) = 1$.

If we view $f_p$ as a rational function on $X$, we have $\div f_p \equiv \pi^*e_{p_0, p} = 2\pi^{-1}e_{p_0, p}\pmod{4}$. Thus, under the correspondence of Lemma~\ref{lem:kummer}(c), $X_{f_p} \to X$ and $X_{g_p} \to X$ are isomorphic over $\overline{k}$. Since $X_{g_p}$ is determined within this $\overline{k}$-isomorphism class by its second-order class at $\pi^{-1}(p_0)$, we can take $g_p = f_p\cdot c(X_{f_p} \to X, \pi^{-1}(p_0))$ without changing $X_{g_p}$ so that \[
c(X_{g_p} \to X, \pi^{-1}(p')) = c(H_p \to \PP^1, p')c(X_{f_p} \to X, \pi^{-1}(p_0))^{\deg(p')}
\]
by \eqref{eq:twisting-second-order}. The last step is to compute $c(X_{f_p} \to X, \pi^{-1}(p_0))$.

By Lemma~\ref{lem:kummer}(a), the branched cover $X \to \PP^1_k$ corresponds to a class in $k(\PP^1_k)^\times/(k(\PP^1_k)^\times)^2$, and let $f_0 \in k(\PP^1_k)^\times$ be a lift of this class.

Since $X \to \PP^1_k$ is ramified at $p_0$, we know that $f_0$ has odd valuation at $p_0$, and by construction $f_p$ has valuation congruent to $\deg(p)$ modulo 2 at $p_0$. Thus $f_pf_0^{\deg(p)}$ has even valuation at $p_0$. Write $f_pf_0^{\deg(p)} = t^nu$ for some even integer $n$ and $u \in \widehat{\sheafO}_{\PP^1_k, p_0}^\times$. Now $t$ has valuation 2 at $\pi^{-1}(p_0)$\details{ because $t$ has valuation 1 at $p_0$ and the cover has $e = 2$ above $p_0$}, so since $\sqrt{f_0}$ has odd valuation at $\pi^{-1}(p_0)$ there is an integer $m$ such that $t^m\sqrt{f_0}$ has valuation 1 at $\pi^{-1}(p_0)$, i.e., is a uniformizer. 

We compute the second-order class $c(X_{f_p} \to X, \pi^{-1}(p_0))$ by writing $f_p$ as the product of a power of a uniformizer and a unit in $\widehat{\sheafO}_{X, \pi^{-1}(p_0)}$, up to squares of units. We have $f_p = f_0^{-\deg(p)}t^nu = (t^m\sqrt{f_0})^{-2\deg(p)}t^{n+2m\deg(p)}u$. Since $n$ is even, the right hand side differs from $u$ by a square in $\widehat{\sheafO}_{X, \pi^{-1}(p_0)}$.

Thus, $c(X_{f_p} \to X, \pi^{-1}(p_0)) =\colon \overline{u}$ is the image of $u$ in $k^\times$.

On the other hand, \[
\overline{u} = c_t(H_p \to \PP^1, p_0)c_t(X \to \PP^1, p_0)^{\deg(p)} = c_t(X \to \PP^1, p_0)^{\deg(p)},
\]
which means \[
c(X_{g_p} \to X, \pi^{-1}(p')) = c(H_p \to \PP^1, p')\overline{u}^{\deg(p')} = c(H_p \to \PP^1, p')c_t(X \to \PP^1_k, p_0)^{\deg(p)\deg(p')},
\]
as we wanted.
\end{proof}

\subsection{Symmetry properties of the R\'edei matrix}\label{sect:symmetry}

We use notation from the previous subsection. In this subsection, we use Lemma~\ref{lem:pairing-by-hyperelliptics} to explicitly show that the pairing $\ip{\cdot}{\cdot}_X'$ satisfies a particular symmetry property, which is the same as the symmetry property of the number field R\'edei matrix coming from quadratic reciprocity. Let $k = \F_q$ be a finite field of odd characteristic. To start, we will give a reformulation of quadratic reciprocity over function fields which fits more smoothly with our language.

\begin{lemma}\label{lem:quadratic-reciprocity}
Fix $p_0 \in \PP^1_k$ of odd degree. Let $p, p' \in \PP^1_k - \{p_0\}$ be distinct. Let $t$ be a uniformizer at $p_0$. Let $H \to \PP^1_k$ and $H' \to \PP^1_k$ be hyperelliptic curves whose branch loci are the supports of the divisors $e_{p_0, p}$ mod 2 and $e_{p_0, p'}$ mod 2, respectively. Then  \[
c(H \to \PP^1, p')c(H' \to \PP^1, p) = c_t(H \to \PP^1, p_0)^{\deg(p')}c_t(H' \to \PP^1, p_0)^{\deg(p)}(-1)^{\frac{q - 1}{2}\deg(p)\deg(p')}
\]
(in particular, the right hand side is independent of the choice of $t$).
\end{lemma}
\begin{proof}
Note that since $q \geq 3$, there are at least $q + 1 \geq 4$ degree-1 points in $\PP^1_k$. Let $\infty \in \PP^1_k$ be a degree-1 point other than $p_0, p, p'$. We identify $\PP^1_k - \{\infty\}$ with $\A^1 = \Spec k[x]$.

Let $h_0, h, h' \in k[x]$ be the monic irreducible polynomials cutting out $p_0, p, p'$, respectively. 

Recall that $e_{p_0, p} = \deg(p_0)[p] - \deg(p)[p_0]$ and $e_{p_0, p'} = \deg(p_0)[p'] - \deg(p')[p_0]$. Let $f = ah^{\deg(p_0)}h_0^{-\deg(p)}$ and $f' = a'(h')^{\deg(p_0)}h_0^{-\deg(p')}$ be rational functions cutting out $e_{p_0, p}$ and $e_{p_0, p'}$, respectively. Here we choose coefficients $a, a' \in k^\times$ such that $H \to \PP^1_k$ and $H' \to \PP^1_k$ be the hyperelliptic curves associated to $f$ and $f'$, respectively, by the correspondence of Lemma~\ref{lem:kummer} with base curve $\PP^1_k$. We will also view $h_0$ as a uniformizer at $p_0$.

We have that $c(H \to \PP^1, p')$ is the image of $f$ in the residue field $k(p')$ modulo squares, which is precisely the quadratic residue symbol \[
\legendre{f}{h'} = \legendre{a}{h'}\legendre{h}{h'}^{\deg(p_0)}\legendre{h_0}{h'}^{-\deg(p)}.
\]
We can make a few observations here. First, $a$ is a square in the residue field $k(p')$ precisely when either $a$ is already a square in $k$ or $k(p')$ is an even-degree extension of $k$. So, $\legendre{a}{h'} \equiv a^{\deg(p')} \pmod{(k^\times)^2}$. Moreover, since $\deg(p_0)$ is odd, we get \[
c(H \to \PP^1, p') = \legendre{f}{h'} = a^{\deg(p')}\legendre{h}{h'}\legendre{h_0}{h'}^{\deg(p)}.
\]
Similarly, $c(H' \to \PP^1, p) = (a')^{\deg(p)}\legendre{h'}{h}\legendre{h_0}{h}^{\deg(p')}$. We also observe that $c_{h_0}(H \to \PP^1, p_0)$ is the image of $ah^{\deg(p_0)}$ in the residue field of $p_0$ modulo squares, which is just the quadratic residue symbol \[
\legendre{ah^{\deg(p_0)}}{h_0} = \legendre{a}{h_0}\legendre{h}{h_0} = a^{\deg(p_0)}\legendre{h}{h_0} = a\legendre{h}{h_0}
\]
because $\deg(p_0)$ is odd. Quadratic reciprocity for polynomials \cite[Theorem 3.3]{rosen} shows \[
\legendre{h}{h_0}\legendre{h_0}{h} = (-1)^{\frac{q - 1}{2}\deg(p)\deg(p_0)}
\]
so that \[
a = c_{h_0}(H \to \PP^1, p_0)\legendre{h}{h_0}^{-1} = c_{h_0}(H \to \PP^1, p_0)\legendre{h_0}{h}(-1)^{\frac{q - 1}{2}\deg(p)\deg(p_0)}
\]
because $\deg(p_0)$ is odd. Thus, \[
c(H \to \PP^1, p') = c_{h_0}(H \to \PP^1, p_0)^{\deg(p')}\legendre{h}{h'}\legendre{h_0}{h'}^{\deg(p)}\legendre{h_0}{h}^{\deg(p')}(-1)^{\frac{q - 1}{2}\deg(p)\deg(p')\deg(p_0)}.
\]
The same computation shows \[
c(H' \to \PP^1, p) = c_{h_0}(H' \to \PP^1, p_0)^{\deg(p)}\legendre{h'}{h}\legendre{h_0}{h}^{\deg(p')}\legendre{h_0}{h'}^{\deg(p)}(-1)^{\frac{q - 1}{2}\deg(p)\deg(p')\deg(p_0)}.
\]
Now quadratic reciprocity gives \[
\legendre{h'}{h}\legendre{h}{h'} = (-1)^{\frac{q - 1}{2}\deg(p)\deg(p')}
\]
so that \[
c(H \to \PP^1, p')c(H' \to \PP^1, p) = c_{h_0}(H \to \PP^1, p_0)^{\deg(p')}c_{h_0}(H' \to \PP^1, p_0)^{\deg(p)}(-1)^{\frac{q - 1}{2}\deg(p)\deg(p')}
\]
Now we examine what happens when replacing $h_0$ with another uniformizer $t$ at $p_0$. If $p$ has even degree, then $H \to \PP^1$ is unramified at $p_0$, so $c_t(H \to \PP^1, p_0)$ is independent of the choice of uniformizer $t$. At the same time, $c_t(H' \to \PP^1, p_0)^{\deg(p)}$ is trivial, so also independent of the choice of uniformizer. Similarly, if $p'$ has even degree, then the right hand side is also independent of the choice of uniformizer.

Finally, if both $p$ and $p'$ have odd degree, then replacing $h_0$ by a different uniformizer class will change $c_{h_0}(H \to \PP^1, p_0)^{\deg(p')}$ and $c_{h_0}(H' \to \PP^1, p_0)^{\deg(p)}$ both by the nontrivial element of $k^\times/(k^\times)^2$, so the class of the right hand side will not change.
\end{proof}

Applying this lemma gives a symmetry property for the pairing $\ip{\cdot}{\cdot}_X'$:

\begin{corollary}\label{cor:pairing-symmetry}
Fix $p_0 \in \PP^1_k$ of odd degree. Let $\pi\colon X \to \PP^1_k$ be a hyperelliptic curve branched at $p_0, p, p'$ (and possibly other points) such that $\deg(p), \deg(p') \geq \deg(p_0)$. Then \[
\ip{\pi^{-1}e_{p_0, p'}}{\pi^{-1}e_{p_0, p}}_X' + \ip{\pi^{-1}e_{p_0, p}}{\pi^{-1}e_{p_0, p'}}_X' = \begin{cases}
    0 &\text{if }q \equiv 1\pmod{4} \\
    \deg(p)\deg(p') &\text{if } q\equiv 3\pmod{4}
\end{cases}
\]
\end{corollary}
\begin{proof}
Let $h_0$ be the uniformizer at $p_0$ coming from Lemma~\ref{lem:quadratic-reciprocity}.
To compute $\ip{\pi^{-1}e_{p_0, p'}}{\pi^{-1}e_{p_0, p}}_X'$ and $\ip{\pi^{-1}e_{p_0, p}}{\pi^{-1}e_{p_0, p'}}_X'$, we use Lemma~\ref{lem:pairing-by-hyperelliptics}. This lemma produces hyperelliptic curves $H_p$ and $H_{p'}$ such that $c_{h_0}(H_p \to \PP^1_k, p_0) = c_{h_0}(H_{p'} \to \PP^1_k, p_0) = 1 \in k^\times/(k^\times)^2$.  We get \[
c(H_p \to \PP^1_k, p')_+ + \deg(p)\deg(p')c_{h_0}(X \to \PP^1_k, p_0)_+ = \ip{\pi^{-1}e_{p_0, p'}}{\pi^{-1}e_{p_0, p}}'_X
\]
and \[
c(H_{p'} \to \PP^1_k, p)_+ + \deg(p)\deg(p')c_{h_0}(X \to \PP^1_k, p_0)_+ = \ip{\pi^{-1}e_{p_0, p}}{\pi^{-1}e_{p_0, p'}}'_X.
\]
Adding these equations together and using Lemma~\ref{lem:quadratic-reciprocity} gives \[
\ip{\pi^{-1}e_{p_0, p'}}{\pi^{-1}e_{p_0, p}}'_X + \ip{\pi^{-1}e_{p_0, p}}{\pi^{-1}e_{p_0, p'}}'_X = \frac{q - 1}{2}\deg(p)\deg(p')
\]
which is the desired result.
\end{proof}

Let $S$ be the set of branch points of $X \to \PP^1_k$. We recall that the pairing $\ip{\cdot}{\cdot}'_X$ is lifted from $\Z\langle S\rangle^0/2\Z\langle S\rangle^0$ along a two-to-one map with kernel $\sum_{p' \in S - \{p_0\}} \pi^{-1}e_{p_0, p'}$ ( see Lemma~\ref{lem:phi-kernel}). Thus, we have \[
\sum_{p' \in S - \{p_0\}} \ip{\pi^{-1}e_{p_0, p'}}{\pi^{-1}e_{p_0, p}}'_X = \ip{\sum_{p' \in S - \{p_0\}} \pi^{-1}e_{p_0, p'}}{\pi^{-1}e_{p_0, p}}'_X = 0
\]
and \[
\sum_{p \in S - \{p_0\}} \ip{\pi^{-1}e_{p_0, p'}}{\pi^{-1}e_{p_0, p}}'_X = \ip{ \pi^{-1}e_{p_0, p'}}{\sum_{p \in S - \{p_0\}}\pi^{-1}e_{p_0, p}}'_X = 0.
\]

Thus, ordering the basis $\{\pi^{-1}e_{p_0, p} \mid p \in S - \{p_0\}\}$ for the preimage of $\Pic^0(X)(k)[2]$ in $\Z\langle S\rangle^0/2\Z\langle S\rangle^0$ by even degrees, then odd degrees, there are two cases for the pairing depending on the residue of $q$ mod $4$:
\begin{itemize}
    \item If $q \equiv 1\pmod{4}$, the matrix of the pairing is symmetric with row (and column) sums 0.
    \item If $q \equiv 3\pmod{4}$, the matrix of the pairing can be written in block form: \[
    \begin{pmatrix}
    A_1 & B \\
    B^\intercal & A_2
    \end{pmatrix}
    \] 
    where $A_1$ is symmetric, $A_2 - A_2^\intercal = J-I$ (and $J$ denotes the matrix with $1$ in every entry and $I$ the identity matrix), and the rows and columns once again sum to 0.
\end{itemize}

In the next section, we will systematically study such matrices and determine their rank distributions.

\section{Theory of Random $C$-symmetric Matrices over Finite Fields}\label{sect:random-matrices}

In this section, we study random matrices satisfying the same symmetry constraints as the R\'edei matrix, which are described in Subsection~\ref{sect:symmetry}. We determine the limiting corank distribution for most such matrices and give effective bounds on the rate of convergence. A similar result about corank distributions of matrices satisfying certain symmetry ``rules'' was obtained by Koymans and Pagano in \cite{Koymans2020EffectiveCO}. Compared to their result, ours treats a less general situation with the benefit of exposing some more linear algebraic structure.

In this section, let $\F_\ell$ be a finite field of order $\ell$, any prime power. Also, in this section, let $C$ be an alternating bilinear form on $\F_\ell^n$, i.e., a skew-symmetric $n\times n$ matrix with zeroes along the diagonal.

\begin{definition}
An $n\times n$ matrix $M$ over $\F_\ell$ is called \textit{$C$-symmetric} if $M - M^\intercal = C$.

If $M - M^\intercal$ has rank $c$ for some even $c \geq 0$, we say $M$ is \textit{$c$-symmetric}.
\end{definition}

Easily, being $O$-symmetric (where $O$ is the all-$0$ matrix) is the same as being symmetric.

The matrix of the pairing is over $\F_2$, and with a suitable ordering of the basis, it is:
\begin{itemize}
    \item If $q \equiv 1\pmod{4}$, $O$-symmetric with row and column sums 0.
    \item If $q \equiv 3\pmod{4}$, $C$-symmetric where \[
    C=\begin{pmatrix}
    O & O \\
    O & J-I
    \end{pmatrix}
    \] 
     while the rows and columns still sum to 0.
\end{itemize}

Later we will show that the matrix of the pairing we obtain in certain families is close to being uniformly distributed over $C$-symmetric matrices whose rows and columns sum to $0$ (for appropriate $C$). These matrices preserve their rank when their last row and column are deleted and become a uniformly distributed $C$-without-last-row-and-column-symmetric matrix; this motivates us to study the rank distribution of a uniformly random $C$-symmetric matrix. We summarize this with the following lemma:
\begin{lemma}\label{lem:remove-column}
    Let $\ell$ be a prime power and let $C$ be an alternating matrix on $\F_\ell^n$ whose rows and columns sum to 0. Let $M_n$ be uniformly distributed among $C$-symmetric $n\times n$ matrices whose rows and columns sum to 0. Let $T_n$ and $C'$ be the matrix obtained by deleting the last row and column of $M_n$ and $C$ respectively. Then $T_n$ is uniformly distributed among $C'$-symmetric $(n-1)\times (n-1)$ matrices, and \[\dim \ker M_n = \dim \ker T_n +1.\]
\end{lemma}

The remainder of this section then studies the rank distribution of a uniformly distributed $C$-symmetric matrix for $C=O$ (symmetric) and more general $C$.

\subsection{Symmetric matrices}\label{sect:symmetric-random-matrices}

The nullity of a random $n\times n$ symmetric matrix over $\F_2$ was computed in \cite[Theorem 2]{MacWilliams01021969}. The same computation works over any finite field.

We define \[
\mu_{S, \ell}(r) \coloneqq \frac{|\wedge^2 \F_\ell^r|}{|\GL_r(\F_\ell)|}\prod_{k=0}^\infty (1 - \ell^{-2k - 1}).
\]

\begin{theorem}[See {\cite[Theorem 2]{MacWilliams01021969}}]
Let $M_n$ be uniformly distributed among symmetric $n\times n$ matrices in $\F_\ell$. Then \[
\PP[\dim\ker M_n = r] = \ell^{\binom{n - r + 1}{2} - \binom{n + 1}{2}}\binom{n}{r}_\ell\prod_{k=0}^{\lfloor\frac{n - r - 1}{2}\rfloor}(1 - \ell^{-2k - 1}).
\]
where $\binom{n}{r}_\ell$ is the number of $r$-dimensional subspaces of $\F_\ell^n$.
\end{theorem}

After some computation (see, e.g., \cite[Theorem 4.1]{fulmangoldstein2015ranks} for a much stronger result) one finds:

\begin{corollary}\label{cor:effective-symmetric-dist}
Let $M_n$ be uniformly distributed among symmetric $n\times n$ matrices over $\F_\ell$. Then when $0 \leq r < n$, we have \[
\left|\PP[\dim \ker M_n = r] - \mu_{S, \ell}(r)\right| = O_{\ell, r}(\ell^{-n}).
\]
\end{corollary}

\subsection{General $C$-symmetric matrices}\label{sect:c-symmetric-random-matrices}

Let $M$ be a $C$-symmetric matrix. If $B \in \GL_n(\F_\ell)$, then $B^\intercal M B$ is $B^\intercal C B$-symmetric. Moreover, $\ker B^\intercal M B = B^{-1}\ker M$. If $M$ is a uniformly random $C$-symmetric matrix then $B^\intercal M B$ is a uniformly random $B^\intercal CB$-symmetric matrix.

The group $\GL_n(\F_\ell)$ acts on the set of alternating bilinear forms of any fixed rank on $\F_\ell^n$ by change of basis, i.e., $B\cdot C \coloneqq B^\intercal C B$, transitively. Thus, if $B$ is a uniformly random automorphism of $\F_\ell^n$ and $M$ is a uniformly random $C$-symmetric matrix for some fixed $C$, then $B^\intercal M B$ is a uniformly random $c$-symmetric matrix for $c = \rank C$. Since $B^\intercal M B$ has the same rank as $M$, the rank distribution of a uniformly random $c$-symmetric matrix is the same as the rank distribution of a uniformly random $C$-symmetric matrix whenever $c = \rank C$.

The goal of this section is to determine this rank distribution. We define \[
\mu_{CL, \ell}(r) \coloneqq \frac{1}{|\GL_r(\F_\ell)|}\prod_{i=r+1}^\infty (1 - \ell^{-i}),
\]
the limiting corank distribution of large uniformly random square matrices over $\F_\ell$ (see, e.g., \cite[Equation (2)]{fulmangoldstein2015ranks}).

We first turn our attention to the probability that a random subspace $V$ of dimension $r$ lies in $\ker M$, which then determines the rank distribution through Möbius inversion:

\begin{lemma}\label{beforetotisotropic}
Let $V \subseteq \F_\ell^n$ with $\dim V = r$. Let $C$ be a fixed alternating form on $\F_\ell^n$. Let $M$ be a uniformly random $C$-symmetric matrix. Then \[
\PP[V \subseteq \ker M] = \begin{cases}
    \ell^{-nr + \frac{r(r - 1)}{2}} &\text{if }V\text{ is totally isotropic for }C; \\
    0 &\text{otherwise}.
\end{cases}
\]
\end{lemma}
Here we say $V$ is \textit{totally isotropic} for $C$ if the restriction of the alternating form $C$ to $V$ is zero, i.e., $v^\intercal C w = 0$ for all $v, w \in V$.
\begin{proof}
Let $e_1, \dots, e_n$ be the standard basis for $\F_\ell^n$. Let $B \in \GL_n(\F_\ell)$ send $\vspan\{e_1, \dots, e_r\}$ to $V$. Then $B^\intercal M B$ is a uniformly random $B^\intercal C B$-symmetric matrix, and \[
\PP[V \subseteq \ker M] = \PP[\vspan\{e_1, \dots, e_r\} \subseteq \ker B^\intercal M B].
\]
This is the probability that the first $r$ columns of $B^\intercal M B$ are zero. If the upper-left $r\times r$ corner of $B^\intercal C B$ is zero, then this probability is $\ell^{-nr + \frac{r(r-1)}{2}}$. Otherwise, the first $r$ columns of $B^\intercal M B$ are never zero.

The condition that the upper-left $r\times r$ corner of $B^\intercal C B$ is zero is equivalent to the condition that the bilinear form $B^\intercal C B$ vanishes on $\vspan\{e_1, \dots, e_r\}$, or equivalently that the bilinear form $C$ vanishes on $V$. 
\end{proof}
Therefore, we focus on the probability that a random $r$-dimensional subspace is totally isotropic for a fixed rank-$c$ alternating form $C$. The main idea is that any alternating form on $\F_\ell^n$ descends to a nondegenerate alternating form after quotienting by the radical of $C$. The limiting probability that an $r$-dimensional subspace is totally isotropic for a nondegenerate form in a high-dimensional space is known, so the problem reduces to understanding the image of a random $r$-dimensional subspace under the quotient map. It turns out that this image is almost always a uniformly random $r$-dimensional subspace of the quotient. To see this, we show that a random $r$-dimensional subspace almost always intersects the radical of $C$ trivially if $C$ has high enough rank:

\begin{lemma}\label{lem:intersection-prob}
Let $K$ be a fixed subspace of $\F_\ell^n$ of dimension $k$ and let $V$ be a uniformly random subspace of $\F_\ell^n$ of dimension $r$. Then for $d > 0$, we have \[
\PP[\dim(V \cap K) \geq d] \leq \frac{\ell^{k + r - n}}{\ell^d - 1}.
\]
\end{lemma}
\begin{proof}
We observe that, for fixed $V'$ of dimension $r$, and for $v$ chosen uniformly at random in $V' - \{0\}$, we have \[
\PP[v \in K] = \frac{\#(V' \cap K) - 1}{\#V' - 1} \geq \frac{\#(V' \cap K) - 1}{\#V'}.
\]
Now consider the random variable \[
\PP[v \in K \mid V]
\]
where $v$ is chosen uniformly at random in $V - \{0\}$. Note that since $V$ is uniformly random, the unconditional distribution of $v$ is uniformly random in $\F_\ell^n - \{0\}$. 

By the law of total expectation, we have \[
\E[\PP[v \in K \mid V]] = \PP[v \in K] = \frac{\#K - 1}{\ell^n - 1} \leq \ell^{k - n}.
\]

Now by Markov's inequality, we have \begin{align*}
    \PP[\dim(V \cap K) \geq d] &= \PP\left[\frac{\#(V \cap K) - 1}{\#V} \geq \frac{\ell^d - 1}{\#V}\right] \\
    &\leq \PP\left[\PP[v \in K \mid V] \geq \frac{\ell^d - 1}{\ell^r}\right] \\
    &\leq \frac{\E[\PP[v \in K \mid V]]\ell^r}{\ell^d - 1} \\
    &= \frac{\ell^{k + r - n}}{\ell^d - 1}
\end{align*}
as we wanted.
\end{proof}
The following result gives two-sided bounds for the probability that a random subspace of fixed dimension $r$ is totally isotropic for an alternating form of rank $c$ when $c$ is large compared to $r$.
\begin{proposition}\label{bound-isotropic-probability}
Let $C$ be an alternating form on $\F_\ell^n$ of rank $c$ and let $V$ be a random subspace of dimension $r \leq c/2$. Then \[
1 - 2\ell^{2r - 1 - n} \leq \frac{\PP[V\text{ is totally isotropic}]}{\ell^{-\binom{r}{2}}} \leq 1 + 2\ell^{\binom{r + 1}{2} - 1 - c}.
\]
\end{proposition}
\begin{proof}
The cases $r = 0, 1$ are immediate.

We establish an easier lower bound first. The number of totally isotropic spaces is bounded below by  \[
\frac{\prod_{i=0}^{r-1} (\ell^{n - i} - \ell^i)}{|\GL_r(\F_\ell)|},
\]
where ${n-i}$ bounds the dimension of the space orthogonal to the first $i$ vectors from below. Divided by the total number of subspaces of dimension $r$ we get the lower bound
\[
\prod_{i=0}^{r-1}\frac{\ell^{n-i} - \ell^i}{\ell^n - \ell^i}  = \prod_{i=0}^{r-1} \frac{\ell^{-i}(1 - \ell^{2i - n})}{1 - \ell^{i-n}} = \ell^{-r(r-1)/2} \prod_{i=0}^{r-1}\frac{1 - \ell^{2i - n}}{1 - \ell^{i - n}}.
\]

As for the upper bound, let $K$ be the radical of $C$ and note that
\begin{align*}
    \PP\left[V\text{ is totally isotropic}\right] &= \sum_{d = 0}^r \PP[V\text{ is totally isotropic} \mid \dim(V \cap K) = d] \cdot \PP[\dim(V \cap K) = d] \\
    &\leq \PP[V\text{ is totally isotropic} \mid V \cap K = 0] + \PP[\dim(V \cap K) \geq 1]
\end{align*}
where the first term is equal to the probability a random $r$-dimensional space is isotropic for a non-degenerate alternating form on the $c$-dimensional quotient space $\F_{\ell}^n/K$:
\[
\frac{\prod_{i=0}^{r-1}(\ell^{c - i} - \ell^i)}{\prod_{i=0}^{r-1}(\ell^c - \ell^i)} = \prod_{i=0}^{r-1}\frac{\ell^{-i}(1 - \ell^{2i - c})}{1 - \ell^{i - c}} = \ell^{-\binom{r}{2}}\prod_{i=0}^{r-1}\frac{1 - \ell^{2i - c}}{1 - \ell^{i - c}}
\]
and the second term is bounded by Lemma \ref{lem:intersection-prob}. Thus, we have \[
\ell^{-\binom{r}{2}}\prod_{i=0}^{r-1}\frac{1 - \ell^{2i - n}}{1 - \ell^{i - n}} \leq \PP[V\text{ is totally isotropic}] \leq \ell^{-\binom{r}{2}}\prod_{i=0}^{r-1}\frac{1 - \ell^{2i - c}}{1 - \ell^{i - c}} + \frac{\ell^{r - c}}{\ell - 1}.
\]
The inequality in the claim follows from combining the above with the following lemma.
\end{proof}
\begin{lemma}\label{lem:product-error}
When $n \geq 2r$ we have \[
\left|\prod_{i=0}^{r-1} \frac{1 - \ell^{2i - n}}{1 - \ell^{i-n}} - 1\right| \leq 2\ell^{2r - 1 - n}.
\]
\end{lemma}
\begin{proof}
The case $r = 1$ is trivial. We have \begin{align*}
\left|\log\prod_{i=0}^{r-1} \frac{1 - \ell^{2i - n}}{1 - \ell^{i-n}} \right| &= \left|\sum_{i=0}^{r-1}\log(1 - \ell^{2i - n}) - \sum_{i=0}^{r-1}\log(1 - \ell^{i - n})\right| \\
&\leq \sum_{i=0}^{r-1}|\log(1 - \ell^{2i - n})| + |\log(1 - \ell^{r - n})|.
\end{align*}
The last inequality comes from cancellation. We get a sum over all odd $i$ no more than $r - 1$ and all even $i$ greater than $r - 1$, and we bound the odd-$i$ terms by even-$i$ terms. We end up with one extra term corresponding to the smallest even number that is at least $r - 1$, and we bound this above by $r$.

Now we know that when $x \in [0, 1/2]$, we get $-x \geq \log(1 - x) \geq -x - x^2 \geq -3x/2$. So, when $n \geq 2r$ we have \[
-3\ell^{2i - n}/2 \leq \log(1 - \ell^{2i - n}) \leq -\ell^{2i - n}
\]
for all $0 \leq i \leq r - 1$, and therefore \[
-\ell^{-n}\frac{3}{2}\cdot\frac{\ell^{2r} - 1}{\ell^2 - 1}\leq \sum_{i=0}^{r-1} \log(1 - \ell^{2i - n}) \leq -\ell^{-n}\frac{\ell^{2r} - 1}{\ell^2 - 1}.
\] 
Thus, we have \begin{align*}
\left|\log\prod_{i=0}^{r-1} \frac{1 - \ell^{2i - n}}{1 - \ell^{i-n}} \right| &\leq \frac{3}{2}\left(\ell^{-n} \frac{\ell^{2r} - 1}{\ell^2 - 1} + \ell^{r - n}\right) \\
&\leq 2\ell^{2r - 1 - n}.
\end{align*}
If $x$ is the term we want to bound in this lemma, we have shown $-x \geq \log(1 - x) \geq -2\ell^{2r - 1 - n}$, which implies $x \leq 2\ell^{2r - 1 - n}$.
\end{proof}

We will also need a one-sided upper bound estimate, which will be used to control the tail of a sum:

\begin{proposition}\label{prop:bound-isotropic-probability-tail}
Let $C$ be an alternating form on $\F_\ell^n$ of rank $c$ and let $V$ be a random subspace of dimension $r$. Then if $c < 2r$, we have \[
\PP[V\text{ is totally isotropic}] \leq \frac{\ell^{-c/2}}{1 - \ell^{-1}}.
\]
\end{proposition}
\begin{proof}
As in the proof of Proposition~\ref{bound-isotropic-probability}, let $K$ be the radical of $C$. The idea is that $\F_\ell^n/K$ is equipped with a nondegenerate symplectic form induced by $C$, such that if $V$ is totally isotropic for $C$, then $V/(V \cap K)$ is totally isotropic for the induced symplectic form on $\F_\ell^n/K$. In particular, if $V$ is totally isotropic for $C$, then \[
\dim V/(V \cap K) \leq \frac{1}{2}\dim \F_\ell^n/K = \frac{c}{2}.
\]
Thus if $V$ is totally isotropic, we have $\dim (V \cap K) \geq r - c/2$.

By Lemma~\ref{lem:intersection-prob}, we have \[
\PP[V\text{ is totally isotropic}] \leq \PP[\dim(V \cap K) \geq r - c/2] \leq \frac{\ell^{r - c}}{\ell^{r - c/2} - 1} \leq \frac{\ell^{-c/2}}{1 - \ell^{-1}}
\]
\details{because $\ell^{r - c/2} - 1 = \ell^{r - c/2}(1 - \ell^{c/2 - r}) \geq \ell^{r - c/2}(1 - \ell^{-1})$}provided $r \geq c/2 + 1$.
\end{proof}

\begin{remark}
    The bounds we obtained in Propositions~\ref{bound-isotropic-probability} and \ref{prop:bound-isotropic-probability-tail} are of the probability that a uniformly random subspace is isotropic for a fixed alternating form. We could write it, for a uniformly random element $B\in \GL_{n}(\F_\ell)$ and fixed alternating form $C$ and subspace $V$, as
    \[ \PP[BV \text{ is isotropic for }C]=\PP[V \text{ is isotropic for }B^{\intercal}CB]\]
    and hence Propositions~\ref{bound-isotropic-probability} and \ref{prop:bound-isotropic-probability-tail} are also bounds on the probability that a fixed subspace $V$ is isotropic for a random rank-$c$ alternating form $C$. Combined with Lemma~\ref{beforetotisotropic}, we obtain:
\end{remark}
\begin{corollary}\label{cor:moment-bound}
    Let $C$ be a uniformly random alternating form on $\F_\ell^n$ of rank $c$ and let $V$ be a fixed subspace of dimension $r$. When $c \geq 2r$, we have \[
1 - 2\ell^{2r - 1 - n} \leq \frac{\PP[V\text{ is totally isotropic for }C]}{\ell^{-\binom{r}{2}}} \leq 1 + 2\ell^{\binom{r + 1}{2} - 1 - c} \]
and when $c < 2r$, we have \[
\PP[V\text{ is totally isotropic for }C] \leq \frac{\ell^{-c/2}}{1 - \ell^{-1}}.
\]
If $M$ is a random $c$-symmetric $n\times n$ matrix and $c \geq 2r$, we have \begin{equation}\label{eq:isotropy-most}
1 - 2\ell^{2r - 1 - n} \leq \frac{\PP[\ker M \supseteq V]}{\ell^{-nr}} \leq 1 + 2\ell^{\binom{r + 1}{2} - 1 - c}
\end{equation}
and if $c < 2r$, we have \begin{equation}\label{eq:isotropy-tail}
\PP[\ker M \supseteq V] \leq \frac{\ell^{-nr - c/2 + \binom{r}{2}}}{1 - \ell^{-1}}.
\end{equation}
\end{corollary}
We now assume $M$ is a uniformly random $c$-symmetric matrix. To turn this into an error bound for the full rank distribution, we quote a version of M\"obius inversion for posets:

\begin{theorem}[See {\cite[Proposition 3.7.2]{stanleyvol1}}]
Let $(P, \leq)$ be a finite poset. Define \[
\mu(s, s) = 1 \text{ for }s \in P, \qquad\qquad \mu(s, u) = -\sum_{s \leq t < u} \mu(s, t), \qquad \text{ for }s < u\text { in }P
\]
If $f, g \colon P \to \R$, then \[
g(s) = \sum_{t \geq s} f(t)
\]
if and only if \[
f(s) = \sum_{t \geq s} g(t)\mu(s, t).
\]
\end{theorem}

We take $P$ to be the set of subspaces of $\F_\ell^n$, ordered by inclusion. The function $f$ is given by $f(V) = \PP[\ker M = V]$, and $g$ is given by $g(V) = \PP[\ker M \supseteq V]$.

M\"obius inversion implies
\[
\PP[\ker M = V] = \sum_{W \supseteq V} \mu(V, W)\PP[\ker M \supseteq W]
\]
and therefore \[
\PP[\dim \ker M = r] = \binom{n}{r}_\ell\sum_{W \supseteq V} \mu(V, W)\PP[\ker M \supseteq W]
\]

where
\[\binom{n}{r}_\ell=\prod_{i=0}^{r-1}\frac{\ell^n-\ell^{i}}{\ell^r-\ell^{i}}=\ell^{rn}\prod_{i=0}^{r-1}({1-\ell^{i-n}})/|\GL_r(\F_\ell)|, \]
the number of $r$-dimensional subspaces of $\F_\ell^n$. We are interested in the case where $r$ is fixed and $c$ grows, so throughout the rest of this section, assume $c > 2r$. 

By \cite[Example 3.10.2]{stanleyvol1}, when $V \subseteq W$ we have \[
\mu(V, W) = (-1)^{\dim W - \dim V}\ell^{\binom{\dim W - \dim V}{2}}.
\]
The number of subspaces of dimension $r + d$ containing a fixed subspace $V$ is the same as the number of subspaces of $\F_\ell^n/V$ of dimension $d$, so that \[
\PP[\dim \ker M = r] = \binom{n}{r}_\ell\sum_{d=0}^{n-r}\binom{n - r}{d}_\ell (-1)^d \ell^{\binom{d}{2}}\PP[\ker M \supseteq \F_\ell^{r + d}].
\]
Instead of directly calculating this probability, we compare it to the probability that an i.i.d. uniform $n\times n$ matrix has an $r$-dimensional kernel, which we denote by $p_r(n)$. Since the probability that an i.i.d. uniform $n\times n$ matrix vanishes on an $r$-dimensional subspace $V$ is $\ell^{-rn}$, we have
\[
p_r(n) = \binom{n}{r}_\ell\sum_{d=0}^{n-r}\binom{n - r}{d}_\ell (-1)^d \ell^{\binom{d}{2}}\ell^{-(r+d)n}.
\]
Then \begin{equation}\label{eq:sum-expression}
\PP[\dim \ker M = r] - p_r(n) =  \binom{n}{r}_\ell\sum_{d=0}^{n-r}\binom{n - r}{d}_\ell (-1)^d \ell^{\binom{d}{2}}\left(\PP[\ker M \supseteq \F_\ell^{r+d}] - \ell^{-(r + d)n}\right).
\end{equation}
To control the sum, we split it into two sums at $r + d = c/2$. When $r + d \leq c/2$, we will use the more precise estimate in Corollary~\ref{cor:moment-bound}, and when $r + d > c/2$, we will use the less precise upper bound.

We bound the first sum by taking logarithm of each term in the product:
\begin{align}\label{eq:log-inequalities-1}
    \log_\ell\binom{n}{r}_\ell  &\leq rn - \log_\ell|\GL_r(\F_\ell)|\\ 
    \label{eq:log-inequalities-2}\log_\ell\binom{n - r}{d}_\ell & \leq d(n - r) - \log_\ell|\GL_d(\F_\ell)|\\
    \log_\ell|\PP[\ker M \supseteq \F_\ell^{r+d}] - \ell^{-(r+d)n}| &\leq -(r+d)n + \max\left\{2r + 2d - n, \binom{r + d + 1}{2} - c\right\}
\end{align}
when $r + d \leq c/2$ by Corollary \ref{cor:moment-bound}.
When $r+d\geq 3$, 
\begin{align*}
     \log_\ell\left|\binom{n}{r}_\ell\binom{n - r}{d}_\ell \ell^{\binom{d}{2}}\left(\PP[\ker M \supseteq \F_\ell^{r+d}] - \ell^{-(r + d)n}\right)\right| 
     &\leq -c + d^2 - \log_\ell|\GL_d(\F_\ell)| + \frac{1}{2}(r^2 + r)
     - \log_\ell|\GL_r(\F_\ell)|. 
\end{align*}
By considering the proportion of invertible matrices among all $d\times d$ matrices, we bound $d^2 - \log_\ell|\GL_d(\F_\ell)|$ by an absolute constant dependent only on $\ell$:
\begin{equation}\label{eq:gld-d^2}
d^2 - \log_\ell|\GL_d(\F_\ell)| = -\sum_{i=0}^{d-1}\log_\ell(1 - \ell^{i-d}) = -\frac{1}{\log \ell}\sum_{i=0}^{d - 1}\log(1 - \ell^{i - d}),
\end{equation}
so $d^2 - \log_\ell|\GL_d(\F_\ell)| = O(1)$. Since $-\log(1 - x) \leq 3x/2$ when $x \in [0, 1/2]$, we have \begin{equation}\label{eq:sum-of-logs}
-\sum_{i=0}^{d-1}\log(1 - \ell^{i-d}) \leq \frac{3}{2}\sum_{i=0}^{d-1}\ell^{i-d} = \frac{3}{2}\ell^{-d} \frac{\ell^d - 1}{\ell - 1} = \frac{3}{2(\ell - 1)}(1 - \ell^{-d}) \leq \frac{3}{2(\ell - 1)}.
\end{equation}
Therefore, when $r + d \geq 3$, we have \begin{align*}
     \log_\ell\left|\binom{n}{r}_\ell\binom{n - r}{d}_\ell \ell^{\binom{d}{2}}\left(\PP[\ker M \supseteq \F_\ell^{r+d}] - \ell^{-(r + d)n}\right)\right| 
     &\leq -c + O_r(1).
\end{align*}
When $r + d < 3$, we get an even better bound \begin{align*}
 \log_\ell\left|\binom{n}{r}_\ell\binom{n - r}{d}_\ell \ell^{\binom{d}{2}}\left(\PP[\ker M \supseteq \F_\ell^{r+d}] - \ell^{-(r + d)n}\right)\right| \leq -n + O_{\ell, r}(1) \leq -c + O_{\ell, r}(1).
\end{align*}
So we have, for all $d$ such that $r + d \leq c/2$, that \[
\left|\binom{n}{r}_\ell\binom{n - r}{d}_\ell \ell^{\binom{d}{2}}\left(\PP[\ker M \supseteq \F_\ell^{r+d}] - \ell^{-(r + d)n}\right)\right| \leq \ell^{-c + O_{\ell, r}(1)} = O_{\ell, r}(1)\ell^{-c}
\]
giving us \begin{equation}\label{eq:main-terms}
\binom{n}{r}_\ell\sum_{d=0}^{c/2 - r}\binom{n - r}{d}_\ell (-1)^d \ell^{\binom{d}{2}}\left(\PP[\ker M \supseteq \F_\ell^{r+d}] - \ell^{-(r + d)n}\right) = O_{\ell, r}(c\ell^{-c}).
\end{equation}
Now we want to control the tail of the sum starting at $d = c/2 + 1 - r$. To do this, we will separate the $\PP[\ker M \supseteq \F_\ell^{r + d}]$ terms and the $\ell^{-(r + d)n}$ terms and show that each of the resulting sums is $O_{\ell, r}(\ell^{-c})$.

We start with the $\PP[\ker M \supseteq \F_\ell^{r + d}]$ terms. By Corollary~\ref{cor:moment-bound}, we have \[
\PP[\ker M \supseteq \F_\ell^{r + d}] \leq \frac{\ell^{-n(r+d) - c/2 + \binom{r + d}{2}}}{1 - \ell^{-1}}
\]
whenever $r + d > c/2$. Using \eqref{eq:log-inequalities-1} and \eqref{eq:log-inequalities-2} we have \begin{align*}
\log_\ell&\left|\binom{n}{r}_\ell\binom{n - r}{d}_\ell \ell^{\binom{d}{2}}\PP[\ker M \supseteq \F_\ell^{r+d}]\right| \\
&\leq rn - \log_\ell|\GL_r(\F_\ell)| + d(n - r) - \log_\ell|\GL_d(\F_\ell)| + \binom{d}{2} - n(r + d) - c/2 + \binom{r + d}{2} - \log_\ell(1 - \ell^{-1}) \\
&= - \log_\ell|\GL_r(\F_\ell)| - rd - \log_\ell|\GL_d(\F_\ell)| + \binom{d}{2} - c/2 + \binom{r + d}{2} - \log_\ell(1 - \ell^{-1}) \\
&= -r^2 - rd - d^2 + \binom{d}{2} - c/2 + \binom{r + d}{2} + O_\ell(1) \\
&= -\binom{r + 1}{2} - d - c/2 + O_\ell(1) \\
&= -d - c/2 + O_{\ell, r}(1)
\end{align*}
where the fourth line uses \eqref{eq:gld-d^2} and \eqref{eq:sum-of-logs} to control $\log_\ell|\GL_r(\F_\ell)|$ and $\log_\ell |\GL_d(\F_\ell)|$. Now when $r + d > c/2$ we have $d \geq c/2 + O_{\ell, r}(1)$, so the right hand side above is bounded above by $-c + O_{\ell, r}(1)$. It follows that \begin{equation}\label{eq:tail-csymm}
\binom{n}{r}_\ell\sum_{d=c/2 - r + 1}^{n - r}\binom{n - r}{d}_\ell (-1)^d \ell^{\binom{d}{2}}\PP[\ker M \supseteq \F_\ell^{r+d}] = O_{\ell, r}((n - c/2)\ell^{-c}).
\end{equation}
Finally, we will control the sum of the $\ell^{-(r + d)n}$ terms in the tail. Analogously to the above calculation, we have \begin{align*}
\log_\ell&\left|\binom{n}{r}_\ell\binom{n - r}{d}_\ell \ell^{\binom{d}{2}}\ell^{-(r + d)n}\right| \\
&\leq rn - \log_\ell|\GL_r(\F_\ell)| + d(n - r) - \log_\ell|\GL_d(\F_\ell)| + \binom{d}{2} - n(r + d) \\
&= - \log_\ell|\GL_r(\F_\ell)| - rd - \log_\ell|\GL_d(\F_\ell)| + \binom{d}{2} \\
&= -r^2 - (r + 1/2)d - d^2/2 + O_\ell(1)\ \\
&\leq -r^2 - (r + 1/2)(c/2 - r + 1) - (c/2 - r + 1)^2/2 + O_\ell(1) \\
&= -c^2/8 - 3c/4 + O_{\ell, r}(1).
\end{align*}
Thus we have \[
\binom{n}{r}_\ell\binom{n - r}{d}_\ell \ell^{\binom{d}{2}}\ell^{-(r + d)n} = O_{\ell, r}(\ell^{-c^2/8 - 3c/4}) = O_{\ell, r}(\ell^{-c}),
\]
whence \begin{equation}\label{eq:tail-iid}
\binom{n}{r}_\ell\sum_{d=c/2 - r + 1}^{n - r}\binom{n - r}{d}_\ell (-1)^d \ell^{\binom{d}{2}}\ell^{-(r + d)n} = O_{\ell, r}((n - c/2)\ell^{-c}).
\end{equation}

Combining \eqref{eq:sum-expression}, \eqref{eq:main-terms}, \eqref{eq:tail-csymm}, and \eqref{eq:tail-iid} yields the following result:

\begin{theorem}\label{thm:c-symmetric-ranks}
Fix $r \geq 0$. For large enough $n$ and for large enough even $c$, let $M$ be a uniformly random $C$-symmetric $n\times n$ matrix over $\F_\ell$, where $C$ is a fixed alternating form of rank $c$. Then we have \[
|\PP[\dim \ker M = r] - p_r(n)| = O_{\ell, r}(n\ell^{-c}).
\] 
\end{theorem}

A computation (e.g., \cite[Theorem 1.1]{fulmangoldstein2015ranks}) shows that $p_r(n)$ converges exponentially quickly in $n$ to $\mu_{CL, \ell}(r)$, leading to the following:

\begin{corollary}\label{cor:c-symmetric-ranks}
Fix $r \geq 0$. For large enough $n$ and for large enough even $c$, let $M$ be a uniformly random $C$-symmetric $n\times n$ matrix over $\F_\ell$, where $C$ is a fixed alternating form of rank $c$. Then we have \[
|\PP[\dim \ker M = r] - \mu_{CL, \ell}(r)| = O_{\ell, r}(n\ell^{-c}).
\] 
\end{corollary}

\section{Equidistribution of Artin Symbols}\label{sect:equidistribution}

Recall in Subsection~\ref{sect:param-curves}, we parametrized hyperelliptic curves by tuples $B = (p_0, p_1, \dots, p_n, t)$, where $p_0, \dots, p_n$ are points in $\PP^1$ and $t$ is a choice of uniformizer at $p_0$, up to squares.

In this section, we will determine the limiting corank distribution of the R\'edei matrix of a hyperelliptic curve drawn uniformly at random subject to some local conditions.

The types of local conditions we are allowed to fix are:
\begin{enumerate}[label=(\roman*)]
    \item the random hyperelliptic curve is constrained to be branched at a particular fixed finite set of points $S$;
    \item the random hyperelliptic curve is constrained to be unbranched at a particular fixed finite set of points $S' \cup S''$ (disjoint from $S$);
    \item the second-order class of the random hyperelliptic curve at a branch point of minimal odd degree is fixed depending on the branch point;
    \item the random hyperelliptic curve is split at $S'$ and inert at $S''$.
\end{enumerate}
For technical reasons, the last type of condition is the hardest to address, and we will have to handle it separately by finding the probability that the R\'edei matrix has a particular corank \textit{and} condition (iv) is satisfied if the curve is drawn uniformly at random from those hyperelliptic curves satisfying the first three conditions. 
 
First, we will show that as $B$ ranges over certain families, the matrix of the pairing $\ip{\cdot}{\cdot}_{X_B}'$ has the same rank distribution as a uniformly random $C$-symmetric matrix over $\F_2$ with row and column sums zero. These families will be determined by a tuple of integers prescribing the degrees of the ramification points of $X_B$. We will obtain a result when these degrees satisfy some numerical assumptions. Then in Subsection~\ref{sect:branch-degrees} we will show that if we pick $X_B$ uniformly at random among all hyperelliptic curves of fixed total branch degree satisfying local conditions of types (i), (ii), (iii), the degrees of the ramification points usually satisfy these assumptions.

We will briefly outline how to prove equidistribution in the subfamilies described above without imposing local conditions. Then, we will explain how this argument must be modified to account for local conditions of types (i), (ii), (iii).

Recall that the $ij$th entry of the R\'edei matrix of $X_B$ is given by \[
\ip{\pi_B^{-1}e_{p_0, p_i}}{\pi^{-1}e_{p_0, p_j}}_{X_B}' = c_t(H_{p_j} \to \PP^1_k, p_i)
\]
by Lemma~\ref{lem:pairing-by-hyperelliptics}. So, the $i$th row of the matrix is the information of the splitting type of $p_i$ in each of the curves $H_{p_j}$ for $j < i$, or equivalently, the splitting type of $p_i$ in the $2^{i-1}$-cover $H_{p_1} \times_{\PP^1} \dots \times_{\PP^1} H_{p_{i-1}} \to \PP^1$.

The Chebotarev density theorem for function fields says that this splitting type should equidistribute in $\F_2^{i-1}$ as $p_i$ varies. Thus, the $i$th row of the R\'edei matrix of $X_B$ should be close to uniformly distributed, conditional on $p_0, p_1, \dots, p_{i-1}$. We work this out in Section~\ref{sect:chebotarev}. Replacing the cover $H_{p_1} \times_{\PP^1} \dots \times_{\PP^1} H_{p_{i-1}} \to \PP^1$ by a carefully chosen modification allows us to exclude certain points from the set of possible values for $p_i$ to account for local conditions of type (ii).

The Riemann hypothesis for function fields affords us an effective version of the Chebotarev density theorem with an error term that shrinks exponentially with the degree of $p_i$. However, this is not quite good enough, since the error term also depends on the complexity of the cover $H_{p_1} \times_{\PP^1} \dots \times_{\PP^1} H_{p_{i-1}} \to \PP^1$. It turns out that the error term we get looks like $a^ib^{-\deg(p_i)}$ for some constants $a, b > 1$.

The number of points in a random subset of $\PP^1_{\F_q}$ of fixed total degree tends to be logarithmic in the degree. Thus, if we order the $p_i$ by increasing degree, we should expect that $\deg(p_i)$ grows exponentially in $i$. As a result, the error term from the Chebotarev density theorem should be small for large enough $i$. This tells us that, conditional on a small upper left corner of the R\'edei matrix, the rest of the matrix is close to uniformly distributed.

To handle the problem coming from the upper left corner, we use a trick due to Alex Smith \cite[Section 6]{smith_2infty-selmer_2017}. The idea is to abandon the hope that the R\'edei matrix itself equidistributes. However, we only need the rank of the R\'edei matrix to behave like the rank of a uniformly random $C$-symmetric matrix. We accomplish this by randomly permuting the rows and columns of the R\'edei matrix in a way that preserves the $C$-symmetry condition (and the rank!). After this additional random permutation, the R\'edei matrix is indeed close to uniformly distributed among all $C$-symmetric matrices.

This trick is remarkably insensitive to the upper left corner of the R\'edei matrix. As such, it allows us to fix a small number of ramification points, handling local conditions of type (i). The whole process is also performed conditional on $p_0$ and $t$, which allows us to account for local conditions of type (iii).

\subsection{Equidistribution from high-degree points}\label{sect:chebotarev}

In this subsection, we show that the distribution of the $i$th row of the (lower triangular part) of the R\'edei matrix is close to uniform conditional on the previous rows. The main input to this is an effective version of the Chebotarev density theorem for function fields:

\begin{theorem}[Consequence of \cite{friedjarden}, Proposition 7.4.8; see \cite{park2024primeselmerrankscyclic}, Theorem 3.1]
Let $Y \to \PP^1_{\F_q}$ be a branched Galois cover of curves with Galois group $G$. Let $g$ denote the genus of $Y$. Suppose the constant field of $Y$ is $\F_q$. Let $C \subset G$ be a conjugacy class. Then \begin{align*}
    &\left|\#\left\{p \in \PP^1 \ \middle|\ p\text{ unramified in }Y, \deg(p) = n, \Frob_p \in C\right\} - \frac{|C|}{|G|}\frac{q^n}{n}\right| \\
    &\qquad \leq \frac{2|C|}{n|G|}\left((|G| + g)q^{\frac{n}{2}} + 2|G|q^{\frac{n}{4}} + (|G| + g)\right).
\end{align*}
\end{theorem}

If $C \subset G$ is a union of conjugacy classes, we write \[
\pi_Y(C, n) \coloneqq \left\{p \in \PP^1 \ \middle|\ p\text{ unramified in }Y, \deg(p) = n, \Frob_p \in C\right\}.
\]

Using the above approximation for $\pi_Y(C, n)$ when $C$ is a conjugacy class, we can get approximations for $\pi_Y(C, n)/\pi_Y(C', n)$ when $C, C'$ are unions of conjugacy classes: 

\begin{corollary}[See \cite{park2024primeselmerrankscyclic}, Corollary 3.2]\label{cor:chebotarev}
Let $Y \to \PP^1_{\F_q}$ be a branched Galois cover of curves with Galois group $G$. Let $g$ denote the genus of $Y$. Suppose the constant field of $Y$ is $\F_q$. Let $C, C' \subset G$ be nonempty unions of conjugacy classes.

When $n \geq 2\log_q(8(|G| + g))$ we have \[
\left|\frac{\#\pi_Y(C, n)}{\#\pi_Y(C', n)} - \frac{|C|}{|C'|}\right| \leq 16\frac{|C|}{|C'|}(|G| + g)q^{-\frac{n}{2}}.
\]
Thus, for $0 < \delta < 1$, when $n \geq \frac{2}{1 - \delta}\log_q(8(|G| + g))$ we have \[
\left|\frac{\#\pi_Y(C, n)}{\#\pi_Y(C', n)} - \frac{|C|}{|C'|}\right|\leq 2\frac{|C|}{|C'|}q^{-\frac{\delta}{2}n}.
\]
\end{corollary}

Let $p_0 \in \PP^1_{\F_q}$ be a point of degree $d_0$ and let $d_1, \dots, d_n$ be positive integers satisfying the following conditions:
\begin{itemize}
    \item $d_1 \leq \dots \leq d_n$;
    \item $d_0$ is odd and is minimal among all odd $d_i$;
    \item $\sum_{i=0}^n d_i$ is even.
\end{itemize}
Let $t$ be a uniformizer class at $p_0$. Let $n_1$ be an integer with $1 \leq n_1 \leq n$. We will allow up to $n_1$ branch points to be fixed. For some choices of $1 \leq i \leq n_1$, let $T_i = \{p_i\}$ be singletons, where $p_i$ has degree $d_i$ (to account for local conditions of type (i) as described in the beginning of Section~\ref{sect:equidistribution}). For all other $1 \leq i \leq n$, let \[
T_i \coloneqq \{p \in \PP^1_k \mid \deg(p) = d_i\}
\]
We also want to condition on our random curve being unramified at a certain fixed set of points $u_1, \dots, u_\ell$ to account for local conditions of type (ii). We will also add one more technical assumption we will use throughout this section: \begin{itemize}
    \item at least one of $d_1, \dots, d_{n_1}$ is odd.
\end{itemize}
This assumption can be removed, but ultimately it is satisfied all the time (Corollary~\ref{cor:assumptions}, ($A_3$)), so it is harmless.

Let $T^*$ be the set of tuples $(p_0, p_1, \dots, p_n, t)$ such that the $p_i$ ($0 \leq i \leq n$) are all distinct and such that $p_i \notin \{u_1, \dots, u_\ell\}$ for each $i$. Recall that for each $B = (p_0, p_1, \dots, p_n, t) \in T^*$, we have a hyperelliptic curve $\pi_B\colon X_B \to \PP^1_{\F_q}$ branched at $p_0, p_1, \dots, p_n$ and such that $c_t(\pi_B, p_0) = 1$ (observe that we are imposing a local condition of type (iii) here). Also, recall from Lemma~\ref{lem:pairing-by-hyperelliptics} that for $1 \leq i < j \leq n$, the pairing $\ip{\pi_B^{-1}e_{p_0,p_i}}{\pi_B^{-1}e_{p_0, p_j}}'_{X_B} \in \F_2$ is independent of $p_r$ for $r \neq 0, i, j$, and we denote this by $\ip{p_i}{p_j}$.

Combined with the symmetry conditions described in Subsection~\ref{sect:symmetry}, this means that the upper-left $i\times i$ corner of the matrix of the pairing $\ip{\cdot}{\cdot}_{X_B}'$ only depends on $p_1, \dots, p_i$. 

We choose $B$ uniformly at random from $T^*$ and want to compute the distribution of $\ip{\cdot}{\cdot}'_{X_B}$. We do this by picking one point $p_i$ at a time. The following proposition says that, once $p_1, \dots, p_{i-1}$ are fixed, then the part of the upper-left $i\times i$ corner of the pairing matrix depending on $p_i$ will be equidistributed (as long as $d_i$ is large enough). 

\begin{proposition}\label{prop:row-equidistribution}
Fix $u_1, \dots, u_\ell \in \PP^1_{\F_q}$ of total degree $\sum_\beta \deg(u_\beta) =\colon d_{\text{exc}}$ (where we allow $\ell = 0$).

Let $i \geq n_1 + 1$ and let $(p_1, \dots, p_{i-1}) \in \prod_{j=1}^{i-1} T_j$ with $p_0, p_1, \dots, p_{i-1}$ distinct and with $p_j \notin \{u_1, \dots, u_\ell\}$ for each $0 \leq j \leq i - 1$. Let \[
T_i^* = T_i - \{p_0, p_1, \dots, p_{i-1}, u_1, \dots, u_\ell\}.
\]
Fix $M_1, \dots, M_{i-1} \in \F_2$. Suppose there is $0 < \delta < 1$ such that \[
d_i \geq \frac{2}{1 - \delta}\log_q(2^{i +1}(i+1)d_i + 2^{i+1}d_\text{exc} + 8),
\] 
If $p_i$ is chosen uniformly at random from $T_i^*$, then \[
\left|2^{i - 1}\PP[\ip{p_i}{p_j} = M_j \text{ for }1 \leq j \leq i - 1] -  1\right| \leq 2q^{-\frac{\delta}{2}d_i}.
\]
\end{proposition}
The factors of $2^{i + 1}$ in the condition for $d_i$ may be improved to $2^i$ if $\ell = 0$.
\begin{proof}
By Lemma~\ref{lem:pairing-by-hyperelliptics}, for each $1 \leq j \leq i - 1$ there is a hyperelliptic curve $H_j$ with branch locus contained in $\{p_0, p_j\}$ such that $\ip{p_i}{p_j} = c_t(H_j, p_i)$. Also, if $\ell > 0$, pick a hyperelliptic curve $\widetilde{H}$ branched at $u_1, \dots, u_\ell$, and possibly also at $p_0$ (if $d_{\text{exc}}$ is odd). Else, let $\widetilde{H} = \PP^1$.

The normalization $H$ of the curve \[
H_1 \times_{\PP^1} \times \dots \times_{\PP^1} H_{i-1} \times_{\PP^1} \widetilde{H} \to \PP^1
\]
is a Galois cover of $\PP^1$ with Galois group $G \coloneqq (\Z/2\Z)^{i-1} \times \Z/2\Z$ (if $\ell > 0$) or $(\Z/2\Z)^{i-1}$ (if $\ell = 0$), and the tuple $(\ip{p_i}{p_j})_{j=1}^{i-1}$ is exactly the image of the Frobenius element associated to $p_i$ under the projection $G \twoheadrightarrow (\Z/2\Z)^{i-1}$. Since $H_j$ is ramified at $p_j$, the cover $H \to \PP^1$ is ramified at $p_1, \dots, p_{i-1}$, and since at least one $1 \leq j \leq i - 1$ (in fact, $1 \leq j \leq n_1$) has $d_j$ odd, we must have that $H \to \PP^1$ is also ramified at $p_0$. Moreover, since $\widetilde{H} \to \PP^1$ is ramified at $u_1, \dots, u_\ell$, the cover $H \to \PP^1$ is also ramified at $u_1, \dots, u_\ell$.

Note that the constant field of $H$ is $\F_q$. To check this, we need to verify that no nontrivial extension of $\F_q$ is in the compositum of the function fields of the $H_i$ and $\widetilde{H}$. This compositum is the multiquadratic function field obtained by adjoining functions $\sqrt{f_i}$ and $\sqrt{\widetilde{f}}$, so it can only contain 2-power-degree extensions of $\F_q$. In other words, we need to avoid having $\F_{q^2}$ in the function field. This occurs only if the square root of a constant is in the function field, i.e., if some product of the $f_i$ and $\widetilde{f}$ is constant. But that cannot occur because they have independent sets of branch points.

Thus, the set $T_i^*$ consists precisely of degree $d_i$ points which are unramified in $H \to \PP^1$, i.e., $T_i^* = \pi_H(G, d_i)$. To apply Corollary~\ref{cor:chebotarev}, we need to understand the genus of $H$.

Let $R$ be the ramification divisor of $H \to \PP^1$. Each $H_j \to \PP^1$ is ramified of degree 2 at $p_j$, which contributes $2^{i-1}d_j \leq 2^{i-1}d_i$ to the degree of $R$. There may be some ramification at $p_0$ as well, which would contribute at most $2^id_0$ to the degree of $R$. However, note that since at least one of the $d_j$ is odd for $1 \leq j \leq i - 1$, we must have $d_0 \leq d_j \leq d_i$ for that value of $j$ by the conditions we placed on the degrees. Finally, each additional point $u_\beta$ contributes $2^{i-1}\deg(u_\beta)$ to the degree of $R$, for a total of $2^{i-1}d_{\text{exc}}$. Thus, we have \[
\deg(R) \leq 2^{i-1}(i-1)d_i + 2^id_i + 2^{i-1}d_{\text{exc}} = 2^{i-1}(i + 1)d_i + 2^{i-1}d_{\text{exc}}.
\]
By Riemann-Hurwitz, the genus of $H$ is bounded by: \[
g(H) \leq 2^{i-2}(i+1)d_i + 2^{i - 2}d_{\text{exc}} - |G| + 1.
\]
Now let $G'$ be the preimage of the vector $M$ under the map $G \twoheadrightarrow (\Z/2\Z)^{i-1}$. We have $|G|/|G'| = 2^{i-1}$, and the probability we seek is $\frac{\#\pi_H(G', d_i)}{\#\pi_H(G, d_i)}$. Corollary~\ref{cor:chebotarev} says that whenever \[
d_i \geq \frac{2}{1 - \delta}\log_q(2^{i +1}(i+1)d_i + 2^{i+1}d_\text{exc} + 8),
\]
we have \begin{align*}
\left|\frac{\#\pi_H(G', d_i)}{\#\pi_H(G, d_i)} - \frac{1}{2^{i-1}}\right| \leq 2\frac{1}{2^{i-1}}q^{-\frac{\delta}{2}d_i}.
\end{align*}
\end{proof}

We want to combine these estimates over each $n_1 + 1 \leq i \leq n$ for some small $n_1$. However, we will do something slightly different for the highest-degree ramification point to account for local conditions of type (iv).

\begin{proposition}\label{prop:last-row-equidistribution}
Fix $u_1, \dots, u_\ell \in \PP^1_{\F_q}$ of total degree $\sum_\beta \deg(u_\beta) = d_{\text{exc}}$ (where we allow $\ell = 0$) and fix $a_1, \dots, a_\ell \in \F_q^\times/(\F_q^\times)^2$.

Let $p_0, \dots, p_{n-1}$ be fixed distinct points in $\PP^1_{\F_q}$ with $\deg(p_j) = d_j$ for $j = 0, \dots, n - 1$ and such that $p_j \notin \{u_1, \dots, u_\ell\}$ for $j = 0, \dots, n - 1$. Let \[
T_n^* = T_n - \{p_0, \dots, p_{n-1}, u_1, \dots, u_\ell\}.
\]
Let $p_n$ be drawn uniformly at random from $T_n^*$ and $B = (p_0, p_1, \dots, p_n, t)$ as before. 
Fix $M_1, \dots, M_{n-1} \in \F_2$. Then, assuming that \[
d_n \geq \frac{2}{1 - \delta}\log_q(2^{\ell + n}(n + 1)d_n + 2^{\ell + n}d_{\text{exc}} + 8)
\] we have 
\begin{align*}
|2^{\ell + n - 1}\PP[\ip{p_n}{p_j} = M_j \text{ for }1 \leq j \leq n - 1\text{ and }c(X_B \to \PP^1, u_\beta) = a_\beta \text{ for }1 \leq \beta \leq \ell] - 1| \leq 2q^{-\frac{\delta}{2}d_n}.
\end{align*}
\end{proposition}
\begin{remark}
The slogan for this result is that the classes $c(X_B \to \PP^1, u_\beta)$ are independent of the values of the pairing $\ip{p_i}{p_j}$ that make up the entries of the R\'edei matrix, so that the distribution of $2\Pic^0(X_B)(\F_q)[4]$ should be insensitive to (unramified) local conditions. We will make this more precise in Corollary~\ref{cor:independence}.
\end{remark}
\begin{proof}
For $1 \leq j \leq n$ let $f_j$ be a rational function on $\PP^1_{\F_q}$ such that $\div f_j = e_{p_0, p_j}$ and such that the hyperelliptic curve $H_j \to \PP^1$ determined by $f_j$ via the correspondence of Lemma~\ref{lem:kummer} satisfies $c_t(H_j \to \PP^1, p_0) = 1$. The rational function $f_j$ is deterministic for $1 \leq j \leq n - 1$ and random for $j = n$. 

The hyperelliptic curve $X_f$ defined by the rational function $\prod_{i=1}^n f_j$ via the correspondence in Lemma~\ref{lem:kummer} has the same branch points as $X_B$ and satisfies $c_t(X_f \to \PP^1, p_0) = c_t(X_B \to \PP^1, p_0) = 1$, so in fact $X_f = X_B$.

Let $X_{n-1}$ be the (deterministic) hyperelliptic curve defined by the rational function $\prod_{j=1}^{n-1} f_j$. Both $X_{n-1}$ and $H_n$ are unramified at $u_\beta$ for each $1 \leq \beta \leq \ell$.

Also, for each $1 \leq \beta \leq \ell$ let $\widetilde{H}_\beta$ be a hyperelliptic curve whose branch locus is the support of the divisor $e_{p_0, u_\beta}$. 

We have  \[
c(X_B \to \PP^1, u_\beta) = c(X_{n-1} \to \PP^1, u_\beta)c(H_n \to \PP^1, u_\beta).
\]
By Lemma~\ref{lem:quadratic-reciprocity}, we have \begin{align*}
c(H_n \to \PP^1, u_\beta)c(\widetilde{H}_\beta \to \PP^1, p_n) &= c_t(H_n \to \PP^1, p_0)^{\deg(u_\beta)}c_t(\widetilde{H}_\beta \to \PP^1, p_0)^{d_n}(-1)^{\frac{q - 1}{2}\deg(u_\beta)d_n} \\
&= c_t(\widetilde{H}_\beta \to \PP^1, p_0)^{d_n}(-1)^{\frac{q - 1}{2}\deg(u_\beta)d_n}.
\end{align*}
The right hand side is a deterministic constant. Set \[
A_\beta \coloneqq c(X_{n-1} \to \PP^1, u_\beta)c_t(\widetilde{H}_\beta \to \PP^1, p_0)^{d_n}(-1)^{\frac{q - 1}{2}\deg(u_\beta)d_n}
\]
so that\[
c(X_B \to \PP^1, u_\beta) = A_\beta c(\widetilde{H}_\beta \to \PP^1, p_n).
\]
Let $\widetilde{H}$ be the (normalization of) the curve $\widetilde{H}_1 \times_{\PP^1} \dots \times_{\PP^1} \widetilde{H}_\ell$. This curve is a Galois cover of $\PP^1$ with Galois group $\widetilde{G} \coloneqq (\Z/2\Z)^\ell$. We determine the distribution of the tuple of classes $c(X_B \to \PP^1, u_\beta)$ using the same Chebotarev density argument as in Proposition~\ref{prop:row-equidistribution}. As such, we will omit some details that are identical to those in the proof of Proposition~\ref{prop:row-equidistribution}.

To get the joint distribution, we consider the normalization $H$ of the cover \[
H_1 \times_{\PP^1} \times \dots \times_{\PP^1} H_{n-1} \times_{\PP^1} \widetilde{H} \to \PP^1
\]
which is Galois with Galois group $G \coloneqq (\Z/2\Z)^{n-1}\times\widetilde{G}$. The $(\Z/2\Z)^{n-1}$ part of $\Frob_{p_n}$ determines the pairing entries $\ip{p_n}{p_j}$, and following the argument above, the $\widetilde{G}$ part determines the $c(X_B \to \PP^1, u_\beta)$.

Let $R$ be the ramification divisor of $H$. Ramification at each $p_i$, $1 \leq i \leq n - 1$, contributes $2^{\ell + (n - 1) - 1}d_i \leq 2^{\ell + n - 2}d_n$ to the degree of $R$ and $p_0$ contributes at most $2^{\ell + (n-1)}d_0 \leq 2(2^{\ell + n - 2})d_n$ for a total of at most $2^{\ell + n - 2}(n + 1)d_n$. Also, ramification at each $u_\beta$ contributes $2^{\ell + n - 2}\deg(u_\beta)$, for a total of $2^{\ell + n - 2}d_{\mathrm{exc}}$. Thus $R$ has \[
\deg(R) \leq 2^{\ell + n - 2}((n + 1)d_n + d_{\text{exc}})
\]
so that the genus of $H$ is bounded by \[
g(H) \leq 2^{\ell + n - 3}((n + 1)d_n + d_{\text{exc}}) - |G| + 1.
\]
Then Corollary~\ref{cor:chebotarev} says that whenever \[
d_n \geq \frac{2}{1 - \delta}\log_q(2^{\ell + n}((n + 1)d_n + d_{\text{exc}}) + 8)
\]
we have \begin{align*}
|2^{\ell + n - 1}\PP[\ip{p_n}{p_j} = M_j \text{ for }1 \leq j \leq n - 1\text{ and }c(X_B \to \PP^1, u_\beta) = a_\beta &\text{ for }1 \leq \beta \leq \ell] - 1| \leq 2q^{-\frac{\delta}{2}d_n}
\end{align*}
as we wanted.
\end{proof}

\begin{corollary}\label{cor:independence}
Fix $u_1, \dots, u_\ell \in \PP^1_{\F_q}$ of total degree $\sum_\beta \deg(u_\beta) = d_{\text{exc}}$ (where we allow $\ell = 0$) and fix $a_1, \dots, a_\ell \in \F_q^\times/(\F_q^\times)^2$.

Suppose there is $0 < \delta < 1$ such that  \[
d_n \geq \frac{2}{1 - \delta}\log_q(2^{\ell + n}(n + 1)d_n + 2^{\ell + n}d_{\text{exc}} + 8).
\]

Let $B = (p_0, p_1, \dots, p_n, t)$ be chosen uniformly at random from $T^*$ and let $Y$ be any random variable independent of $B$. Let $E_1$ be any event depending only on $Y$ and the values of $\ip{p_i}{p_j}$ for each $1 \leq j < i \leq n$, and let $E_2$ be the event that $c(X_B \to \PP^1, u_\beta) = a_\beta$ for each $1 \leq \beta \leq \ell$. Then \[
\left|\PP[E_1 \cap E_2] - 2^{-\ell}\PP[E_1]\right| \leq 2^{\binom{n}{2} -\ell - n + 3}q^{-\frac{\delta}{2}d_n}.
\]
\end{corollary}
\begin{proof}
The case $\ell = 0$ is immediate. We consider tuples $M = (M_{ij})_{1 \leq j < i \leq n}$ with $M_{ij} \in \F_2$. Let $E_1'(M)$ be the event that $\ip{p_i}{p_j} = M_{ij}$ for each $1 \leq j < i \leq n$. 

Condition on $p_1, \dots, p_{n-1}$ and $Y$. Then $p_n$ is uniformly random among points of degree $d_n$ excluding $p_0, p_1, \dots, p_{n-1}$ and $u_1, \dots, u_\ell$. Define \[
g(M, p_1, \dots, p_{n-1}) = \begin{cases}
    1 &\text{if } \ip{p_i}{p_j} = M_{ij} \text{ for each } 1 \leq j < i \leq n - 1; \\
    0 &\text{otherwise.}
\end{cases}
\]
From Proposition~\ref{prop:last-row-equidistribution} we have \[
|\PP[E_1'(M) \cap E_2 \mid p_1, \dots, p_{n-1}] - 2^{-\ell - n + 1}g(M, p_1, \dots, p_{n-1})| \leq 2^{-\ell - n + 2}q^{-\frac{\delta}{2}d_n}.
\]
From Proposition~\ref{prop:row-equidistribution} with $i = n$ we have \[
|\PP[E_1'(M) \mid p_1, \dots, p_{n-1}] - 2^{-n + 1}g(M, p_1, \dots, p_{n-1})| \leq 2^{-n + 2}q^{-\frac{\delta}{2}d_n},
\]
By the law of total probability we have \[
|\PP[E_1'(M) \cap E_2] - 2^{-\ell}\PP[E_1'(M)]| \leq 2^{-\ell - n + 3}q^{-\frac{\delta}{2}d_n}
\]
and if $E_1'$ is any event depending only on the values of $\ip{p_i}{p_j}$ for each $1 \leq j < i \leq n$, then \[
|\PP[E_1' \cap E_2] - 2^{-\ell}\PP[E_1']| \leq 2^{\binom{n}{2} -\ell - n + 3}q^{-\frac{\delta}{2}d_n}.
\]
We obtain the claimed result by conditioning on $Y$ and running the above argument.
\end{proof}

To get the conditional distribution of the rest of the matrix given the upper-left corner, we use the following lemma. The proof is straightforward, but can be found in, e.g. \cite[Lemma 4.12]{gorokhovsky2024timeinhomogeneousrandomwalksfinite}.

\begin{lemma}\label{lem:err-combining}
    Let $x_1, \dots, x_n \geq -1$ be real numbers such that $\sum_{m=1}^n \max\{0, x_m\} \leq \log 2$. Then \[
    \left|\prod_{m=1}^n (1 + x_m) - 1\right| \leq 2\sum_{m=1}^n |x_m|.
    \]
\end{lemma}

We apply this with $x_i$ being the expression in the absolute value on the left hand side of the inequality in Proposition~\ref{prop:row-equidistribution} for $n_1 + 1 \leq i \leq n$.

\begin{corollary}\label{cor:conditional-equidist}
Fix $u_1, \dots, u_\ell \in \PP^1_{\F_q}$ of total degree $\sum_\beta \deg(u_\beta) =\colon d_{\text{exc}}$ (where we allow $\ell = 0$).

Let $2 \leq n_1 \leq n$ and fix $(p_1, \dots, p_{n_1}) \in \prod_{j=1}^{n_1} T_j$ with $p_0, p_1, \dots, p_{n_1}$ distinct and such that $p_j \notin \{u_1, \dots, u_\ell\}$ for $0 \leq j \leq n_1$. Let $(p_{n_1 + 1}, \dots, p_n) \in \prod_{j=n_1 + 1}^n T_j$ be chosen uniformly at random among all tuples such that $p_0, p_1, \dots, p_n$ are all distinct and such that $p_j \notin \{u_1, \dots, u_\ell\}$ for $0 \leq j \leq n$. Let $B = (p_0, p_1, \dots, p_n, t)$ as before. For $n_1 + 1 \leq i \leq n$ and $1 \leq j \leq i - 1$, fix $M_{ij} \in \F_2$. Assume that: \begin{itemize}
    \item At least one of $d_1, \dots, d_{n_1}$ is odd;
    \item For $n_1 + 1 \leq i \leq n$ there is a $0 < \delta_i < 1$ such that we have \[
    d_i \geq \frac{2}{1 - \delta_i}\log_q(2^{i + 1}(i+1)d_i + 2^{i+1}d_\text{exc} + 8).
    \]
    \item We have \[
    2\sum_{i=n_1 + 1}^n q^{-\frac{\delta_i}{2}d_i} \leq \log 2.
    \]
\end{itemize}
Then \[
\left|2^{\binom{n}{2}-\binom{n_1}{2}}\PP\left[\ip{p_i}{p_j} = M_{ij} \text{ for all } n_1 + 1 \leq i \leq n, 1 \leq j < i\right] - 1\right| \leq 4\sum_{i=n_1 + 1}^n q^{-\frac{\delta_i}{2}d_i}.
\]
\end{corollary}
\details{
\begin{proof}
For $n_1 + 1 \leq i \leq n$ let $A_i$ be the event that $\ip{p_i}{p_j} = M_{ij}$ for $1 \leq j \leq i - 1$. Let $B_i = \bigcap_{j=n_1}^i A_j$ (we say $B_{n_1-1}$ is the full sample space). Proposition~\ref{prop:row-equidistribution} gives an estimate for $\PP[A_i \mid p_1, \dots, p_{i-1}]$ independent of $p_1, \dots, p_{i-1}$ (provided they satisfy the assumptions). In particular, the same estimate holds for $\PP[A_i \mid B_{i-1}]$ by the law of total probability.
For $i = n_1 + 1, \dots, n$ let \[
x_i = 2^{i-1}\PP[A_i \mid B_{i-1}] - 1 \qquad\text{ so that }\qquad  |x_i| \leq 2q^{-\frac{\delta_i}{2}d_i}.
\] 
The claim follows by \[
\PP[A_n] = \PP[A_n | B_{n-1}]\PP[A_{n-1} | B_{n-2}] \dots \PP[A_{n_1} \mid B_{n_1 - 1}]\PP[B_{n_1}]
\]
and Lemma~\ref{lem:err-combining}.
\end{proof}
}

\subsection{Averaging trick}\label{sect:averaging}

When proving equidistribution, we run into an issue when we are ranging over low-degree points in relatively high-degree extensions. Although we expect that degrees of points in a random selection get big fast, this isn't quite enough to deal with the large error terms in the upper left corner. We are going to use a trick from \cite[Section 6.2]{smith_2infty-selmer_2017} that ``mixes in the bad corner'' by randomly permuting the rows and columns.

In this section, we will use some notation and definitions which will not reappear in later sections for the sake of expressing the main results as generally as possible. The following notation will be in use for the rest of this section and nowhere else.

Let $F$ be a finite set with $\#F = \ell$ (later, we will have $F = \F_2$). Let $1 \leq n_1 \leq n$ be integers. For $1 \leq i < j \leq n$ let $f_{ij}\colon F \to F$ be a bijection. For $1 \leq j < i \leq n$ let $f_{ij} = f_{ji}^{-1}$. Write $f$ for the collection of $f_{ij}$ over all $1 \leq i \neq j \leq n$.

We say a function $M\colon \{(i, j) \in [n]^2 \mid i \neq j\} \to F$ is \textit{$f$-symmetric} if it satisfies $M(j, i) = f_{ij}(M(i, j))$ for all $1 \leq i \neq j \leq n$. We write $M_{n_1}$ for the restriction of $M$ to $\{(i, j) \in [n_1]^2 \mid i \neq j\}$.

\begin{remark}\label{rmk:csymm-is-fsymm}
If $\widetilde{M}$ is a $C$-symmetric matrix, then $M(i, j) = M_{ij}$ is an $f$-symmetric function valued in $F = \F_\ell$, with $f_{ij}(x) = x + C_{ji}$ for $x \in \F_\ell$. In our case of interest, we will work with $C$-symmetric matrices with row and column sums zero, so they are determined by their off-diagonal entries. In other words, a $C$-symmetric matrix with row and column sums zero is the same data as an $f$-symmetric function for this choice of $f$. The function $M_{n_1}$ is the data of the upper-left $n_1 \times n_1$ corner of $\widetilde{M}$.
\end{remark}

We will not actually use the data of the functions $f_{ij}$ in this section; we only need them to express the key idea that an $f$-symmetric function $M$ is uniquely determined by fixing one of $M(i, j)$ and $M(j, i)$ for each $1 \leq i < j \leq n$.

Let $G \subseteq S_n$ be a subgroup such that for $\sigma \in G$ we have $f_{ij} = f_{\sigma(i), \sigma(j)}$ for all $i \neq j$.

Consider the action of $G$ on the collection of subsets of $[n]$ of size $n_1$. Denote by $G[n_1]$ the orbit of the set $[n_1]$ under this action. Let
\[S(G,d,n_1) \coloneqq \frac{\#\{S\in G[n_1] \mid \# (S \cap [n_1])=d \}}{\# G[n_1]}\]
be the fraction of sets in this orbit that intersect $[n_1]$ at exactly $d$ elements. The rate at which $S(G, d, n_1)$ decays with increasing $d$ is a measure of how well the $G$-action mixes the first $n_1$ indices of $[n]$ into the rest of the positions.

For $\sigma \in G$, and $M$ an $f$-symmetric function, let $M^\sigma$ be the function given by $M^\sigma(i, j) = M(\sigma^{-1}(i), \sigma^{-1}(j))$. Note that for $1 \leq i < j \leq n$, we have \[
M^\sigma(j, i) = M(\sigma^{-1}(j), \sigma^{-1}(i)) = f_{\sigma^{-1}(i), \sigma^{-1}(j)}(M(\sigma^{-1}(i), \sigma^{-1}(j))) = f_{ij}(M^\sigma(i, j))
\]
so that $M^\sigma$ is $f$-symmetric. We write $M^\sigma_{n_1}$ for $(M^\sigma)_{n_1}$.

Note that if $M$ represents a $C$-symmetric matrix (with row and column sums zero), then $M^\sigma$ represents the $C$-symmetric matrix obtained by permuting the rows and columns of $M$ according to $\sigma$; in particular, the matrix represented by $M^\sigma$ has the same rank as the matrix represented by $M$.

The following lemma is the key technical result of this subsection. In Corollary~\ref{cor:f-symmetric-equidistribution}, we will use it to prove that a random $f$-symmetric function which is equidistributed away from $[n_1]^2$ is equidistributed after randomly permuting the domain according to the action defined above.

\begin{lemma}\label{lem:averaging}
With setup as above, let $\sigma$ be chosen uniformly at random from $G$. Let $P$ be a fixed $f$-symmetric function. Then \[
\sum_M\left|\PP[P_{n_1} = M^\sigma_{n_1}] - \ell^{-\binom{n_1}{2}}\right| \leq \ell^{\binom{n}{2}-\binom{n_1}{2}}\sqrt{\sum_{d\geq 0} S(G,d,n_1)\ell^{n_1d} - 1}
\]
where the sum ranges over all $f$-symmetric functions $M$.
\end{lemma}
\begin{remark}
The heuristic slogan for this lemma is that for \textit{most} $C$-symmetric matrices $M$, after randomly permuting the rows and columns in a way that respects the $C$-symmetry condition, the upper-left $n_1\times n_1$ corner is close to uniformly random. We can make this more precise.

Let $W(P, M) \coloneqq \PP[P_{n_1} = M^\sigma_{n_1} \mid M]$. Let $M'$ be chosen uniformly at random among all $f$-symmetric functions (independently of $\sigma$) and let $W(P) \coloneqq W(P, M')$.

In the proof of Lemma~\ref{lem:averaging}, we show that \[
\E\left[\left|W(P) - \ell^{-\binom{n_1}{2}}\right|\right] \leq \ell^{-\binom{n_1}{2}}\sqrt{\sum_{d\geq 0} S(G,d,n_1)\ell^{n_1d} - 1}.
\]
By Markov's inequality, for $\varepsilon > 0$, we have \[
\PP\left[\left|W(P) - \ell^{-\binom{n_1}{2}}\right|  \geq \varepsilon\right] \leq \varepsilon^{-1}\ell^{-\binom{n_1}{2}}\sqrt{\sum_{d\geq 0} S(G,d,n_1)\ell^{n_1d} - 1}.
\]
Now by the union bound, \[
\PP\left[\max_P \left|W(P) - \ell^{-\binom{n_1}{2}}\right| \geq \varepsilon\right] \leq \varepsilon^{-1}\sqrt{\sum_{d\geq 0} S(G,d,n_1)\ell^{n_1d} - 1}.
\]
Later, we will give bounds for $S(G, d, n_1)$ in our case of interest which will show that the square root expression on the right hand side converges to 0 as long as $n_1$ grows slowly enough relative to $n$.

So, we have shown that for any $\varepsilon > 0$, for most (aside from a fraction going to $0$) $f$-symmetric functions $M$, the restriction $M^\sigma_{n_1}$ has $L^\infty$-distance to the uniform distribution less than $\varepsilon$. 

We will not use this conclusion. In Corollary~\ref{cor:f-symmetric-equidistribution}, we use the statement of Lemma~\ref{lem:averaging} directly to prove an equidistribution statement about the whole function $P$, rather than its restriction $P_{n_1}$.
\end{remark}

\begin{proof}
Let \[
W(M) \coloneqq \PP[P_{n_1} = M^\sigma_{n_1} \mid M]
\]
be the conditional probability. Let $M'$ be chosen uniformly at random among all $f$-symmetric functions and let $W \coloneqq W(M')$. By the law of total probability, \[
\E[W] = \PP[P_{n_1} = (M')^\sigma_{n_1}] = \ell^{-\binom{n_1}{2}}
\]
because, for each value of $\sigma$, the restriction $(M')^\sigma_{n_1}$ is uniformly random among the $\ell^{\binom{n_1}{2}}$ functions $\{(i, j) \in [n_1]^2 \mid i \neq j\} \to F$ satisfying the symmetry conditions imposed by $f$.

Thus, the sum on the left hand side of the lemma is \[
\sum_M\left|\PP[P_{n_1} = M^\sigma_{n_1}] - \ell^{-\binom{n_1}{2}}\right| = \sum_M \left|W(M) - \E[W]\right| = \ell^{\binom{n}{2}}\E\left[\left|W - \E[W]\right|\right]
\]
where again the sum ranges over all $f$-symmetric functions $M$.

By Jensen's inequality, \[
\E[|W - \E[W]|] \leq \sqrt{\Var(W)}.
\]
Thus, we only need to bound the variance of $W$. We are going to do this by computing the average of $W^2$.

We have \[
W(M)^2 = \PP[P_{n_1} = M^{\sigma_1}_{n_1} = M^{\sigma_2}_{n_1} \mid M]
\]
where $\sigma_1, \sigma_2$ are independent uniformly random elements of $G$. Then, by the law of total probability, \[
\E[W^2] = \PP[P_{n_1} = (M')^{\sigma_1}_{n_1} = (M')^{\sigma_2}_{n_1}].
\]
We will bound the probability on the right hand side conditional on $\sigma_1$ and $\sigma_2$.

Fix values $\sigma_1, \sigma_2 \in G$. 

We view the equations $P_{n_1} = (M')^{\sigma_1}_{n_1} = (M')^{\sigma_2}_{n_1}$ as a system of $2\binom{n_1}{2}$ equations, two for each pair of values of $P_{n_1}$. Each of these equations is of the form $P(i', j') = M'(\sigma_1^{-1}(i'), \sigma_1^{-1}(j'))$ or $P(i', j') =  M'(\sigma_2^{-1}(i'), \sigma_2^{-1}(j'))$ for some $1 \leq i' < j' \leq n_1$. However, there may be some redundancy in these equations coming from pairs of indices $1 \leq i' < j' \leq n_1$ and $1 \leq i'' < j'' \leq n_1$ with $(\sigma_1^{-1}(i'), \sigma_1^{-1}(j')) = (\sigma_2^{-1}(i''), \sigma_2^{-1}(j''))$. Let $I$ be the set of pairs $(i, j)$ with $1 \leq i < j \leq n$ such that $M'(i, j)$ appears in the system of equations $P_{n_1} = (M')^{\sigma_1}_{n_1} = (M')^{\sigma_2}_{n_1}$, i.e., $1 \leq \sigma_1(i), \sigma_1(j) \leq n_1$ or $1 \leq \sigma_2(i), \sigma_2(j) \leq n_1$.

The event that $P_{n_1} = (M')^{\sigma_1}_{n_1} = (M')^{\sigma_2}_{n_1}$ has probability (conditional on $\sigma_1, \sigma_2$) either 0 or $\ell^{-\#I}$, depending on whether the conditions $P_{n_1} = (M')^{\sigma_1}_{n_1}$ and $P_{n_1} = (M')^{\sigma_2}_{n_1}$ contradict each other. So, to bound $\E[W^2]$ from above, we want a lower bound on $\#I$.

To get a lower bound on $\#I$, we want an upper bound on the number of entries of $M'$ that appear more than once among the $2\binom{n_1}{2}$ conditions imposed by $P(i', j') = M'(\sigma_1^{-1}(i'), \sigma_1^{-1}(j')) = M'(\sigma_2^{-1}(i'), \sigma_2^{-1}(j'))$ for $1 \leq i' < j' \leq n_1$. These are indexed by pairs $1 \leq i \neq j \leq n$ such that $1 \leq \sigma_1(i) < \sigma_1(j) \leq n_1$ \textit{and} $1 \leq \sigma_2(i) < \sigma_2(j) \leq n_1$.

Let $d(\sigma_1, \sigma_2)$ be the number of indices $1 \leq i \leq n$ such that $\sigma_1(i), \sigma_2(i) \leq n_1$. Then the number of entries of $M'$ that appear more than once among the conditions imposed by $P_{n_1} = (M')^{\sigma_1}_{n_1} = (M')^{\sigma_2}_{n_1}$ is bounded above by $n_1d(\sigma_1, \sigma_2)$\details{there are at most $n_1$ choices for $\sigma_1(j)$, so at most $n_1$ choices for $j$}, and so we have $\#I \geq 2\binom{n_1}{2} - n_1d(\sigma_1, \sigma_2)$.

Thus, \[
\PP[P_{n_1} = (M')^{\sigma_1}_{n_1} = (M')^{\sigma_2}_{n_1} \mid \sigma_1, \sigma_2] \leq \ell^{-2\binom{n_1}{2} + n_1d(\sigma_1, \sigma_2)}
\]
(as functions of $\sigma_1, \sigma_2$) which means \[
\PP[P_{n_1} = (M')^{\sigma_1}_{n_1} = (M')^{\sigma_2}_{n_1}] \leq \sum_{d \geq 0}  \PP[d(\sigma_1, \sigma_2) = d]\ell^{-2\binom{n_1}{2} + n_1d}.
\]
So, we want to bound $\PP[d(\sigma_1, \sigma_2) = d]$ for each $d$. Indices $i$ such that $1 \leq \sigma_1(i), \sigma_2(i) \leq n_1$ correspond to indices $i' = \sigma_1(i) \leq n_1$ such that $\sigma_2\sigma_1^{-1}(i') \leq n_1$. Thus, the number of such indices is \[
\#(\sigma_2\sigma_1^{-1}([n_1]) \cap [n_1]).
\] 
Note that this only depends on $\sigma_2\sigma_1^{-1}$, which is uniformly random in $G$. The probability that $\#(\sigma_2\sigma_1^{-1}([n_1]) \cap [n_1]) = d$ is precisely $S(G, d, n_1)$. So,\[
\E[W^2] \leq \ell^{-2\binom{n_1}{2}}\sum_{d \geq 0} S(G, d, n_1) \ell^{n_1d}
\]
and therefore
\begin{align*}
    \Var(W) &= \E[W^2] - \E[W]^2 \\
    &\leq \ell^{-2\binom{n_1}{2}}\sum_{d\geq 0} S(G,d,n_1)\ell^{n_1d} - \ell^{-2\binom{n_1}{2}} \\
    &= \ell^{-2\binom{n_1}{2}}\left(\sum_{d\geq 0} S(G,d,n_1)\ell^{n_1d} - 1\right)
\end{align*}
so that \[
\E[|W - \E[W]|] \leq \ell^{-\binom{n_1}{2}}\sqrt{\sum_{d\geq 0} S(G,d,n_1)\ell^{n_1d} - 1}
\]
which completes the proof.
\end{proof}

\begin{corollary}\label{cor:f-symmetric-equidistribution}
Let $P$ be a random $f$-symmetric function. Suppose that for any $f$-symmetric function $M$ satisfying $\PP[P_{n_1} = M_{n_1}] \neq 0$ we have \[
\left|\PP[P = M \mid P_{n_1} = M_{n_1}] - \ell^{\binom{n_1}{2} - \binom{n}{2}}\right| \leq \delta\ell^{\binom{n_1}{2} - \binom{n}{2}}
\]
for some $\delta > 0$. Then if $M'$ is a uniformly random $f$-symmetric function and $\sigma$ is a uniformly random element of $G$, we have \[
d_{TV}(P^\sigma, M') \leq  \frac{\delta}{2} + \frac{1}{2}\sqrt{\sum_{d\geq 0} S(G,d,n_1)\ell^{n_1d} - 1}
\]
where $d_{TV}$ is total variation distance.
\end{corollary}
In our application, $P$ will be the R\'edei matrix, and the assumption on the conditional probabilities $\PP[P = M \mid P_{n_1} = M_{n_1}]$ will follow from Corollary~\ref{cor:conditional-equidist}.
\begin{proof}
We will bound \[
d_{TV}(P^\sigma, M') = \frac{1}{2}\sum_M \left|\PP[P^\sigma = M] - \ell^{-\binom{n}{2}}\right|
\]
where the sum is over all $f$-symmetric functions $M$.

We have \[
\PP[P^\sigma = M] = \PP[P = M^{\sigma^{-1}}] = \PP[P = M^{\sigma^{-1}} \mid P_{n_1} = M^{\sigma^{-1}}_{n_1}]\PP[P_{n_1} = M^{\sigma^{-1}}_{n_1}].
\]
Thus, \begin{align*}
\sum_M \left|\PP[P^\sigma = M] - \ell^{-\binom{n}{2}}\right| &\leq \sum_M\PP[P_{n_1} = M^{\sigma^{-1}}_{n_1}]\left|\PP[P = M^{\sigma^{-1}} \mid P_{n_1} = M^{\sigma^{-1}}_{n_1}] - \ell^{\binom{n_1}{2} - \binom{n}{2}}\right| \\
&\qquad\qquad+ \sum_M \ell^{\binom{n_1}{2} - \binom{n}{2}}\left|\PP[P_{n_1} = M^{\sigma^{-1}}_{n_1}] - \ell^{-\binom{n_1}{2}}\right|.
\end{align*}
To bound the first term, we use the assumption in the corollary statement and the observation that $\sum_M \PP[P_{n_1} = M^{\sigma^{-1}}_{n_1}] = \ell^{\binom{n}{2} - \binom{n_1}{2}}$.\details{
We have \[
\sum_M \PP[P_{n_1} = M^{\sigma^{-1}}_{n_1}] = \frac{1}{|G|}\sum_M \sum_{\sigma_0 \in G}\PP[P_{n_1} = M^{\sigma_0^{-1}}_{n_1}] = \frac{1}{|G|}\sum_{\sigma_0 \in G}\sum_M \PP[P_{n_1} = M^{\sigma_0^{-1}}_{n_1}].
\]
We can split the sum over $M$: \[
\sum_M \PP[P_{n_1} = M^{\sigma_0^{-1}}_{n_1}] = \sum_{N_{n_1}}\sum_{M^{\sigma_0^{-1}}_{n_1} = N_{n_1}}\PP[P_{n_1} = M^{\sigma_0^{-1}}_{n_1}] = \ell^{\binom{n}{2} - \binom{n_1}{2}}\sum_{N_{n_1}}  \PP[P_{n_1} = N_{n_1}] = \ell^{\binom{n}{2}-\binom{n_1}{2}}
\]
where $N_{n_1}$ ranges over all possible $f$-symmetric functions $n_1$ with inputs in $[n_1]$.
}
Finally, the second term is bounded by Lemma~\ref{lem:averaging}.
\end{proof}

\begin{lemma}\label{lem:C-symm-mixing}
Use notation from Lemma~\ref{lem:averaging}. Let $[n] = O \sqcup E$ be a partition of $[n]$ into subsets of size $\#O =\colon n'$ and $\#E =\colon n''$. Let $n_1' \coloneqq \#(O \cap [n_1])$ and $n_1'' \coloneqq \#(E \cap [n_1])$. Let $G = S_{n'} \times S_{n''} \subseteq S_n$ act on $O$ and $E$ independently. Then \[
S(G, d, n_1) \leq \frac{1}{d!}\left(\frac{(n_1')^2}{n'} + \frac{(n_1'')^2}{n''}\right)^d.
\]
where we adopt the convention $0^2/0 = 0$.
\end{lemma}
\begin{remark}
The same argument gives a similar bound for any products of symmetric or alternating groups on at least $n_1 + 2$ letters, using multiple transitivity of such groups.
\end{remark}
\begin{proof}
We give the proof in the most important case where $n', n'' > 0$. The case where $n'' = 0$ is strictly easier. $S(G, d,n_1)$, the proportion of sets $\pi([n_1])$ in the orbit $G[n_1]$ such that $\#(\pi([n_1]) \cap [n_1]) = d$, is equal to
\[
S(G, d,n_1)=\frac{\sum_{h'+h''=d}\binom{n_1'}{h'}\binom{n'-n_1'}{n_1'-h'}\binom{n_1''}{h''}\binom{n''-n_1''}{n_1''-h''}}{\binom{n'}{n_1'}\binom{n''}{n_1''}}.
\]
We note that 
\begin{align*}
    \frac{\binom{n_1'}{h'}\binom{n'-n_1'}{n_1'-h'}}{\binom{n'}{n_1'}}&=\frac{n_1'!^2(n'-h')!}{n'!h'!(n_1'-h')!^2}\cdot\frac{(n'-n_1')!^2}{(n'-h')!((n'-n_1')-(n_1'-h'))!}\\
    &\leq \frac{n_1'!^2(n'-h')!}{n'!h'!(n_1'-h')!^2}\\
    &\leq\frac{1}{h'!}\cdot \left(\frac{n_1'^2}{n'}\right)^{h'}.
\end{align*}
Treating the rest of the terms similarly, we receive
\begin{align*}
    S(G, d,n_1)&\leq\sum_{h'+h''=d}\frac{1}{h'!}\cdot \left(\frac{n_1'^2}{n'}\right)^{h'}\frac{1}{h''!}\cdot \left(\frac{n_1''^2}{n''}\right)^{h''}\\
    &=\frac{1}{d!}\left(\frac{n_1'^2}{n'}+\frac{n_1''^2}{n''}\right)^{d}.
\end{align*}
\end{proof}

\subsection{Corank distribution of the R\'edei matrix in subfamilies}\label{sect:subfamilies}

In this subsection, we combine the results of the previous two sections and give some conditions on the degrees $d_0, d_1, \dots, d_n$ under which the matrix of the pairing $\ip{\cdot}{\cdot}'_{X_B}$ has the same corank distribution as uniformly random $C$-symmetric matrix when $B$ is chosen uniformly at random from a family $T^*$ as defined in Subsection~\ref{sect:chebotarev}.

\begin{theorem}\label{thm:subfamily-equidist}
Let $q$ be an odd prime power. Fix distinct points $u_1, \dots, u_\ell \in \PP^1_{\F_q}$ with $\sum_\beta \deg(u_\beta) = d_{\text{exc}}$. (where we allow $\ell = 0$) and fix $a_1, \dots, a_\ell \in \F_q^\times/(\F_q^\times)^2$.

Let $p_0 \in \PP^1_{\F_q} - \{u_1, \dots, u_\ell\}$ be a point of odd degree $d_0$ and $d_1, \dots, d_n$ be positive integers such that:
\begin{itemize}
    \item $d_1 \leq \dots \leq d_n$;
    \item if $d_i$ is odd for some $1 \leq i \leq n$, then $d_0 \leq d_i$;
    \item $\sum_{i=0}^n d_i$ is even;
\end{itemize}
and let $t$ be a uniformizer class at $p_0$. 

Also, suppose the following conditions are satisfied:
\begin{enumerate}[label=$(A_{\arabic*})$]
    \setcounter{enumi}{1}
    \item the number of $d_i$ ($0 \leq i \leq n$) which are odd and the number of $d_i$ ($0 \leq i \leq n$) which are even are both at least $\varepsilon n$ for some $0 < \varepsilon < 1$;
    \item there is at least one $1 \leq i < \sqrt{\log n}$ with $d_i$ odd;
    \item for $i \geq \sqrt{\log n}$, we have $d_i \geq 4\binom{i}{2}$.
\end{enumerate}

Choose a subset of indices $1 \leq i < \sqrt{\log n}$ and set $T_i \coloneqq \{r_i\}$ for these indices, where the $r_i$ are distinct points in $\PP^1_{\F_q} - \{p_0, u_1, \dots, u_\ell\}$ of degree $d_i$. For all other $1 \leq i \leq n$, let \[
T_i \coloneqq \{p \in \PP^1_{\F_q} - \{u_1, \dots, u_\ell\} \mid \deg(p) = d_i\}
\]
and let $T^* \subseteq \{p_0\} \times \prod_{i=1}^n T_i \times \{t\}$ be the set of tuples $(p_0, p_1, \dots, p_n, t)$ such that $p_0, p_1, \dots, p_n$ are all distinct. (Note that for certain $i \leq \sqrt{\log n}$ we are forcing $p_i = r_i$.)

Let $B = (p_0, p_1, \dots, p_n, t)$ be chosen uniformly at random from $T^*$. Let $\mu \coloneqq \mu_{CL, 2}$ if $q\equiv 3\pmod{4}$ and $\mu\coloneqq \mu_{S, 2}$ if $q\equiv 1\pmod{4}$. Let $r \geq 0$ be an integer. Let $E$ be the event that $c(X_B \to \PP^1, u_\beta) = a_\beta$ for $1 \leq \beta \leq \ell$. Then \begin{align*}
\left|\PP[\{\dim_{\F_2}2\Pic^0(X_B)(\F_q)[4] = r\} \cap E] - 2^{-\ell}\mu(r)\right| = O_{q, r, \ell, \varepsilon, d_{\text{exc}}}\left(n^{-1/4}\right).
\end{align*}
\end{theorem}
\begin{remark}
Since the result holds uniformly in $p_0$ and $t$, it also holds if $p_0$ and $t$ are chosen randomly.
\end{remark}
\begin{remark}\label{rmk:A1}
Let $d = \sum_{i=0}^n d_i$. If we also introduce a condition \begin{enumerate}[label=$(A_{\arabic*})$]
    \item $n \geq (1 - \gamma)\log d$
\end{enumerate}
for some $\gamma < 1$, then we can conclude that \[
\left|\PP[\{\dim_{\F_2}2\Pic^0(X_B)(\F_q)[4] = r\} \cap E] - 2^{-\ell}\mu(r)\right| = O_{q, r, \ell, \gamma, \varepsilon, d_{\text{exc}}}\left((\log d)^{-1/4}\right).
\] 
This will be useful when passing from Theorem~\ref{thm:subfamily-equidist} to the main results, where the number of branch points is random but their total degree is fixed.
\end{remark}
\begin{proof}
First, we will check that the assumptions we have made on the growth of the degrees imply the assumptions needed to apply Propositions~\ref{prop:row-equidistribution}, \ref{prop:last-row-equidistribution}, and their corollaries. If $d_i \geq 4\binom{i}{2}$, we observe that for $i \geq 3$ we have $d_i \geq 4i$, so $q^{\frac{d_i}{4}}/d_i \geq q^i/4i$, and for sufficiently large $i$ (depending on $\ell, d_{\text{exc}}$, and $q \geq 3$) this is larger than $2^{\ell + i + 1}(i + 1) + 2^{\ell + i + 1}d_{\text{exc}}/d_i + 8/d_i$. In particular, for large enough $n$ (depending on the same variables) the degrees $d_i$ satisfy the assumptions in Corollaries~\ref{cor:independence} and \ref{cor:conditional-equidist} with $\delta_i = 1/2$ as long as $i \geq \sqrt{\log n}$.

Let $M_B$ be the matrix representing the pairing $\ip{\cdot}{\cdot}'_{X_B}$. By Lemma~\ref{lem:remove-column} and the construction of the pairing $\ip{\cdot}{\cdot}'_{X_B}$, we have \[
\dim_{\F_2}\ker M_B = \dim_{\F_2}2\Pic^0(X_B)(\F_q)[4] + 1.
\]

The matrix $M_B$ is $C$-symmetric, where $C$ is described at the start of Section~\ref{sect:random-matrices}: \begin{itemize}
    \item If $q \equiv 1\pmod{4}$, then $C$ is the all-zeroes matrix.
    \item If $q \equiv 3\pmod{4}$, write $n' \geq 1$ for the number of odd degrees among $d_1, \dots, d_n$. Since $d_0$ is odd and $\sum_{i=0}^n d_i$ is even, we must have that $n'$ is odd. The rank of $C$ is therefore $n' - 1$.
\end{itemize}
Also, let $n'' = n - n'$ be the number of even degrees among $d_1, \dots, d_n$.

We want to show that for large $n$, the nullity distribution of $M_B$ is close to that of a uniformly random $C$-symmetric matrix over $\F_2$ with row and column sums zero, which was computed in Section~\ref{sect:random-matrices}.

If $M$ is an $n\times n$ matrix and $\sigma \in S_n$, denote by $M^\sigma$ the matrix whose $ij$th entry is $M_{\sigma^{-1}(i), \sigma^{-1}(j)}$. If $M$ is $C$-symmetric, then $M^\sigma$ is $C^\sigma$-symmetric. Moreover, $M$ and $M^\sigma$ have the same nullity. Let $G = S_{n'} \times S_{n''} \subseteq S_n$ permute basis elements corresponding to odd and even degrees independently. If $\sigma \in G$, then $C^\sigma = C$.

Let $E$ be the event that $c(X_B \to \PP^1, u_\beta) = a_\beta$ for $1 \leq \beta \leq \ell$.

Let $\sigma$ be a uniformly random element of $G$. Let $M'$ be drawn uniformly at random from $C$-symmetric matrices with row and column sums zero. 

Since $\dim\ker M_B^\sigma = \dim\ker M_B$, it suffices to estimate $\PP[\{\dim\ker M_B^\sigma = r + 1\} \cap E]$. To start, by Corollary~\ref{cor:independence}, we have \[
\left|\PP[\{\dim\ker M_B^\sigma = r + 1\} \cap E] - 2^{-\ell}\PP[\dim\ker M_B^\sigma = r + 1]\right| \leq 2^{\binom{n}{2} - \ell - n + 3}q^{-\frac{d_n}{4}}.
\]
So, it remains to estimate $2^{-\ell}\PP[\dim\ker M_B^\sigma = r + 1]$.

We will use the equidistribution results of Subsections~\ref{sect:chebotarev} and \ref{sect:averaging}. We will view $C$-symmetric matrices as $f$-symmetric functions for a particular choice of $f$ as described in Remark~\ref{rmk:csymm-is-fsymm}. Set \[
n_1 \coloneqq  \lfloor \sqrt{\log n}\rfloor.
\] 
By Corollary~\ref{cor:conditional-equidist}, the random matrix $M_B$ satisfies the assumption of Corollary~\ref{cor:f-symmetric-equidistribution} with $\delta = 4\sum_{i=n_1 + 1}^n q^{-\frac{d_i}{4}}$. Thus, by Corollary~\ref{cor:f-symmetric-equidistribution} we get that \[
d_{TV}(M_B^\sigma, M') \leq 2\sum_{i=n_1 + 1}^n q^{-\frac{d_i}{4}} + \frac{1}{2}\sqrt{\sum_{d\geq 0}S(G, d, n_1)2^{n_1d} - 1}.
\]
In particular, \begin{equation}\label{eq:equidist-err}
|2^{-\ell}\PP[\dim\ker M_B^\sigma = r + 1] - 2^{-\ell}\PP[\dim\ker M' = r + 1]| \leq 2^{- \ell + 1}\sum_{i=n_1 + 1}^n q^{-\frac{d_i}{4}} + 2^{-\ell - 1}\sqrt{\sum_{d\geq 0}S(G, d, n_1)2^{n_1d} - 1}.
\end{equation}
Let $n_1'$ be the number of odd degrees among $d_1, \dots, d_{n_1}$ and $n_1''$ be the number of even degrees among $d_1, \dots, d_{n_1}$. By Lemma~\ref{lem:C-symm-mixing}, we have \[
S(G, d, n_1) \leq \frac{1}{d!}L^d \qquad\text{ where }\qquad L \coloneqq \frac{(n_1')^2}{n'} + \frac{(n_1'')^2}{n''}
\]so the second term on the right hand side is bounded by \[
2^{-\ell - 1}\sqrt{\sum_{d\geq 0}S(G, d, n_1)2^{n_1d} - 1} \leq 2^{-\ell - 1}\sqrt{\sum_{d\geq 0}\frac{1}{d!}L^d2^{n_1d} - 1} = 2^{-\ell - 1}\sqrt{\exp\left(2^{n_1}L\right) - 1}.
\]
Finally, by Corollary~\ref{cor:effective-symmetric-dist} and Corollary~\ref{cor:c-symmetric-ranks} (and using Lemma~\ref{lem:remove-column}), it follows that \[
\left|2^{-\ell}\PP[\dim\ker M' = r + 1] - 2^{-\ell}\mu(r)\right| = O_r(n\cdot 2^{-n'-\ell}).
\]
(Note that we have a better bound from Corollary~\ref{cor:effective-symmetric-dist} in the case $q \equiv 1\pmod{4}$, but it is not needed here.) 

To summarize, we have four error terms:\begin{itemize}
    \item \textbf{Coming from using Corollary~\ref{cor:independence} to treat local conditions independently:} we have\[
    \left|\PP[\{\dim\ker M_B^\sigma = r + 1\} \cap E] - 2^{-\ell}\PP[\dim\ker M_B^\sigma = r + 1]\right| \leq 2^{\binom{n}{2} - \ell - n + 3}q^{-\frac{d_n}{4}} = \boxed{O_\ell\left(\left(\frac{2}{q}\right)^{\binom{n}{2}}2^{-n}\right)}.
    \]

    \item \textbf{Coming from the averaging trick:} this is the second term on the right hand side of Equation~\ref{eq:equidist-err}. We bound $n_1'$ and $n_1''$ above by $n_1 \leq \sqrt{\log n}$, and we bound $n'$ and $n''$ below by $\varepsilon n - 1$ to get $L \leq \frac{2\log n}{\varepsilon n - 1}$, so the error bound we get is \[
    2^{-\ell - 1}\sqrt{\exp\left(\frac{2^{1 + \sqrt{\log n}}\log n}{\varepsilon n - 1}\right) - 1}
    \]
    For large enough $n$, we have \[
    \frac{1}{2} \log (\varepsilon n - 1) \geq \log\log n + \frac{\sqrt{\log n}}{\log 2} + \log 2 
    \]
    so that the term in the exponent is at most $(\varepsilon n - 1)^{-\frac{1}{2}}$. Then the full error term is $\boxed{O_{\ell, \varepsilon}(n^{-\frac{1}{4}})}$.
    
    \item \textbf{Coming from Chebotarev:} this is the first term on the right hand side of Equation~\ref{eq:equidist-err}. We have an error of
        \[
        2^{-\ell + 1}\cdot \sum_{i=n_1 + 1}^n q^{-\frac{d_i}{4}} \leq 2^{-\ell + 1}\sum_{i=n_1 + 1}^\infty q^{-\binom{i}{2}} \leq \frac{2^{-\ell + 1}}{1 - q^{-1}}\left(\frac{1}{q}\right)^{\binom{n_1}{2}} = \boxed{O_{q, \ell}\left(n^{-\frac{1}{2}\log\left(q\right)}q^{\frac{3}{2}\sqrt{\log n}}\right)} = O_{q, \ell}\left(n^{-\frac{1}{4}\log (q)}\right).
    \]
    We note that since $q \geq 3$, we have $\frac{1}{4}\log(q) > \frac{1}{4}$.
    \details{In the last line, we used the bound $q^{-\binom{i}{2}} \leq q^{-\binom{n_1}{2} - (i - n_1)}$ to compare to a geometric series.}

    \item \textbf{Coming from rank distributions of $C$-symmetric matrices:} \[
    \left|2^{-\ell}\PP[\dim\ker M' = r + 1] - 2^{-\ell}\mu(r)\right| = O_r(n\cdot 2^{-n'}) = \boxed{O_{r, \ell, \varepsilon}(n\cdot 2^{-\varepsilon n})}
    \]
    by Corollary~\ref{cor:c-symmetric-ranks}.
\end{itemize}
The largest error term comes from the averaging trick.
\end{proof}

\subsection{Branch degrees}\label{sect:branch-degrees}

In this subsection, we move from the result in the previous subsection about corank distribution in subfamilies of curves with prescribed degrees of individual ramification points (Theorem~\ref{thm:subfamily-equidist}) to showing the same corank distribution arises if we take our curves uniformly among those with fixed total degree of ramification points.  

The idea is that a uniformly random curve with ramification points of total even degree $d$ (including $r_1, \dots, r_k$ and excluding $u_1, \dots, u_\ell$) can be sampled in two steps:
\begin{enumerate}
    \item Sample a random set of distinct points in $\PP^1$ with total degree $d - \sum_{\alpha=1}^k \deg(r_\alpha)$, excluding $r_1, \dots, r_k$ and $u_1, \dots, u_\ell$. Remember the degrees of these points with multiplicity, but not the points themselves.
    \item Given a set of degrees with multiplicity, sample a random hyperelliptic curve branched at $r_1, \dots, r_k$ whose other ramification points have these degrees and multiplicities.
\end{enumerate}
This works because any two curves with the same degrees of ramification points are equally likely to be sampled. In Theorem~\ref{thm:subfamily-equidist}, we showed that, conditional on some properties of the multiset of degrees obtained in the first step, the Picard group of a curve sampled in the second step will have the desired 4-rank distribution and (morally) that this 4-rank distribution is independent of local conditions imposed at $u_1, \dots, u_\ell$.

In this section, we use the general theory of random logarithmic combinatorial structures developed by Arratia, Barbour, and Tavar\'e \cite{arratiabarbourtavare} to bound the probability that the degrees sampled from the first step fail to have the desired properties. 

A \textit{selection} is a random process defined as follows. We start with a universe of objects, each having a positive integer \textit{weight}, such that there are exactly $m_i$ objects of weight $i$. In the case of interest, our objects will be points in $\PP^1_{\F_q}$ excluding $r_1, \dots, r_k, u_1, \dots, u_\ell$, weighted by degree. The selection is the data of this universe together with a subset of the universe chosen uniformly at random among all subsets of total weight $d - \sum_{\alpha = 1}^k \deg(r_\alpha)$. 

A selection is called \textit{logarithmic} if it satisfies $m_i \sim \frac{\theta y^i}{i}$ as $i \to \infty$ for some $y > 1$ and $\theta > 0$. If $m_i$ is the number of points of degree $i$ in $\PP^1_{\F_q}$ excluding $r_1, \dots, r_k$ and $u_1, \dots, u_\ell$, then for $i$ large enough (e.g., $i > 1$ and $i > \deg(r_\alpha), \deg(u_\beta)$ for all $\alpha, \beta$), $m_i$ is the number of irreducible monic polynomials of degree $i$ over $\F_q$. For such $i$, we have the estimate \begin{equation}\label{eq:mi-estimate}
\frac{q^i}{i} - \frac{q^{\frac{i}{2}}}{i} - q^{\frac{i}{3}} \leq m_i \leq \frac{q^i}{i} + \frac{q^{\frac{i}{2}}}{i} + q^{\frac{i}{3}}
\end{equation}
(see, e.g., \cite[Theorem 2.2]{rosen}), so our selection of interest is logarithmic with $\theta = 1$ and $y = q$.

For positive integers $i$, let $Z_i$ be independent binomial random variables with $m_i$ trials and success probability $p_i = y^{-i}/(1 + y^{-i})$. Set \[
\theta_i = i\E[Z_i] = im_ip_i.
\]
Given a selection picking subsets of total weight $d$, we denote by $C^{(d)}_i$ the number of selected objects of weight $i$. We denote by $C^{(d)}$ the tuple of these data and by $C^{(d)}[m_1, m_2]$ the tuple $(C^{(d)}_{m_1}, \dots, C^{(d)}_{m_2})$. Similarly, we denote by $Z[m_1, m_2]$ the tuple $(Z_{m_1}, \dots, Z_{m_2})$. Logarithmic selections satisfy the property that counts of small-weight objects are asymptotically independent:

\begin{theorem}[See {\cite[Theorem 3.3]{arratiabarbourtavare}}]\label{thm:small-components}
For a logarithmic selection satisfying \[
\left|\frac{\theta_i}{\theta} - 1\right| = O(i^{-g_1}) \qquad\text{ and }\qquad |\theta_i - \theta_{i+1}| = O(i^{-g_2}) 
\]
for some $g_1 > 0$, $g_2 > 1$, we have \[
d_{TV}(C^{(d)}[1, b], Z[1, b]) = O\left(\frac{b}{d}\right)
\]
for any positive integer $b$.
\end{theorem}
Here the implied constant can depend on all the available data, i.e., the full sequence of counts $m_i$, but not on $b$ or $d$.

We also have Poisson approximation for the total number of selected objects:

\begin{theorem}[See {\cite[Theorem 8.15]{arratiabarbourtavare}}]\label{thm:num-components}
For a logarithmic selection satisfying \[
\left|\frac{\theta_i}{\theta} - 1\right| = O(i^{-g_1}) \qquad\text{ and }\qquad |\theta_i - \theta_{i+1}| = O(i^{-g_2}) 
\]
for some $g_1 > 0$, $g_2 > 1$, let $N$ be Poisson distributed with mean $\theta\log d$. Then \[
d_{TV}\left(\sum_{i=1}^d C^{(d)}_i, N\right) = O\left(\frac{1}{\sqrt{\log d}}\right).
\]
\end{theorem}
Here again the implied constant can depend on all the available data.

Both above theorems are true in much greater generality.

In our case of interest (picking points in $\PP^1_{\F_q}$), we have \[
\theta_i = \frac{im_iq^{-i}}{1 + q^{-i}} = 1 + O_q\left(q^{-\frac{i}{2}}\right)
\]
so that the assumptions in Theorems~\ref{thm:small-components} and \ref{thm:num-components} are easily satisfied. 

Sample a uniformly random set of points in $\PP^1_{\F_q}$ of total degree $d$ containing $r_1, \dots, r_k$ and excluding $u_1, \dots, u_\ell$ (with $\sum_{\alpha = 1}^k \deg(r_\alpha) =\colon d_{\text{inc}}$ and $\sum_{\beta = 1}^\ell \deg(u_\beta) =\colon d_{\text{exc}}$). Let $\{d_0, d_1, \dots, d_n\}$ be their degrees, ordered so that: 
\begin{itemize}
    \item $d_1 \leq \dots \leq d_n$;
    \item $v_2(d_0) \leq v_2(d_i)$ for $1 \leq i \leq n$ and, if $v_2(d_i) = v_2(d_0)$, then $d_0 \leq d_i$.
\end{itemize} 
(More precisely, we consider the selection on $\PP^1_{\F_q} - \{r_1, \dots, r_k, u_1, \dots, u_\ell\}$, weighted by degree, with total degree $d - \sum_{\alpha=1}^k\deg(r_\alpha)$. Then add back in the points $r_1, \dots, r_k$.) For notational convenience, write $\vec{r}$ and $\vec{u}$ for $(r_1, \dots, r_k)$ and $(u_1, \dots, u_\ell)$, respectively.

Let $X \to \PP^1_{\F_q}$ be a curve chosen uniformly at random among all hyperelliptic curves whose ramification points have degrees $d_0, \dots, d_n$, include $r_1, \dots, r_k$, and exclude $u_1, \dots, u_\ell$. Using Theorem~\ref{thm:subfamily-equidist} and Remark~\ref{rmk:A1}, we have an estimate for the distribution of the $4$-rank of $\Pic^0(X)(\F_q)$ if there are some $\gamma < 1$ and $0 < \varepsilon < 1$ such that: \begin{enumerate}[label=$(A_\arabic*)$:]
    \item $n \geq (1 - \gamma)\log d$;
    \item the number of $d_i$ ($0 \leq i \leq n$) which are odd and the number of $d_i$ ($0 \leq i \leq n$) which are even are both at least $\varepsilon n$;
    \item there is at least one $1 \leq i < \sqrt{\log n}$ with $d_i$ odd---in particular, $d_0$ is odd;
    \item for $i \geq \sqrt{\log n}$, we have $d_i \geq 4\binom{i}{2}$.
\end{enumerate}

We will do this by directly applying Theorems~\ref{thm:small-components} and \ref{thm:num-components} in a few different ways. 

As a warm-up, we prove some easier results about the total number of branch points and about the distribution of the largest degree $d_n$:

\begin{corollary}\label{cor:num-points}
With notation as above, for $0 < \gamma < 1$ we have \[
\PP[|n - \log d| \geq \gamma\log d] = O_{\gamma, q, \vec{r}, \vec{u}}\left(\frac{1}{\sqrt{\log d}}\right).
\]
\end{corollary}
\begin{proof}
Let $N$ be Poisson distributed with mean $\log(d - d_{\text{inc}})$ (and therefore variance $\log(d - d_{\text{inc}})$). Theorem~\ref{thm:num-components} says that the number of points in our selection behaves like $N + k$, i.e., $n$ behaves like $N + k - 1$. Let $N'$ be Poisson distributed with mean $\log d - k + 1$. By \cite[Theorem 2.1]{yannaros_poisson_1991}, which gives a bound on total variation distance between any two Poisson distributions, we have \[
d_{TV}(N, N') \leq \frac{|\log d - \log(d - d_{\text{inc}}) - k + 1|}{\sqrt{\log d - k + 1} + \sqrt{\log(d - d_{\text{inc}})}} = O_{k, d_{\text{inc}}}\left(\frac{1}{\sqrt{\log d}}\right).
\]
By Chebyshev's inequality, for $V > 0$ we have \[
\PP[|N' - (\log d - k + 1)| \geq V] \leq \frac{\log d - k + 1}{V^2}
\]
so that, by Theorem~\ref{thm:num-components}, we have \begin{align*}
\PP[|n - \log d| \geq V] &= \PP[|N + k - 1 - \log d| \geq V] + O_{q, \vec{r}, \vec{u}}\left(\frac{1}{\sqrt{\log d}}\right) \\
&= \PP[|N' - (\log d - k + 1)| \geq V] + O_{q, \vec{r}, \vec{u}}\left(\frac{1}{\sqrt{\log d}}\right) \\
&\leq \frac{\log d - k + 1}{V^2} + O_{q, \vec{r}, \vec{u}}\left(\frac{1}{\sqrt{\log d}}\right).
\end{align*}
Then the result follows by setting $V = \gamma\log d$.
\end{proof}

\begin{corollary}\label{cor:assumptions}
Fix $0 < \gamma < 1$ and $0 < \varepsilon < \frac{1}{4(1 + \gamma)}$. With notation as above, we have \[
\PP[d_0, \dots, d_n\text{ fail to satisfy } (A_1)\text{--}(A_4)] = O_{\gamma, \varepsilon, q, \vec{r}, \vec{u}}\left(\frac{1}{\log\log\log d}\right).
\] 
\end{corollary}
\begin{proof}
It suffices to bound the probability that $d_0, \dots, d_n$ satisfy $(A_1), (A_2), (A_3)$ but not $(A_4)$, as well as the probability that they satisfy $(A_1)$, $(A_2)$, $(A_4)$ but not $(A_3)$, and so on. The idea is to modify and rephrase the conditions $(A_i)$ to conditions $(A_i')$ in such a way that all the $(A_i')$ together imply all the $(A_i)$. Then the probability of $d_0, \dots, d_n$ failing $(A_1)$--$(A_4)$ is bounded by $\sum_{i=1}^4 \PP[d_0, \dots, d_n \text{ fail to satisfy }(A_i')]$.
\begin{enumerate}[label=$(A_{\arabic*}):$]
    \item Set $(A_1')$ to be the condition that $(1 - \gamma)\log d \leq n \leq (1 + \gamma)\log d$. Then $(A_1')$ implies $(A_1)$.
    
    The fact that $\PP[d_0, \dots, d_n \text{ fail to satisfy }(A_1')] = O_{\gamma, q, \vec{r}, \vec{u}}\left(\frac{1}{\sqrt{\log d}}\right)$ follows from Corollary~\ref{cor:num-points}.
    
    \item We will instead bound the probability of the condition $(A_2')$: the number of $d_i$ which are odd and the number of $d_i$ which are even are both more than $\varepsilon(1 + \gamma)\log d$. If $(A_1')$ is satisfied, then $n \leq (1 + \gamma)\log d$, in which case $(A_2')$ implies $(A_2)$. However, $(A_2')$ is easier to work with because we do not need to compare to the random variable $n$.

    For fixed $b \ll d$, the expected number of odd $d_i \leq b$ after approximating by the $Z_i$ is, in the worst case that each of $r_1, \dots, r_k$ has even degree and each of $u_1, \dots, u_\ell$ has odd degree, \[
    \sum_{i=0}^{\lceil b/2\rceil - 1} \E[Z_{2i + 1}] = \sum_{i=0}^{\lceil b/2\rceil - 1} \frac{q^{-2i-1}m_{2i+1}}{1 + q^{-2i-1}} = \sum_{i=0}^{\lceil b/2\rceil - 1} \frac{q^{-2i-1}m_{2i+1}}{1 + q^{-2i-1}}.
    \]
    For $i$ large enough (certainly, if $i > 1$ is also larger than the largest degree of any excluded points), we have the estimate from \eqref{eq:mi-estimate}. In the worst case that each of $u_1, \dots, u_\ell$ has odd degree, some of the $m_i$ may be smaller than the number of irreducible monic polynomials of degree $i$, and the total defect contributed by this consideration is less than $\ell$ (the case where all of $u_1, \dots, u_\ell$ have degree 1). Thus, \[
    \sum_{i=0}^{\lceil b/2\rceil - 1} \E[Z_{2i + 1}] \geq \frac{1}{2}\left(\sum_{i=0}^{\lceil b/2\rceil - 1} \frac{1}{2i + 1} - \sum_{i=0}^{\lceil b/2\rceil - 1} \frac{q^{-i}}{2i + 1}- \sum_{i=0}^{\lceil b/2\rceil - 1} \frac{q^{-2i/3}}{2i + 1} - \ell\right)
    \]
    using \eqref{eq:mi-estimate}. The last three terms on the right hand side are uniformly bounded, and the first term is bounded below by \[
    \sum_{i=0}^{\lceil b/2\rceil - 1} \frac{1}{2i + 1} \geq \sum_{i=1}^{\lceil b/2\rceil}\frac{1}{2i} \geq \int_{1}^{b/2} \frac{du}{2u} = \frac{1}{2}\log\left(\frac{b}{2}\right)
    \]
    so that \begin{equation}\label{eq:ave-odd}
    \E\left[\sum_{i=0}^{\lceil b/2\rceil - 1} Z_{2i + 1}\right] \geq \frac{1}{4}\log b - O_{q, \ell}(1). 
    \end{equation}
    Similarly, the expected number of even $d_i \leq b$ after approximating by the $Z_i$ is \[
    \E\left[\sum_{i=1}^{\lfloor b/2\rfloor} Z_{2i}\right] \geq \frac{1}{4}\log b - O_{q, \ell}(1).
    \]
    Since the $Z_i$ are independent and \begin{equation}\label{eq:var-bound}
    \Var(Z_i) = m_ip_i(1 - p_i) \leq \frac{1}{i}\cdot\frac{1}{(1 + q^{-i})^2} + O_q(q^{-i/2}) \leq \frac{1}{i} + O_q(q^{-i/2})
    \end{equation}
    \details{Note that $m_i$ is bounded above by the number of irreducible squarefree polynomials of degree $i$ (possibly $+1$ if $i = 1$), so we can use \eqref{eq:mi-estimate} for an upper bound even if $m_i$ is smaller due to excluded points.}
    we have \begin{equation}\label{eq:var-odd}
    \Var\left(\sum_{i=0}^{\lceil b/2\rceil - 1} Z_{2i + 1}\right) \leq \sum_{i=0}^{\lceil b/2\rceil - 1} \frac{1}{2i + 1} + O_q(1) \leq \frac{1}{2}\log b + O_q(1)
    \end{equation}
    and \[
    \Var\left(\sum_{i=1}^{\lfloor b/2\rfloor} Z_{2i}\right) \leq \sum_{i=0}^{\lceil b/2\rceil - 1} \frac{1}{2i + 1} + O_q(1) \leq \frac{1}{2}\log b + O_q(1).
    \]
    By Chebyshev's inequality, we have \[
    \PP\left[\sum_{i=0}^{\lceil b/2\rceil - 1} Z_{2i + 1} \leq \varepsilon(1 + \gamma)\log d \right] \leq \frac{\log b + O_q(1)}{2\left(\varepsilon(1 + \gamma)\log d - \frac{1}{4}\log b + O_{q, \ell}(1)\right)^2}
    \]
    Choose $4\varepsilon(1 + \gamma) < \delta < 1$ (here we are using the constraint on $\varepsilon$) and set $b = \lfloor d^\delta\rfloor$ so that $\log b \leq \delta\log d$ and we have \[
    \PP\left[\sum_{i=0}^{\lceil b/2\rceil - 1} Z_{2i + 1} \leq \varepsilon(1 + \gamma)\log d \right] \leq \frac{\delta\log d + O_q(1)}{2\left(\log d + O_{q, \ell, \gamma, \varepsilon, \delta}(1)\right)^2}\cdot\frac{1}{\left(\varepsilon(1 + \gamma) - \frac{\delta}{4}\right)^2} = O_{q, \ell, \gamma, \varepsilon, \delta}\left(\frac{1}{\log d}\right).
    \]
    We get the same bound for the sum of the even terms. Finally, by Theorem~\ref{thm:small-components}, we find that the same estimates hold for the number of even and odd $d_i$ (excluding those coming from $r_1, \dots, r_k$) up to an error on the order of $\frac{b}{d} \leq d^{\delta - 1}$, which is absorbed by the error term. Hence, \[
    \PP[d_0, \dots, d_n \text{ fail to satisfy } (A_2')] = O_{\gamma, \varepsilon, q, \vec{r}, \vec{u}}\left(\frac{1}{\log d}\right).
    \]

    \item We will observe that if $(A_1')$ and $(A_4)$ are satisfied, then for each $1 \leq i \leq n$, \[
    \text{if }\qquad d_i \leq  4\binom{\lfloor \sqrt{\log \log d + \log(1 - \gamma)}\rfloor}{2} \leq 4\binom{\lfloor \sqrt{\log n}\rfloor}{2} \qquad\text{ then }\qquad i < \sqrt{\log n}.
    \]
    Set \[
    D \coloneqq 4\binom{\lfloor \sqrt{\log \log d + \log(1 - \gamma)}\rfloor}{2}.
    \]
    To find an odd degree $d_i$ with $1 \leq i < \sqrt{\log n}$, it suffices to find an odd degree $1 \leq d_i \leq D$. Thus, we set $(A_3')$ to be the condition that there is an odd $1 \leq d_i \leq D$ and note that $(A_1')$, $(A_3')$, and $(A_4')$ together imply $(A_3)$.

    By Theorem~\ref{thm:small-components}, the number of odd degrees bounded by $D$ is approximated by $\sum_{i=0}^{\lceil D/2 \rceil - 1} Z_{2i + 1}$ (plus however many of $r_1, \dots, r_k$ have odd degree). We have estimates for the mean and variance of this number from \eqref{eq:ave-odd} and \eqref{eq:var-odd}. Using these, Chebyshev's inequality gives \[
    \PP\left[\sum_{i=0}^{\lceil D/2 \rceil - 1} Z_{2i + 1} \leq 1\right] \leq \frac{\log(D) + O_q(1)}{2\left(\frac{1}{4}\log D + O_{q, \ell}(1) - 1\right)^2} = O_{q, \ell}\left(\frac{1}{\log D}\right) = O_{\gamma, q, \ell}\left(\frac{1}{\log\log\log d}\right).
    \]
    Thus, using Theorem~\ref{thm:small-components}, we have \[
    \PP[d_0, \dots, d_n \text{ fail to satisfy } (A_3')] = O_{\gamma, q, \vec{r}, \vec{u}}\left(\frac{1}{\log\log\log d}\right).
    \]
    \item Let $(A_4')$ be the condition that $d_i \geq 4\binom{i}{2}$ for $\sqrt{\log \log d + \log(1 - \gamma)} \leq i \leq (1 + \gamma)\log d$. Note that $(A_1')$ and $(A_4')$ together imply $(A_4)$.

    We observe that, since the $d_i$ are nondecreasing for $i \geq 1$, we can reformulate the condition $(A_4')$ as asking that for each $\sqrt{\log \log d + \log(1 - \gamma)} \leq i \leq (1 + \gamma)\log d$, the number of degrees $d_j$ strictly smaller than $4\binom{i}{2}$ is less than $i$. Recall that we have a number of deterministic branch degrees coming from $\deg(r_1), \dots, \deg(r_k)$, so $(A_4')$ is satisfied if for each $i$ in that range, the number of random selected points of degree smaller than $4\binom{i}{2}$ is less than $i - k$.

    We have \[
    \E\left[\sum_{j=1}^{4\binom{i}{2}}Z_j\right] = \sum_{j=1}^{4\binom{i}{2}}\E[Z_j] = \sum_{j=1}^{4\binom{i}{2}} \frac{q^{-j}m_j}{1 + q^{-j}}.
    \] 
    Using the estimate from \eqref{eq:mi-estimate} for large enough $j$ (depending on $r_1, \dots, r_k$ and $u_1, \dots, u_\ell$), we have \[
    \E\left[\sum_{j=1}^{4\binom{i}{2}}Z_j\right] = \sum_{j=1}^{4\binom{i}{2}} \frac{q^{-j}m_j}{1 + q^{-j}} \leq \sum_{j=1}^{4\binom{i}{2}} \frac{1}{j} + O_{q, \ell}(1) = 2\log i + O_{q, \ell}(1)
    \]
    and, using \eqref{eq:var-bound}, we have \[
    \Var\left(\sum_{j=1}^{4\binom{i}{2}}Z_j\right) = \sum_{j=1}^{4\binom{i}{2}}\Var(Z_j) \leq \sum_{j=1}^{4\binom{i}{2}} \frac{1}{j} + O_q(1) = 2\log i + O_q(1).
    \]
    Then by Chebyshev's inequality, we have \[
    \PP\left[\sum_{j=1}^{4\binom{i}{2}}Z_j \geq i - k\right] \leq \frac{2\log i + O_q(1)}{\left(i - k - 2\log i + O_{q, \ell}(1)\right)^2}.
    \]
    We will only need to use this inequality when $i \geq \sqrt{\log \log d + \log(1 - \gamma)}$. In that case, for $d$ large enough (depending on $q$, $k$, $\ell$, and $\gamma$), the numerator is bounded above by $3\log i$ and the denominator is bounded below by $(i/2)^2$ so that \[
    \PP\left[\sum_{j=1}^{4\binom{i}{2}}Z_j \geq i - k\right] = O_{\gamma, q, k, \ell}\left(\frac{\log i}{i^2}\right).
    \]
    Then, by the union bound, \begin{align*}
    \PP&\left[\sum_{j=1}^{4\binom{i}{2}}Z_j \geq i - k \text{ for some }\sqrt{\log \log d + \log(1 - \gamma)} \leq i \leq (1 + \gamma)\log d\right] \\
    &\qquad\qquad\qquad\qquad\qquad\qquad\qquad\qquad= O_{\gamma, q, \ell}\left(\sum_{\sqrt{\log \log d + \log(1 - \gamma)} \leq i \leq (1 + \gamma)\log d} \frac{\log i}{i^2}\right) \\
    &\qquad\qquad\qquad\qquad\qquad\qquad\qquad\qquad= O_{\gamma, q, k, \ell}\left(\sum_{\sqrt{\log \log d + \log(1 - \gamma)} \leq i \leq (1 + \gamma)\log d} \frac{1}{i^{3/2}}\right).
    \end{align*}
    We have \begin{align*}
        \sum_{\sqrt{\log \log d + \log(1 - \gamma)} \leq i \leq (1 + \gamma)\log d} \frac{1}{i^{3/2}} &\leq \int_{\sqrt{\log \log d + \log(1 - \gamma)}}^{(1 + \gamma)\log d + 1} \frac{du}{u^{3/2}} \\
        &\leq  \int_{\sqrt{\log \log d + \log(1 - \gamma)}}^{\infty} \frac{du}{u^{3/2}} \\
        &= \frac{2}{(\log \log d + \log(1 - \gamma))^{1/4}}
    \end{align*}
    which means \[
    \PP\left[\sum_{j=1}^{4\binom{i}{2}}Z_j \geq i - k \text{ for some }\sqrt{\log \log d + \log(1 - \gamma)} \leq i \leq (1 + \gamma)\log d\right] = O_{\gamma, q, k, \ell}\left(\frac{1}{(\log\log d)^{1/4}}\right)
    \]
    and, using Theorem~\ref{thm:small-components} with $b = 4\binom{\lfloor(1 + \gamma)\log d\rfloor}{2}$, we get \[
    \PP[d_0, \dots, d_n \text{ fail to satisfy } (A_4')] = O_{\gamma, \varepsilon, q, \vec{r}, \vec{u}}\left(\frac{1}{(\log\log d)^{1/4}}\right).
    \] 
\end{enumerate}
\end{proof}

\subsection{Main results}\label{sect:main-results}

We now have all the ingredients we need to state and prove our main results. We start by setting up a procedure for drawing hyperelliptic curves from the distributions of interest.

Fix an odd prime power $q$. Let $S, S', S'' \subset \PP^1_{\F_q}$ be disjoint finite sets of points. Write $\deg(S) \coloneqq \sum_{p \in S} \deg(p)$, and similarly for $S', S''$. Let $g \geq 1$ be an integer. We are going to construct a random hyperelliptic curve $X$ of total branch degree $d \coloneqq 2g + 2$, ramified at $S$ and unramified at $S' \cup S''$. Assume $\sqrt{\log\log d - \log 2} > \max_{p \in S} \deg(p)$, and also assume that there exists a set $S_0 \subset \PP^1_{\F_q} - (S \cup S' \cup S'')$ such that $\deg(S_0) + \deg(S) = d$.

Draw points of total degree $d - \deg(S)$ from the selection consisting of points in $\PP^1_{\F_q} - (S \cup S' \cup S'')$ weighted by degree as described in Subsection~\ref{sect:branch-degrees}. Let $\vec{d}$ be the random multiset of degrees of points obtained this way together with the degrees of the points of $S$. Note that the elements $d_0, d_1, \dots, d_n$ of $\vec{d}$ may be uniquely ordered as follows (if any degree $d_i$ is odd, this agrees with the ordering in Subsection~\ref{sect:param-curves}):
\begin{itemize}
    \item $d_1 \leq \dots \leq d_n$;
    \item $v_2(d_0) \leq v_2(d_i)$ for $1 \leq i \leq n$;
    \item if $v_2(d_0) = v_2(d_i)$ for some $1 \leq i \leq n$, then $d_0 \leq d_i$.
\end{itemize}
We now split into two cases depending on whether $d_0$ is odd.

If $d_0$ is even (so all points drawn in the above step have even degree), let $X$ be a hyperelliptic curve drawn uniformly at random among the two hyperelliptic curves branched at all points selected in the above step and all points in $S$.

If $d_0$ is odd, we will resample the branch points after having chosen the degree of each branch point, as follows:

For each $0 \leq i \leq n$ we define a subset $T_i$ of degree $d_i$ points in $\PP^1_{\F_q}$ as follows: first, choose an order for the points in $S$. Then for each point $p \in S$ in this order, let $i$ be minimal such that $d_i = \deg(p)$ and $T_i$ has not been defined yet. We set $T_i = \{p\}$. Once this step has been done for each point $p \in S$, for each remaining $0 \leq i \leq n$ with $T_i$ undefined we set $T_i = \{p \in \PP^1_{\F_q} - (S' \cup S'') \mid \deg(p) = d_i\}$ to be the set consisting of all points of $\PP^1_{\F_q}$ of degree $d_i$.

For each point $p \in \PP^1_{\F_q}$, let $t(p)$ be a random uniformizer class at $p$ chosen from any distribution. Let $T$ be the set of ordered tuples $(p_0, \dots, p_n, t(p_0))$ such that $p_i \in T_i$ for $0 \leq i \leq n$. Let $T^* \subset T$ be the set of tuples $(p_0, \dots, p_n, t) \in T$ such that $p_0, \dots, p_n$ are distinct. Note that since the $t(p)$ are random, $T$ and $T^*$ are random sets.

Let $B$ be chosen uniformly at random from $T^*$, and let $X \coloneqq X_B$ be the hyperelliptic curve associated to $B$ via the parametrization described in Section~\ref{sect:param-curves}.

Then note that even though $B$ is a random \textit{ordered} tuple of points (with uniformizer class), the set of branch points of $X_B$ is drawn uniformly at random among all subsets of $\PP^1_{\F_q}$ of total degree $d$ which include $S$ and exclude $S' \cup S''$. 

\begin{example}\label{ex:XB-dist}
\begin{itemize}
\item If $t(p)$ is uniformly random for $p \in \PP^1_{\F_q}$, then $X$ is uniformly random among all hyperelliptic curves of genus $g$ which are ramified at $S$ and unramified at $S' \cup S''$.
\item If $T_0 = \{p_0\}$ is a singleton with $d_0$ odd and $t\coloneqq t(p_0)$ is fixed, then $X_B$ is uniformly random among all hyperelliptic curves of genus $g$ which are ramified at $S$ and unramified at $S' \cup S''$, and which have $c_t(X_B \to \PP^1, p_0) = 1$. 
\end{itemize}
\end{example}

\begin{proposition}\label{prop:local-conditions-prob}
We have \[
|\PP[X \text{ is split at }S'\text{ and inert at }S''] - 2^{-\#(S' \cup S'')}| = O_{q, S, S', S''}\left(\frac{1}{\log\log \log g}\right).
\]
\end{proposition}
\begin{proof}
Let $\ell = \#(S' \cup S'')$ and $d_{\mathrm{exc}} = \deg(S') + \deg(S'')$.

By Corollary~\ref{cor:assumptions}, we accrue an $O_{q,S, S', S''}\left(\frac{1}{\log\log \log g}\right)$ error term by conditioning on the assumptions $(A_1)$--$(A_4)$ being satisfied. In particular we are conditioning on $\deg(p_0)$ being odd, so that $X = X_B$ with $B = (p_0, \dots, p_n, t(p_0))$.

Condition on $\vec{d}$, $p_0, \dots, p_{n-1}$, and $t(p_0)$ satisfying $(A_1)$--$(A_4)$. In particular, for $g$ large enough, $(A_4)$ ensures that $d_n > \max_{p \in S}\deg(p)$, and so $T_n$ is not a singleton. The conditional distribution of $p_n$ is uniform over all points of degree $d_n$ excluding $p_0, \dots, p_{n-1}$ and the points in $S', S''$.  

Let $E$ be the event that $X_B$ is split at $S'$ and inert at $S''$. This is equivalent to a condition on $c(X_B \to \PP^1, p)$ for each $p \in S' \cup S''$. We observe that, if $d_n \geq 4\binom{\frac{1}{2}\log d}{2}$, then for large enough $n$ (depending on $q$, $\ell$, and $\deg(S') + \deg(S'')$) we have \[
d_n \geq 4\log_q(2^{\ell + 1}(d_{exc} + 2d_n))
\]
so that if $d_n \geq 4\binom{\frac{1}{2}\log d}{2}$ then \[
\left|\PP\left[E\ \middle|\ \vec{d}, p_0, \dots, p_{n-1}, t(p_0)\right] - 2^{-\ell}\right| \leq 2q^{-\binom{\frac{1}{2}\log d}{2}}.
\]
Therefore, \[
\left|\PP\left[E\ \middle|\ d_n \geq 4\binom{\frac{1}{2}\log d}{2}\right] - 2^{-\ell}\right| \leq 2q^{-\binom{\frac{1}{2}\log d}{2}}.
\]
Now note that $d_n \geq 4\binom{\frac{1}{2}\log d}{2}$ is implied by conditions $(A_1)$ and $(A_4)$ with $\gamma = 1/2$. Therefore, by Corollary~\ref{cor:assumptions}, we have \[
\PP\left[d_n < 4\binom{\frac{1}{2}\log d}{2}\right] = O_{q, S, S', S''}\left(\frac{1}{
\log\log\log d}\right).
\]
The conclusion follows.
\end{proof}

\begin{theorem}\label{thm:mainthm}
Fix an odd prime power $q$. Let $S, S', S'' \subset \PP^1_{\F_q}$ be disjoint finite sets of closed points.

Let $X$ be a hyperelliptic curve of genus $g \geq 2\deg(S)$ drawn uniformly at random among all hyperelliptic curves of genus $g$ which are ramified at $S$, split at $S'$, and inert at $S''$. 

Let $\mu \coloneqq \mu_{CL,2}$ if $q \equiv 3\pmod{4}$ and $\mu \coloneqq \mu_{S, 2}$ if $q \equiv 1\pmod{4}$. Let $r \geq 0$ be an integer. Then \[
|\PP[\dim_{\F_2}2\Pic^0(X)(\F_q)[4] = r] - \mu(r)| = O_{q, r, S, S', S''}\left(\frac{1}{\log\log\log g}\right).
\]
Moreover, suppose $S$ contains a point of odd degree and fix $p_0 \in S$ of minimal odd degree and a uniformizer class $t$ at $p_0$. If $S' \cup S''$ contains all closed points of $\PP^1_{\F_q}$ of odd degree strictly less than $\deg(p_0)$, then the above result also holds if $X$ is drawn uniformly at random among all hyperelliptic curves of genus $g$ which are ramified at $S$, split at $S'$, inert at $S''$, and satisfy $c_t(X \to \PP^1, p_0) = 1$.
\end{theorem}
Theorem~\ref{thm:intro-thm-monic} follows by picking a ring of integers $\F_q[x] \subset \F_q(x) = k(\PP^1_{\F_q})$ and associated point at infinity $\infty \in \PP^1_{\F_q}$. With notation as in Theorem~\ref{thm:intro-thm-monic}, split into two cases. If $d$ is even, we use the first part of Theorem~\ref{thm:mainthm} with $S, S'' = \varnothing$ and $S' = \{\infty\}$. If $d$ is odd, we use the second part of Theorem~\ref{thm:mainthm} with $p_0 = \infty \in S$ with $t = 1/x$.

\begin{proof}
By Example~\ref{ex:XB-dist}, the conclusion in both cases follows after proving the statement for the random hyperelliptic curves $X$ defined at the start of this subsection, conditional on $X$ being split at $S'$ and inert at $S''$.

Condition on $\vec{d}$, assuming $d_0$ is odd. Let $E'$ be the event that $\dim_{\F_2}2\Pic^0(X)(\F_q)[4] = r$ and that $X$ is split at $S'$ and inert at $S''$.

By Theorem~\ref{thm:subfamily-equidist} and Remark~\ref{rmk:A1}, if $\vec{d}$ satisfies $(A_1)$, $(A_2)$, $(A_3)$, $(A_4)$ with, e.g., $\varepsilon = 1/7$ and $\gamma = 1/2$, then \[
\left|\PP[E' \mid \vec{d}] - 2^{-\#(S' \cup S'')}\mu(r)\right| = O_{q, r, S', S''}\left((\log g)^{-1/4}\right).
\]
Therefore, by Corollary~\ref{cor:assumptions}, we have \[
\left|\PP[E'] - 2^{-\#(S' \cup S'')}\mu(r)\right| = O_{q, r, S, S', S''}\left(\frac{1}{\log\log\log g}\right).
\]
Now let $E$ be the event that $X$ is split at $S'$ and inert at $S''$. The conditional probability we wish to estimate is $\PP[E']/\PP[E]$. By Proposition~\ref{prop:local-conditions-prob} we have \[
\left|2^{\#(S' \cup S'')}\PP[E] - 1\right| = O_{q, S, S', S''}\left(\frac{1}{\log\log\log g}\right).
\]
When the left hand side is at most $1/2$, we also have \[
\left|\frac{1}{2^{\#(S' \cup S'')}\PP[E]} - 1\right| \leq 2\left|2^{\#(S' \cup S'')}\PP[E] - 1\right| = O_{q, r, S, S', S''}\left(\frac{1}{\log\log\log g}\right).
\]
Then by Lemma~\ref{lem:err-combining} we have \[
\left|\frac{\PP[E']}{\mu(r)\PP[E]} - 1\right| = O_{q, r, S, S', S''}\left(\frac{1}{\log\log\log g}\right).
\]
and the conclusion follows.
\end{proof}

\subsection*{Acknowledgements}

The first author was supported by an NSF Graduate Research Fellowship. The authors thank Carlo Pagano, Alex Smith, and Melanie Matchett Wood for valuable conversations. The authors also thank \'Etienne Fouvry, Hunter Handley, Melanie Matchett Wood, Peter Koymans, and especially Isaac Rajagopal for helpful comments on earlier drafts of this manuscript. This work was partially funded by NSF grant DMS-2140043.

\printbibliography
\end{document}

%% file: settings.tex
\usepackage[utf8]{inputenc}
\usepackage{amsmath}
\usepackage{amssymb}
\usepackage{amsthm}
\usepackage{bbm}
\usepackage{fancyhdr}
\usepackage{enumitem}
\usepackage{parskip}
\usepackage{mathrsfs}
\usepackage{mathtools}
\usepackage{xcolor}
\usepackage[lmargin=0.5in, rmargin=0.5in]{geometry}
\usepackage{hyperref}
\hypersetup{
  colorlinks=true,
  linkcolor=blue,
  citecolor=blue
}
\usepackage{tikz-cd}
\usepackage{chngcntr}
\usepackage{extpfeil}

\usepackage[
backend=biber,
style=alphabetic,
]{biblatex}

\makeatletter
\let\@old@begintheorem=\@begintheorem
\def\@begintheorem#1#2[#3]{%
  \gdef\@thm@name{#3}%
  \@old@begintheorem{#1}{#2}[#3]%
}
\def\namedthmlabel#1{\begingroup
   \edef\@currentlabel{\@thm@name}%
   \label{#1}\endgroup
}
\makeatother

\theoremstyle{definition}

\makeatletter
\renewcommand*\env@matrix[1][*\c@MaxMatrixCols c]{%
  \hskip -\arraycolsep
  \let\@ifnextchar\new@ifnextchar
  \array{#1}}
\makeatother

\theoremstyle{theorem}
\newtheorem{theorem}{Theorem}
\newtheorem*{theorem*}{Theorem}
\newtheorem{lemma}[theorem]{Lemma}
\newtheorem{proposition}[theorem]{Proposition}

\newtheorem{corollary}[theorem]{Corollary}
\newtheorem{example}[theorem]{Example}

\theoremstyle{definition}
\newtheorem{definition}[theorem]{Definition}
\newtheorem{remark}[theorem]{Remark}
\newtheorem*{definition*}{Definition}
\newtheorem*{lemma*}{Lemma}

\numberwithin{equation}{section}
\numberwithin{theorem}{section}

\newcommand{\R}{\mathbb R}
\newcommand{\Z}{\mathbb Z}
\newcommand{\Q}{\mathbb Q}

\newcommand{\A}{\mathbb A}
\newcommand{\F}{\mathbb F}

\newcommand{\PP}{\mathbb P}

\newcommand{\GL}{\operatorname{GL}}

\newcommand{\Hom}{\operatorname{Hom}}

\newcommand{\Aut}{\operatorname{Aut}}

\newcommand{\Spec}{\operatorname{Spec}} 

\renewcommand{\ker}{\operatorname{ker}}

\newcommand{\Frac}{\operatorname{Frac}} 

\newcommand{\ip}[2]{\left\langle#1,#2\right\rangle} 
\newcommand{\Gal}{\operatorname{Gal}} 

\makeatletter
\def\namedlabel#1#2{\begingroup
   \def\@currentlabel{#2}%
   \label{#1}\endgroup
}
\makeatother

\newcommand{\E}{\mathbb E}
\newcommand{\Var}{\operatorname{Var}}

\newcommand{\vspan}{\operatorname{span}}

\newcommand{\rank}{\operatorname{rank}}

\renewcommand{\div}{\operatorname{div}}
\newcommand{\Div}{\operatorname{Div}}
\newcommand{\Prin}{\operatorname{Prin}}
\newcommand{\Cl}{\operatorname{Cl}}

\newcommand{\Pic}{\operatorname{Pic}}

\newcommand{\Frob}{\operatorname{Frob}}
\newcommand{\sheafO}{\mathcal{O}}
\newcommand{\ab}{{\text{ab}}}

\newcommand{\etale}{{\operatorname{\acute{e}t}}}

\newcommand{\twist}{{\operatorname{twist}}}
\newcommand{\legendre}[2]{\left(\dfrac{#1}{#2}\right)}
\newcommand{\Br}{\operatorname{Br}}
\newcommand{\G}{\mathbb{G}}